\documentclass{amsart}
\usepackage{amssymb,amsmath}
\usepackage[usenames,dvipsnames]{xcolor}
\usepackage{tikz-cd}
\usepackage{fullpage}
\usepackage{hyperref}
\usepackage[shortlabels]{enumitem}
\usepackage{cleveref}
\usepackage{bm} 

\newcommand{\nc}{\newcommand}

\newcommand{\B}[1]{\overline{#1}}
\newcommand{\dd}{\textrm{d}}
\newcommand{\m}[2]{\left\langle #1,#2 \right\rangle}
\newcommand{\f}[1]{\mathfrak{#1}}
\newcommand{\ad}{\operatorname{ad}}
\newcommand{\R}{\mathbb{R}}

\newcommand{\id}{\operatorname{Id}}

\newcommand{\comment}[1]{}
\newcommand{\mc}[1]{\mathcal{#1}}
\nc{\Or}{\mathsf{O}}

\nc{\mud}{{\bm{\mu}}} 
\nc{\nud}{{\bm{\nu}}}
\nc{\ip}{{\langle \,\cdot \,,\cdot \,\rangle }} 
\nc{\la}{\langle} \nc{\ra}{\rangle}
\nc{\lla}{\langle \! \langle}
\nc{\rra}{\rangle \! \rangle}
\nc{\iip}{{\lla \cdot \, , \, \cdot \rra}}
\nc{\cca}{\mathcal{C}} 
\nc{\gca}{\mathcal{G}}
\nc{\mca}{\mc M}
\nc{\TT}{{T \oplus T^*}}
\nc{\TTH}{{(T \oplus T^*)_H}}
\nc{\lb}{[\cdot, \cdot]}
\nc{\G}{\mathsf{G}}
\nc{\bG}{\overline{\mc G}}
\nc{\infECA}{V}
\nc{\ggo}{\mathfrak{g}}
\nc{\lgo}{\mathfrak{l}}
\nc{\gggo}{\mathfrak{g} \oplus \mathfrak{g}^*}
\nc{\tr}{\operatorname{tr}}
\nc{\Lie}{\operatorname{Lie}}
\nc{\Ll}{\mathsf{L}}
\nc{\End}{\operatorname{End}}
\nc{\Ad}{\operatorname{Ad}}
\nc{\mm}{\operatorname{M}}
\nc{\Ricb}{\operatorname{Ric}^B}
\nc{\scal}{\operatorname{R}}

\theoremstyle{definition}
\newtheorem{thm}{Theorem}[section]
\newtheorem{defn}[thm]{Definition}
\newtheorem{prop}[thm]{Proposition}

\newtheorem{lem}[thm]{Lemma}
\newtheorem{cor}[thm]{Corollary}
\newtheorem{quest}[thm]{Question}
\newtheorem{example}[thm]{Example}
\newtheorem{remark}[thm]{Remark}

\newtheorem{thmintro}{Theorem}

\newtheorem{corintro}[thmintro]{Corollary}

\title{The homogeneous generalized Ricci flow}
 \author{Elia Fusi}
 \author{Ramiro Lafuente}
 \author{James Stanfield}

\begin{document}

 \begin{abstract}
 We develop a framework inspired by Lauret's ``bracket flow'' to study the generalized Ricci flow, as introduced by Streets, on discrete quotients of Lie groups. As a first application, we establish global existence on solvmanifolds in arbitrary dimensions, a result which is new even for the pluriclosed flow. We also define a notion of generalized Ricci soliton on exact Courant algebroids that is geometrically meaningful and allows for non-trivial expanding examples. On nilmanifolds, we show that these solitons arise as rescaled limits of the generalized Ricci flow, provided the initial metrics have ``harmonic torsion", and we classify them in low dimensions. Finally, we provide a new formula for the generalized Ricci curvature of invariant generalized metrics in terms of a moment map for the action of a non-reductive real Lie group.
 \end{abstract}
 %
\maketitle

\setcounter{tocdepth}{1} 

\tableofcontents

\section{Introduction}\
  
Motivated by non-K\"ahler Hermitian geometry and mathematical physics, the \emph{generalized Ricci flow} is an analogue of the classical Ricci flow in the setting of Hitchin's \emph{generalized geometry} \cite{Cou90, CSCW11, GF19, GRFbook, GuaAnnals, KW07,LWX97, Sletters, SV17, SV20}. It is a flow of \emph{generalized Riemannian metrics} on an exact Courant algebroid. The latter can be described, up to isomorphism, by a smooth manifold $M$ and a choice of cohomology class $[H]\in H^3(M, \R)$, whereas the former can  be thought of as a pair $(g,b)$, with $g$ a Riemannian metric on $M$, and $b \in \Omega^2(M)$ a 2-form on $M$. A smooth family $(g_t,b_t)_{t \in I}$ of generalized Riemannian metrics on $(M,[H])$ solves the \emph{generalized Ricci flow} if it satisfies the following weakly parabolic PDE:
\begin{equation}\label{eqn_GRFintro}
	\begin{cases}
		\tfrac{\partial}{\partial t} g_t &= \ -2 \operatorname{Ric}(g_t) + \frac 12 (H + \dd b_t)^2_{g_t},\\
		\tfrac{\partial}{\partial t} b_t &= \  - \dd^*_{g_t}(H + \dd b_t),
	\end{cases}
\end{equation}
where, for a $3$-form $H$ and Riemannian metric $g$, $H_g^2 \in \operatorname{Sym}^2(T^*M)$ is defined by $H^2_g(X,X) = |\iota_XH|_g^2$ for all $X \in TM$. The classical Ricci flow is the particular case where $H = 0$ and $b_t \equiv 0$.
By abuse of notation, we will sometimes call the pair $(g,H + \dd b)$ a generalized Riemannian metric on $M$, meaning that the exact Courant algebroid is $(M,[H])$ and the generalized metric is the pair $(g,b)$. Similarly, we may refer to a family $(g_t,H_t := H + \dd b_t)$ as a generalized Ricci flow on $M$.


The primary goal of this article is to give a systematic treatment of the generalized Ricci flow in the locally homogeneous case. To simplify the presentation, we focus on exact Courant algebroids over Lie groups (or their discrete quotients), endowed with left-invariant generalized metrics. That is, we consider exact Courant algebroids of the form $(\mathsf G,[H])$, where $\mathsf{G}$ is a Lie group and $[H] \in H^3(\mathsf G, \R)^{\mathsf{G}}$ is an invariant cohomology class, and we restrict ourselves to generalized metrics $(g,b)$ that are left-invariant. In this direction, our first main result is that the invariant generalized Ricci flow exhibits long-time existence on \emph{solvmanifolds}:

\begin{thmintro}\label{thmIntro_solvmfdGRF} Any invariant solution to the generalized Ricci flow on a solvmanifold exists for all positive times.
\end{thmintro}

 More generally, we characterise the maximal existence time for the invariant generalized Ricci flow on any Lie group $\mathsf G$ in terms of the blow-up behaviour of the \emph{generalized scalar curvature} (see \Cref{thm_blowup}). Here by \emph{solvmanifold} we mean any (not necessarily compact) locally homogeneous space of the form $\mathsf{G} / \Gamma$, where $\mathsf{G}$ is a solvable Lie group, and $\Gamma \leq \mathsf{G}$ is a discrete subgroup. It is called a \emph{nilmanifold} if $\mathsf{G}$ is nilpotent. An \emph{invariant solution} is simply a solution to \eqref{eqn_GRFintro} on $\mathsf{G}/\Gamma$ that lifts to a left-invariant solution on $\mathsf{G}$. The class of solvmanifolds is known to be very rich. Examples in four dimensions include  Inoue surfaces and primary Kodaira surfaces.

Notice that \Cref{thmIntro_solvmfdGRF} is new even in the case of the pluriclosed flow. Recall that when $(M,J,g)$ is Hermitian and $H$ is the torsion of the Bismut connection, \eqref{eqn_GRFintro} is gauge-equivalent to the pluriclosed flow on $(M,J)$ \cite{GRFbook}. As an immediate application of \Cref{thmIntro_solvmfdGRF} we get 

\begin{corintro} \label{cor_main_PCF}
Any invariant solution to the pluriclosed flow on a solvmanifold exists for all positive times.
\end{corintro}

This was previously known in the case of complex surfaces, for nilmanifolds, for $6$-dimensional solvmanifolds with trivial canonical bundle, for almost abelian solvmanifolds, for some almost nilpotent solvmanifolds, and for Oeljeklaus--Toma manifolds \cite{Bol2016,FV15,AL19,AN22,FP22,FV23}. Our approach gives a unified treatement of all these results, and applies to SKT solvmanifolds on which the pluriclosed flow has not yet been studied. For instance, the Endo--Pajitnov manifolds \cite{EP20} which were recently shown to admit SKT structures in \cite{COS24}. Moreover, the existence of SKT metrics was characterized  by Freibert--Swann in \cite{FS21} on a large class of 2-step solvmanifolds,   and on solvmanifolds admitting a codimension $2$ abelian ideal by Guo--Zheng \cite{GZ23} and Cao--Zheng \cite{CZ23}.

It is conjectured that the pluriclosed flow on a compact complex manifold $(M,J)$ exists for all time if and only if $-c_1^A(M,J)$ is a nef class (that is, $-c_1^A(M,J)$ lies in the closure of the cone of classes in $H_A^{1,1}(M,J)$ containing a positive representative), where $c_1^A(M,J)\in H_A^{1,1}(M,J)$ is the \emph{first Aeppli--Chern class} of $(M,J)$ \cite{S16conj, S20conj}. It is well-known that the first Chern class of solvmanifolds is zero, hence the first Aeppli--Chern class vanishes as well. In particular, $-c_1^A(\mathsf{G}/\Gamma,J)$ is a nef class on compact solvmanifolds, and thus \Cref{cor_main_PCF} verifies the conjecture in this setting. 

After having established global existence, a natural problem is to determine the long-time asymptotics of solutions. In this direction, on a nilmanifold $M = \mathsf{G}/\Gamma$ with harmonic $H_0$  (i.e.~ $\dd\dd^*H_0 = 0$), we can characterise the possible limits as  follows:

\begin{thmintro}\label{thm_main_nil} Let  $(M,g_t,H_t = H_0 + \dd b_t)_{t \in [0,\infty)}$ be an invariant solution to the generalized Ricci flow on a nilmanifold $M=\mathsf{G}/\Gamma$ with  $H_0$ harmonic. Then, after pulling back to the universal cover, the rescaled generalized metrics $(\frac{1}{1+t}g_t,\frac{1}{1+t}H_t)_{t \in [0,\infty)}$ on the time-dependent exact Courant algebroids $(\mathsf G, \frac{1}{1+t}[H_0])$ converge, in the \emph{generalized Cheeger--Gromov} topology, to a left-invariant, expanding generalized Ricci soliton $(g_\infty,H_\infty)$ on a nilpotent Lie group $\mathsf{G}_\infty$.
\end{thmintro}

By \emph{generalized Ricci soliton} we mean a solution $(g_t, H_t = H + \dd b_t)$ to \eqref{eqn_GRFintro} which, up to pull-back by time-dependent diffeomorphisms, evolves only by scaling. That is, there exist $c(t) \in \R_{>0}$, and $\varphi_t \in {\rm Diff}(M)$ such that, for all $t$,  it holds that
\[
		 \left( \varphi_t^* g_t,  \varphi_t^* H_t\right) = (c(t) \, g_0, c(t)\, H_0).
\]
Notice that the generalized metrics $(c(t) \, g_0, c(t)\, H_0)$ live on the time-dependent exact Courant algebroids $(M, c(t) [H])$. By allowing $H$ to scale, this definition differs from others found in the literature \cite[4.4.2]{GRFbook} (compare however with \cite[Remark 6.31]{GRFbook}), but agrees with the definition posed in \cite{PR23}. We adpot it here because it allows for non-trivial, non-steady solitons, and because examples arise naturally as asymptotic limits by Theorem \ref{thm_main_nil}. We refer the reader to Section \ref{sec_GG} for more technical details regarding this notion and the scaling of exact Courant algebroids, and to Section \ref{sec_LongNil} for the definition of generalized Cheeger-Gromov topology.

Notice that the limit group may be non-isomorphic to $\G$ (this occurs already in the classical setting \cite{nilRF}).  We conjecture that the assumption $H_0$ harmonic  is not necessary. This will be the subject of forthcoming work.  However, let us mention that in dimensions up to $4$, all  left-invariant generalized Ricci solitons on nilpotent Lie groups do have harmonic $H_0$. This follows from our complete classification  in \Cref{sec_classifi}.


Let us now turn to the methods used to prove our main results. Proceeding more precisely, the exact Courant algebroid over $M$ correponding to the cohomology class $[H_0] \in H^3(M, \R)$ can be identified with the bundle $(TM\oplus T^*M,[\cdot,\cdot]_{H_0})$ endowed with the \emph{twisted Dorfmann bracket} $[\cdot,\cdot]_{H_0}$, a bracket law on $TM \oplus T^*M$ determined by $H_0$ and the usual Lie bracket of vector fields $[\cdot,\cdot]$ on $TM$, see \eqref{eqn_[]H-general}. The pair $(g,b)$ then determines a generalized metric on $TM \oplus T^*M$, which is an endomorphism $\mathcal{G}(g,b) \in \operatorname{End}(TM \oplus T^*M)$. The \emph{generalized Ricci curvature} $\mathcal Rc(\mathcal G) \in \operatorname{End}(TM \oplus T^*M)$ of $\mathcal{G}$ is a $\mathcal{G}$-symmetric endomorphism on $TM \oplus T^*M$ (see \Cref{sec_genRic} for its definition) and the generalized Ricci flow is the following PDE:
\[
\frac{\partial}{\partial t}\mathcal G = -2 \, \mathcal{G}  \, \mathcal Rc(\mathcal G).
\]

If $M = \mathsf{G}$ is a Lie group and $[H_0] \in H^3(\mathsf{G}, \R)^{\mathsf G}$ is invariant, then the exact Courant algebroid structure $(T\mathsf{G}\oplus T^*\mathsf{G},[\cdot,\cdot]_{H_0})$ can be identified with the pair $(\f g\oplus \f g^*,\mud_{H_0})$, where $\f g$ is the Lie algebra of $\mathsf{G}$, and $\mud_{H_0}$ is a bracket on $\f g \oplus \f g^*$ determined by $H_0 \in \Lambda^3\f g^*$ and $\mu$, the Lie bracket of $\f g$. The invariant generalized Ricci flow is then equivalent to an ODE on the space of endomorphisms of $\f g \oplus \f g^*$. We now follow Lauret's \emph{bracket flow} approach. The non-reductive group $\mathsf{L} = \B{\mathsf{GL}(\f g)}\cdot\operatorname{exp}(\Lambda^2 \f g^*) \cong \mathsf{GL}(\f g) \ltimes \Lambda^2\f g^*$ acts transitively on the space of invariant generalized metrics. Thus, by fixing a background generalized metric $\B{\mathcal{G}} = \mathcal{G}(g,0)$, we pull-back the generalized Ricci flow solution to an ODE on the space of brackets on $\f g \oplus \f g^*$ via the action of $\mathsf{L}$. 

\begin{thmintro}\label{thm_BFMain} The invariant generalized Ricci flow on $(\mathsf{G},[H_0])$ is gauge-equivalent to the following ODE of Dorfman brackets on $\f g \oplus \f g^*$:
\begin{equation} 
	\frac{\dd}{\dd t}\mud(t) = - \widetilde{\mc Rc}_{\mud(t)}\mud(t) + \mud(t)(\widetilde{\mc Rc}_{\mud(t)}\cdot,\cdot) + \mud(t)(\cdot,\widetilde{\mc Rc}_{\mud(t)}\cdot).
\end{equation}
\end{thmintro}
Here, for a twisted Dorfman bracket $\mud = \mud_{H,\mu}$, $\widetilde{\mc Rc}_\mud := \mc Rc_\mud - A_{\mud}$ where $\mathcal Rc_{\mud} \in \operatorname{End}(\f g \oplus \f g^*)$ is the generalized Ricci curvature of the infinitesimal exact Courant algebroid structure $(\f g \oplus \f g^*,\mud_H)$, and $A_{\mud} := \operatorname{Skew}_{\B{\mathcal G}}(\dd^*H)$, where we consider the two-form $\dd^*H \colon \f g \to \f g^*$ as a map from $\f g$ to $\f g^*$.

Let us remark that a bracket flow approach to the homogeneous generalized Ricci flow was considered by Paradiso in \cite{Par} using the action of $\mathsf{GL}(\f g) \leq \mathsf L$. The novelty of our approach, and what makes \Cref{thm_BFMain} work, is that we consider the bigger group $\mathsf L$ which acts transitively on the space of generalized metrics.
The two approaches agree only when $H_0$ is harmonic (see \Cref{rmk_par}).

Another important ingredient in our methods is a new formula for the generalized Ricci curvature of an invariant exact Courant algebroid resembling that of the classical Ricci curvature on a Riemannian homogeneous manifold:
\begin{equation}\label{eqn_GenBFIntro}
\mc Rc_{\mud} - A_{\mud} = \mm_{\mud} -\tfrac 12 \B{B} -\B{\operatorname{Sym}_g(\ad_U)}.
\end{equation}
Here, $U \in \f g$ is the \emph{mean curvature vector} defined by $g(U,X) = \tr(\ad_X),$ for $X \in \f g$, and $B$ is the \emph{Killing form} defined by $g(BX,Y) = \tr(\ad_X \circ \ad_Y)$ for $X,Y \in \f g$. For $F \in \operatorname{End}(\f g)$, we define $\B F := F \oplus -F^* \in \operatorname{End}(\f g \oplus \f g^*)$. Finally, $\mm_{\mud} \in \operatorname{End}(\f g \oplus \f g^*)$ is the \emph{moment map} for the natural action of $\mathsf{L}$ on the space of brackets $\Lambda^2(\f g \oplus \f g^*)^* \otimes (\f g \oplus \f g^*)$  (see \cite[Definition 1.2]{GIT20}).

The identification of a moment map in the structure of the generalized Ricci curvature allows us to use techniques from real geometric invariant theory \cite{GIT20}. In particular, for nilmanifolds with $H_0$ harmonic, it holds that $\mathcal Rc_\mud = \mm_{\mud}$. This is the key ingredient in the proof of \Cref{thm_main_nil} since after normalizing, \Cref{eqn_GenBFIntro} becomes a real analytic gradient flow (see \Cref{prop_BFConv}). If $H_0$ is not harmonic, we have $\widetilde{\mathcal Rc}_{\mud} = \mm_{\mud}$ which is not symmetric, as $\mathsf{L}$ is not a reductive Lie group.

There have been many other recent works on generalized Ricci flows with symmetries. The Bismut flat spaces - which are fixed points of the generalized Ricci flow - were classified by \cite{CS26, CS26bis, WYZ20}. They are covered by compact Lie groups with bi-invariant metric, and the form $H$ given by the contraction of the Lie bracket with said metric. By a recent result of Lee \cite{Lee24CY}, they are dynamically stable with respect to the GRF. Furthermore, Barbaro in \cite{Bar22, Bar23} proved  global stability for Bismut-flat metrics on complex  Bismut--flat manifolds with finite fundamental group with respect to the pluriclosed flow using a stability result appeared in  \cite{GFJS23}. The first examples of homogeneous Bismut--Ricci flat metrics which are non-flat were given by Podestà--Raffero \cite{PR22, PR24}. Lauret--Will in \cite{LW23} constructed further examples in this setting. The stability of many of these examples was studied by Gutierrez in \cite{GU24}. In \cite{PR23} Podestà--Raffero construct examples of $\mathsf{SO}(3)$-invariant steady, gradient generalized solitons on $\mathbb{R}^3$, resembling the structure of the classical Bryant soliton for the Ricci flow, see \cite{Brsol}. On the other hand, it was recently shown that all compact shrinking solitons of gradient type are in fact classical solitons \cite{LY24}. Moreover, in \cite{SU21},  Streets--Ustinovskiy classified  generalized K\"ahler--Ricci solitons on complex surfaces. Beyond soliton solutions, the generalized Ricci flow was studied on nilmanifolds by Paradiso \cite{Par} and on cohomogeneity-one manifolds with nilpotent symmetry by Gindi--Streets \cite{GS23}. The toric, Fano case of the generalized K\"ahler--Ricci flow was treated by Apostolov--Streets and Ustinovskiy in \cite{ASU22}.

The rest of the article is structured as follows: In \Cref{sec_GG}, we recall the notion of an exact Courant algebroid and generalized metrics. We also define \emph{scaled} exact Courant algebroids with the motivation of defining non-trivial expanding or shrinking solitons. We then discuss rigorously the generalized Ricci curvature and generalized Ricci flow. In \Cref{sec_solitons}, we define and discuss generalized Ricci solitons, and present in detail some examples on the Heisenberg group. In \Cref{sec_homGG} and \Cref{sec_BF}, we treat homogeneous generalized geometry and the generalized bracket flow, then derive \Cref{eqn_GenBFIntro}, and prove \Cref{thm_BFMain}. In Sections \ref{sec_solv} -- \ref{sec_LongNil} we prove our main theorems \Cref{thmIntro_solvmfdGRF} and \Cref{thm_main_nil} using the generalized bracket flow and techniques from real geometric invariant theory. We discuss in \Cref{sec_PCF} the relationship between the pluriclosed flow and the generalized Ricci flow in our (not necessarily compact) setting, leading to \Cref{cor_main_PCF}. Finally, in \Cref{sec_classifi} we classify generalized Ricci solitons on nilmanifolds in dimensions less than or equal to 4.

\subsection*{Acknowledgements} This project originated during the first-named author's visit to the University of Queensland. For this,  he wishes to express his deep gratitude to it and to the second and third author for their cheerful hospitality and for many inspiring conversations. The third-named author would like to thank Christoph B\"ohm, Alberto Raffero, and Luigi Vezzoni for fruitful conversations and their interest in this work. The first-named author is supported by GNSAGA of INdAM. The second-named author was supported by the ARC project DE190101063. The third-named author is funded by the Deutsche Forschungsgemeinschaft (DFG, German Research Foundation) under Germany's Excellence Strategy EXC 2044 –390685587, Mathematics Münster: Dynamics–Geometry–Structure and was previously supported by an Australian Government Research Training Program (RTP) Scholarship.








\section{Preliminaries on generalized geometry}\label{sec_GG}

In this section we review the basic notions of generalized Geometry and setup our notation. We mainly follow \cite{GRFbook}. We also introduce a scaling of exact Courant algebroids that will be central in our definition of soliton solutions for the generalized Ricci flow.  

 For additional motivation and details on exact Courant algebroids we refer the reader to \cite{Cou90, GuaAnnals, LWX97,Sletters, SV17, SV20} and the references therein.

\subsection{Exact Courant algebroids}

An exact Courant algebroid (ECA, for short) is a vector bundle $E\to M^n$ of rank $2n$ together with the following data:
\begin{itemize}
	\item A non-degenerate, symmetric bilinear form $\ip$ of signature $(n,n)$, called the \emph{neutral inner product};
	\item A bracket $\lb$ on $\Gamma(E)$, called the \emph{Dorfman bracket};
	\item A bundle map $\pi : E \to TM$, called the \emph{anchor map};
\end{itemize}
subject to the following axioms being satisfied for all $a,b,c \in \Gamma(E)$ and $f\in C^\infty(M)$:
\begin{enumerate}[({ECA}{\arabic*})]
	\item $[a,[b,c]] = [[a,b],c] + [b,[a,c]]$;
	\item $\pi[a,b] = [\pi a, \pi b]$;
	\item $[a,fb] = f[a,b] + \pi(a) fb$;
	\item $\pi(a)   \la b,c \ra  = \la [a,b], c \ra + \la b, [a,c]\ra$;
	\item $[a,b] + [b,a] =  (\pi^* \circ \dd) \la a,b \ra$;
	\item there is an exact sequence of vector bundles 
	\begin{equation}\label{eqn_exactseq}
		 0 \longrightarrow  T^*M \overset{\pi^*}\longrightarrow E \overset{\pi} \longrightarrow TM \longrightarrow 0 \, .
	\end{equation}
\end{enumerate}

In the last two axioms  we have abused notation and considered $\pi^*: T^* M \to E^*$ as a map $\pi^* : T^* M \to E$, after composing with the isomorphism $E^* \simeq E$ given by $\ip$. Notice also that  $\dd$ in  (ECA5) is the exterior differential in $M$,  and the bracket in the right-hand-side of (ECA2) is the Lie bracket of vector fields in $M$. The Dorfman bracket is not necessarily skew-symmetric, but it does satisfy the Jacobi identity, see \cite[Lemma 2.5]{GRFbook} for the proof. Although we will not be making use of it, it is worth mentioning that,   in the literature, the skew-symmetrization of the Dorfman bracket is called \emph{Courant bracket} and it is defined by:
$$
[a, b]_c:=\frac12([a,b]-[b, a])\,, \quad a,b \in \Gamma(E)\,.
$$ 
 Despite being skew-symmetric, the Courant bracket does not satisfy the Jacobi identity.  Indeed, one can show that
 $$
 [a, [b, c]_c]_c+[c, [a,b]_c]_c+[b, [c, a]_c]_c=\frac13\dd(\la[a, b]_c, c\ra+\la[b, c]_c, a\ra+\la[c, a]_c, b\ra)
  $$ holds for all $a, b,c\in\Gamma(E)$.
  
The first example of exact Courant algebroid is the generalized tangent bundle. On this, we can consider different structures of exact Courant algebroids depending on the choice of a closed $3$-form on $M$.

\begin{example} \label{ex_untwistedTB}
The generalized tangent bundle $TM \oplus T^* M$ has a natural structure of exact Courant algebroid, with data given by:
\begin{equation}\label{eqn_TT}
		\la X+\xi, Y+ \eta \ra := \tfrac12 (\xi(Y) + \eta(X)), \qquad [X+\xi, Y+\eta] := [X,Y] + \mc L_X \eta - \iota_Y \dd \xi, \qquad \pi(X+\xi) := X,
\end{equation}
for $X+ \xi, Y+ \eta\in \Gamma(TM\oplus T^*M)$. 

Moreover, for any closed $3$-form $H\in \Omega^3( M)$, we can consider another ECA structure on $\TT$, denoted by
\[
  \TTH := (TM\oplus T^*M, \ip,\lb_H,\pi).
\]
Here, $\ip$ and $\pi$ are as in \Cref{eqn_TT}, and $\lb_H$ denotes the \emph{$H$-twisted Dorfman bracket} defined by:
\begin{equation}\label{eqn_[]H-general}
		[X+\xi, Y+\eta]_H := [X,Y] + \mc L_X \eta - \iota_Y \dd \xi + \iota_Y\iota_XH,\qquad X+\xi,Y+\eta \in \Gamma(TM\oplus T^*M).
\end{equation}
\end{example}

Let  $(E_i \to M_i,\ip_i,\lb_i,\pi_i)$, $i = 1,2$,   be  two exact Courant algebroids. Recall that a vector bundle isomorphism is a pair $(F,f)$ where $f: M_1 \to M_2$, $F: E_1 \to E_2$ are diffeomorphisms, and,  for each $p\in M_1$, $F$ restricts to a linear isomorphism between the fiber over $p$ and the fiber over $f(p)$. We call such a pair an \emph{isomorphism of exact Courant algebroids} if,  for any $a,b\in \Gamma(E_1)$,  we have
\begin{itemize}
	\item $\la F a, F b \ra_2 = \la a, b\ra_1$;
	\item  $[F a, F b]_2 = F [a, b]_1$.
\end{itemize} 
It follows from  the second condition above that $\dd f\circ \pi_1 = \pi_2 \circ F$, see for instance \cite[Lemma 2.5]{KW07} for the detailed proof.\\
 The next Proposition guarantees that any exact Courant algebroid is isomorphic  to one described in \Cref{ex_untwistedTB}.

\begin{prop}\cite[Proposition~2.10]{GRFbook}\label{prop_EsimeqTTH}
Let $E$ be an exact Courant algebroid. Any isotropic splitting $\sigma$ of \eqref{eqn_exactseq} induces an isomorphism of exact Courant algebroids
\[
	E \simeq_\sigma (\TT)_H \,.
\]
Here $H \in \Omega^3 (M)$ is the closed $3$-form given by
\begin{equation}\label{eqn_H_sigma}
		H(X,Y,Z) = 2 \,  \la [\sigma X,  \sigma Y],  \sigma Z \ra, \qquad X,Y,Z \in \Gamma(TM).
\end{equation}
Moreover, given another isotropic splitting $\tilde \sigma$, the corresponding isomorphism satisfies 
\[
	E \simeq_{\tilde \sigma} (\TT)_{H + \dd b} \, ,
\]
for some $b\in \Omega^2( M)$.
\end{prop}

More generally, \v Severa's classification of exact Courant algebroids is the content of the following Theorem.

\begin{thm}\cite{Sletters} The isomorphism classes of exact Courant algebroids over $M$ are in one-to-one correspondence with \allowbreak $H^3(M,\R) / \operatorname{Diff}(M) = H^3(M,\R)/\Gamma_M$, where $\Gamma_M = \operatorname{Diff}(M)/\operatorname{Diff}_0(M)$ is the mapping class group of $M$.
\end{thm}

Given the correspondence in \Cref{prop_EsimeqTTH}, we can characterise \emph{automorphisms} of any exact Courant algebroid.  Before stating the  result itself, let us introduce two classes of maps which play an important role in generalized Geometry. Any $b\in \Omega^2 (M)$ may be viewed as a map $b : TM \to T^*M$ and thus it gives rise to an endomorphism of the generalized tangent bundle which in matrix form is represented by
\[
		\begin{pmatrix}0&0\\b&0\end{pmatrix} : T \oplus T^* \longrightarrow T \oplus T^*.
\]
We then set
\[
		e^b := \begin{pmatrix}\id&0\\b&\id\end{pmatrix} \in \End(T \oplus T^*), \qquad b\in \Omega^2 (M).
\] The latter transformations are classically known as \emph{B-field transformations} which appear naturally  as a generalization of the electromagnetic field in string theory. \\
Additionally,  fixed $f\in {\rm Diff}(M)$, we can produce a  bundle map  of the generalized tangent bundle covering $f$, as follows:
$$
\bar f:=\begin{pmatrix}\dd f & 0 \\0 & (f^{-1})^*\end{pmatrix}\colon \TT\to \TT\,.
$$
Then, the group of automorphism of $(\TT)_H$, and then of any ECA,  consists  of  bundle maps which are composition of  maps introduced above. 
%

\begin{thm} \cite[~Proposition 2.21]{GRFbook} For any closed $3$-form $H \in \Omega^3 (M)$, the group of  automorphisms of the exact Courant algebroid $\TTH$ is given by: 
\[
\operatorname{Aut}(\TTH) = \left\{\left(\bar f e^{b}
, f\right) : f \in \operatorname{Diff}(M),\quad b \in \Omega^2(M),\quad f^*H = H - \dd b\right\}.
\]
Moreover, the Lie algebra of $\operatorname{Aut}(\TTH)$ is
\[
\f{aut}(\TTH) = \{X + b \in \Gamma(TM) \oplus \Omega^2(M) : \mc{L}_XH = -\dd b\}.
\]
\end{thm}
 We refer the reader to \cite{RT20} for a detailed study of the ILH Lie group structure  on $\operatorname{Aut}(\TTH) $.\\
To conclude this Section, we observe that there is a natural way of scaling exact Courant algebroids. This will play a key role in the next Sections. 

\begin{lem}
Let $(E,\ip, \lb, \pi)$ be an exact Courant algebroid and let $c\in \mathbb{R}\backslash\{0\}$.  Then, the data $(c \, \ip, \lb, \pi)$ defines a new exact Courant algebroid structure on $E$. 
\end{lem}

\begin{proof}
Trivially,  $c \ip$ is a non-degenerate symmetric bilinear form of signature $(n,n)$, regardless of the sign of $c$. Moreover,  all the axioms except for (ECA5) in the definition of ECA clearly remain true after scaling $\ip$. As regards  (ECA5), we can notice that the $\pi^*$ in the right-hand-side involves a composition with the isomorphism $\varphi_{\ip} : E^* \simeq E$ induced by the neutral inner product. For $c\ip$,  this isomorphism is simply 
\begin{equation}\label{eqn_scaling_pi*}
	\varphi_{c\ip} = c^{-1} \varphi_{\ip},
\end{equation}
which gives the claim.
\end{proof}

\begin{defn}\label{def_scalingECA}
Let $(E,\ip,\lb,\pi)$ be an exact Courant algebroid and $c\in \mathbb{R}\backslash \{0 \}$. We define the \emph{scaled exact Courant algebroid} $c \cdot E$ to be the one associated with the data $(E, c \, \ip, \lb, \pi)$. 
\end{defn} 
We now study the effect of scaling on the  $3$-form $H$ after choosing an isotropic splitting:

\begin{lem}
Let $E$ be an ECA with isotropic splitting $\sigma$. Then,  the isomorphism $E\simeq_\sigma \TTH$, $H\in \Omega^3 (M)$ closed, induces an isomorphism
\[
	c \cdot E \simeq_\sigma (\TT)_{c H}.
\]
\end{lem}
\begin{proof}
This follows immediately from \Cref{prop_EsimeqTTH} and \eqref{eqn_H_sigma}.
\end{proof}






\subsection{Generalized metrics}\label{sec_genmet}

A \emph{generalized metric} on an ECA  $E$ is an orthogonal endomorphism $\gca \in \Or(E,\ip)$ such that the bilinear form 
\[
		(a,b) \mapsto \la \gca a, b \ra, \qquad a,b \in \Gamma(E),
\]
is symmetric and positive definite. We call the pair $(E,\mc G)$ a \emph{metric Courant algebroid}. A Courant algebroid isomorphism $(F,f) \colon (E_1, \mc G_1) \to (E_2, \mc G_2)$ is a \emph{metric Courant algebroid isometry} (or simply \emph{isometry}) if $F \circ \mc G_1 = \mc G_2 \circ F$.

The presence of a generalized metric on  $E$ allows us to consider a preferred isotropic splitting. We quickly recall  its construction since it will  be useful in what follows.  We refer the reader to \cite[Section 2.3]{GRFbook} for the detailed description.
Since   $\mc G^2={\rm Id}$,  we can consider the eigenbundles  $E_{\pm}$ of $E$ associated to $\pm 1$. It is not hard to see that $\pi\colon E_{\pm}\to TM$ is an isomorphism of vector bundles. Then, we can define 
$$
\sigma_{\pm}:= (\pi_{|E_{\pm}})^{-1}\colon TM\to E_{\pm}\,, \quad \tau_{\pm}(X, Y)=\la \sigma_{\pm}X, \sigma_{\pm}Y\ra\,, \quad  X, Y\in \Gamma(TM)\,. 
$$  
With these ingredients, we can consider the following splitting:
$$
\sigma=\sigma_{\pm}-\frac12\pi^*\tau_{\pm}\,.
$$  
It is easy to see that $\sigma $ is  isotropic and it does not depend on the choices of $\sigma_{\pm}$, see \cite[Lemma 2.29]{GRFbook} for the proof.  Moreover, we have that 
\begin{equation}\label{riemmetric}
\tau_{+}(X, X)=\la\mc G\sigma_{+}X, \sigma_{+}X \ra >0\,, \quad  X\in \Gamma(TM)\,, X\ne 0 \,,
\end{equation} 
thus determining a Riemannian metric $g$ on $M$. 


\begin{prop} \label{prop_SplittingFromMetric}\cite[~Proposition 2.38]{GRFbook} Let $E$ be an ECA with isotropic splitting $\sigma$.  Under the isomorphism $E\simeq_\sigma \TTH$, any generalized metric  on $E$ corresponds to a generalized metric on $\TTH$ of the form
\[
 \mc G(g,b) := e^b \begin{pmatrix}0&g^{-1}\\g&0\end{pmatrix} e^{-b},
\]
where $b \in \Omega^2(M)$ and $g$ is a Riemannian metric on $M$. Moreover, the bundle map
\begin{equation}\label{eqn_G(g,b)}
\left(e^{-b},\id_M\right)\colon \left(\TTH,\mc G(g,b)\right) \to \left((\TT)_{H + \dd b}, \mc G(g,0)\right),
\end{equation}
is an isometry.
\end{prop}

\begin{remark}
The isometry in \eqref{eqn_G(g,b)} indicates that a generalized metric $\mc G$ on an ECA $E$ with \v{S}evera class $\alpha\in H^3 (M, \mathbb{R})$ is equivalent to a choice of Riemannian metric $g$ on $M$ and a \emph{preferred representative} $H\in \alpha$. This can be also be observed by considering the \emph{preferred} isotropic splitting induced by $\mc G$ and the corresponding isomorphism $E\simeq \TTH$ under which $\mc G$ corresponds to $\mc G(g,0)$.
\end{remark}

Furthermore, the existence of a preferred isotropic splitting induced by a generalized metric $\mc G$ allows us to characterize generalized isometries on a fixed exact Courant algebroid.

\begin{prop}[{\cite[~Proposition 2.41]{GRFbook}}] \label{prop_Isometries} The group of generalized isometries of the metric  Courant algebroid  $(\TTH,\mc G(g,0))$ is given by
\[
\mathsf{Iso}(\TTH,\mc G(g,0)) = \left\{\left(\bar f 
, f\right): f \in \mathsf{Iso}(M,g),\quad f^*H = H\right\} \subset \operatorname{Aut}(\TTH).
\]
\end{prop}

Let us now examine the effect of scaling on generalized metrics.

\begin{lem}\label{lem_scaling}
Let $(E,\mc G) \simeq_\sigma (\TTH,\mc{G}(g,b))$ be a metric ECA and $c>0$. Then, $\pm \mc G$ is a generalized metric on $\pm c \cdot E$ and the isomorphism $\pm c\cdot E \simeq_\sigma (\TT)_{\pm cH}$ gives an isometry 
\[
		(\pm c \cdot E, \pm \gca) \simeq_\sigma ((\TT)_{\pm cH}, \gca(cg, \pm cb)).
\]
\end{lem}
\begin{proof}
The first claim is clear, since the  bilinear form $\pm c \la \pm \gca \cdot, \cdot \ra = c \, \la \gca \cdot, \cdot \ra$ is symmetric and positive definite.
As regards the second one,  we note that the scaling on the Riemannian metric follows directly from \eqref{riemmetric}. On the other hand,  considering  $\sigma_1$ to be the preferred isotropic splitting induced by $\mc G$, then $(\sigma_1-\sigma)(X)\in \ker \pi={\rm Im}(\pi^*)$, for any $X\in \Gamma(TM)$. Hence,  we can write
\begin{equation}\label{bsigmas}
b=(\pi^*)^{-1}(\sigma_1-\sigma)\colon TM\to T^*M\,,
\end{equation} where the fact that $b$ is a $2$-form is straightforward to check, using  that both $\sigma_1$ and $\sigma$ are isotropic. Then, the scaling on $b$ follows from \eqref{eqn_scaling_pi*}. \end{proof}

\begin{remark}\label{rmk_scal}
The isometry in the second claim  in \Cref{lem_scaling} can be also explicitly constructed.  Assuming $c>0$, we can use  the isotropic splitting $\sigma$ to obtain the isometry $(\pm c\cdot E, \mc G)\simeq_{\sigma}(\pm c\cdot(\TT)_H, \mc G(g, b))$. Then, one can easily see that
$$
\begin{pmatrix}{\rm Id}& 0 \\0& \pm c{\rm Id}
\end{pmatrix}\colon(\pm c\cdot(\TT)_H, \mc G(g, b))\to ((\TT)_{\pm cH}, \mc G(cg, \pm cb))
$$ is an isometry. Hence, composing the two isometries above, we obtain the claim.
\end{remark}

\subsection{The generalized Ricci curvature}\label{sec_genRic}

In this short Section, we will briefly discuss the explicit expression of the generalized Ricci curvature in  the case of $((\TT)_H, \mc G(g, 0 ))$ which will be central in what follows. For a detailed investigation of the general definition, we refer to \cite{CSCW11}, \cite{GF19}, \cite{GRFbook}, \cite{SV17}  and \cite{SV20}.  
Before doing that, we recall that if   $(M, g)$ is a Riemannian manifold  and $H\in \Omega^3(M)$,  there exist  unique connections $\nabla^{\pm} $, called \emph{Bismut connections}, such that 
$$
\nabla g=0\,, \quad gT^{\pm} =\pm H\,,
$$ where the tensor $T^{\pm}$ is the torsion of $\nabla^{\pm}$. In what follows, we will abuse the notation denoting with 
  $$
  {\rm Ric}^B_{g, H}:={\rm Ric}_g-\frac14H^2
    $$ the symmetric part of the Ricci tensor of the Bismut connections $\nabla^{\pm}$ associated to $g$ and $H$.  The symmetric $(0,2)$-tensor $H^2$ is defined by $ H^2(X, Y)=g(\iota_XH, \iota_YH)$,  for all $X, Y\in \Gamma(TM)$.\\\
 Then, following the notation in \cite{GRFbook}, the generalized metric $\gca (g, 0)$ determines two eigenbundles  as in \Cref{sec_genmet},  which,  in this particular case,   takes the following form: 
 $$E_{\pm}=\{X\pm g(X) :  X\in \Gamma(TM)\}\, .$$
   Now, using \cite[Definition 3.31]{GRFbook}  and \cite[Proposition 3.30]{GRFbook}, it is easy to see that, for any $X\in \Gamma(TM)$, 
\begin{equation}\label{eqn_ricpm}
\begin{aligned}
\mc{R}c^+(X-g(X))=&\, g^{-1}{\rm  Ric}^B_{g, H}X-\frac12g^{-1}\dd_g^*HX+ {\rm  Ric}^B_{g, H}X-\frac12\dd_g^*HX\,,\\
\mc{R}c^-(X+g(X))=&\, -g^{-1}{\rm  Ric}^B_{g, H}X-\frac12g^{-1}\dd_g^*HX+ {\rm  Ric}^B_{g, H}X+\frac12\dd_g^*HX\,.
\end{aligned}
\end{equation} 

$$
\mc Rc (\mc G)=\mc Rc^+-\mc Rc^-=\begin{pmatrix}
g^{-1}{\rm  Ric}^B_{g, H}&\frac 12\,g^{\,-1}\, \dd^*_gHg^{-1}\\
-\frac 12 \dd^*_g H&-{\rm  Ric}^B_{g, H}g^{-1}\end{pmatrix}.
$$
 In \cite[Theorem 3]{SV17},  the authors prove the ${\rm Aut}(E)$-equivariance  of the generalized Ricci curvature which will be extensively used  in the next Sections.

\subsection{The generalized Ricci flow}

 A one-parameter family $(\gca(t))_{t\in [0,T)}$ of generalized metrics on an ECA $E$ is a solution to the generalized Ricci flow if it satisfies
\begin{equation}\label{eqn_GRF}
	 \gca^{-1} \frac{\partial}{\partial t} \gca =-2  {\mc Rc}(\gca).
\end{equation}
If we use the isotropic splitting $\sigma_0$ associated to the initial metric $\gca(0)$ to identify $E \simeq_{\sigma_0} (\TT)_{H_0}$, then $\gca(t)$ corresponds to $\gca(g(t),b(t))$ for some one-parameter families of Riemannian metrics $g(t)$ and two-forms $b(t)$ on $M$. Then, Equation \eqref{eqn_GRF} is equivalent to the following coupled system: 
\begin{equation}\label{renormgroupflow}
\begin{aligned}
		\frac{\partial}{\partial t} g &= - 2 \, {\rm Ric}_{g, H}^B, \\
		\frac{\partial}{\partial t} b &= - \dd^*_g H,
\end{aligned}
\end{equation}
where $H(t) = H_0 + \dd  b(t)$.  Since $H_0$ is closed, it follows that $H(t)$ evolves accordingly to the classical heat equation: 
\[
		\frac{\partial}{\partial t} H = \Delta_{g} H .
\]

We should emphasize that \Cref{renormgroupflow} firstly appeared in \cite{CFMP85}  as the \emph{renormalization group flow}. We refer to \cite{ST17} for a detailed explanation of the relation between the renormalization group flow and generalized Geometry.

Since the generalized Ricci curvature is ${\rm Aut}(E)$-equivariant, the PDE system  \eqref{eqn_GRF} defining the generalized Ricci flow is only weakly parabolic.  A suitable application of the DeTurck  trick can be applied to obtain short-time existence and uniqueness, provided $M$ is compact, see \cite[Theorem 5.6]{GRFbook} for details. 

More generally, a one-parameter family $(\mc G(t))_{t\in[0,T)}$ of generalized metrics on an ECA $E$ solves the \emph{gauge-fixed generalized Ricci flow} if there exists a one-parameter family $(F_t)_{t\in[0,T)}\subseteq {\rm Aut}(E)$ so that 
$F_t\mc G(t)F_t^{-1} $ is a solution to the generalized Ricci flow \eqref{eqn_GRF}.  Using again the isotropic splitting $\sigma_0$ induced by $\mc G(0)$ to obtain $(E, \mc G(t))\simeq_{\sigma_0}( (\TT)_{H_0}, \mc G(g(t), b(t)))$ as above,  it can be seen that $\mc G(t)$ is a solution of the gauge-fixed generalized Ricci flow if only if there exists  a one-parameter family of generalized vector fields $X_t+B_t\in\f{aut}((\TT)_{ H_0})$ such that $(g(t), b(t))_{t\in [0, T)}$ is a solution of the following  coupled system:
\begin{equation}\label{eqn_gaugedGRF}
	\begin{aligned}
	\frac{\partial}{\partial t}g=&\, - 2 \, {\rm Ric}^B_{g, H}+ \mc L_Xg\,,\\
	\frac{\partial }{\partial t}b=&\, -\dd^*_gH-B+\mc L_Xb\,.
	\end{aligned}
\end{equation}
It follows that $H(t)$ will evolve by
\begin{equation}\label{eqn_gaugedGRF_H}
	\frac{\partial }{\partial t}H=\Delta_{g}H+ \mc L_XH\,.
\end{equation}


\section{Solitons of the generalized Ricci flow} \label{sec_solitons}
 Motivated by  \Cref{def_scalingECA} and   by the effect of scaling on the $3$-form $H$ described in \Cref{lem_scaling}, we give the following new definition of solitons for the generalized Ricci flow. 

\begin{defn}\label{defn_soliton}
A solution $(E, \gca(t))_{t\in [0,T)}$ to the generalized Ricci flow \eqref{eqn_GRF} is called a \emph{generalized Ricci soliton} if there exists a one-parameter family of scalings $c(t) > 0$, $c(0) = 1$,  and a one-parameter family of generalized isometries 
\[
		F_t : (c(t)\cdot E,  \gca(0)) \longrightarrow (E, \mc G(t)).
\]
\end{defn}

 Differentiating the family $F_t$ at $t=0$ gives rise to a pair $(X,  \bm{D})\in \Gamma(TM)\times {\rm End}(E)$ satisfying the following properties:
  \begin{enumerate}
  \item \label{der1}$\bm{D}\in {\rm Der}([\cdot, \cdot])$, i.e. $\bm{D}[a, b]=[\bm{D}a, b]+[a, \bm{D}b]$, for any $a, b\in \Gamma(E)$;
  \item \label{der2} $X\ip= \la ({\bm D}+ \lambda{\rm Id})\cdot, \cdot\ra+ \la\cdot,  ({\bm D}+ \lambda{\rm Id})\cdot\ra,$
 where $c'(0)=-2\lambda$;
 \item \label{der3} $ [\mc Rc(\mc G(0))+ \bm{D},\mc G(0) ]=0.$
  \end{enumerate}
  As one can easily see, the dependence of \eqref{der2} on the scalings is only through $c'(0)$. This motivates the following definition.

 \begin{defn}\label{defn_lambdader}
 Let $E$ be a  ECA and $\lambda\in \mathbb R$. We say that $(X, \bm{D})\in \Gamma(TM)\times {\rm End}(E)$ is a $\lambda$-derivation if \Cref{der1} and  \Cref{der2}  are satisfied.  We denote by ${\rm Der}_{\lambda}([\cdot, \cdot])$ the set of all $\lambda$-derivations. 
 \end{defn}
 It can easily be observed that  $0$-derivations are precisely the derivations of the Dorfman bracket introduced in \cite{GuaAnnals}. Then, $\lambda$-derivations has to be interpreted as the generalization of derivations when a scaling process is taken into account. 
 
 We can  give a characterization of both the set of isomorphisms  between $c\cdot(\TT)_H$ and $(\TT)_H$ and the set of $\lambda$-derivations which will be useful in the next Sections.
 \begin{lem}\label{lem_der} Let $H\in \Omega^3(M)$ be  closed and $c\in \R\backslash\{0\}$. Then, 
 $$
 {\rm Aut}(c\cdot(\TT)_H, (\TT)_H)=\{\bar f_c e^b :  b\in \Omega^2(M), \quad  f\in {\rm Diff}(M), \quad f^*H=c(H-\dd b)\}\,, 
 $$
  where $$
   \bar f_c=\begin{pmatrix} \dd f & 0 \\0 & c(f^{-1})^*\end{pmatrix}\,.
   $$
Moreover, if $\lambda \in \R$, 
 $$
 {\rm Der}_{\lambda}([\cdot, \cdot])=\{X+ b\in \Gamma(TM)\oplus \Omega^{2}(M) : \mathcal L_XH =-\dd b -2\lambda H \}\,.
 $$
 \end{lem}
 \begin{proof}
 Let $(F, f)\in  {\rm Aut}(c\cdot(\TT)_H, (\TT)_H)$. Since $F$ preserves the Dorfman bracket, we have that  $\pi\circ F=\dd f \circ\pi$. This naturally gives us that if
 $$
 F=e^B\begin{pmatrix}\dd f & 0\\ 0 & h^*
 \end{pmatrix}e^b
 $$ then $B=0$.  On the other hand,  the condition $\la F\cdot, F\cdot\ra=c \la\cdot,\cdot\ra$
 forces $b\in \Omega^2(M)$ and $\dd h=c\dd f^{-1}\,.$
  Moreover, we have that, for any $X+\xi, Y+\eta\in \Gamma (\TT)$,
 $$
 \begin{aligned}
\bar f_c\left [(\bar f_c)^{-1}(X+\xi),(\bar f_c)^{-1}(Y+\eta)\right]=&\,  [ X,  Y]+\mathcal L_{X}\eta- \iota_Y\dd\xi+ c\iota_{Y}\iota_{X}(f^{-1})^*H.\\
\end{aligned}
 $$ Composing this with the usual  action of  $e^b$,  we have that  if $F\in {\rm Aut}( c\cdot (\TT)_H, (\TT)_H) $ then 
 $$
  F=\bar f_ce^b\,, \qquad f^*H=c(H-\dd b)\,.
 $$ The viceversa is trivial, concluding the proof of the first claim. 
 
Let us consider a one parameter of scalings  $c(t)>0 $ such that $c(0)=1$ and $c'(0)=-2\lambda$ and  a one parameter family $
 (F_t, f_t)\in {\rm Aut}( c(t)\cdot(\TT)_H,  (\TT)_H).
 $  Differentiating $(F_t, f_t)$ at $t=0$ gives $X+b\in \Gamma(TM)\oplus \Omega^2(M)$ such that 
 $$
 \mathcal L_XH=-\dd b- 2\lambda H\,.
 $$ Viceversa, given $X+ b\in  \Gamma(TM)\oplus \Omega^2(M)$ such that $ \mathcal L_XH=-\dd b- 2\lambda H$,  we can consider $f_t$ the one parameter family of diffeomorphisms generated by $X$  and a family of scaling $c(t)>0$ such that $c(0)=1$ and $c'(0)=-2\lambda$. Hence, we can define 
 $$
 \bar b_t=\int_0^tc(s)f_s^*b_sds\,,
$$ where $b_s\in \Omega^2(M)$ such that $ \mathcal L_{X_s}H=-\dd b_s+ \frac{c'(s)}{c(s)} H$. 
 	Then, we have that 
 $$
 \begin{aligned}
\dd \bar b_t=&\,
-\int_0^t\frac{1}{c(s)}f_s^*\left(-\frac{c'(s)}{c(s)} H + \mathcal L_{X_s}H\right)\dd s 
 =-\int_0^t\frac{\dd}{\dd s}\left(\frac{1}{c(s)}f_s^* H\right)\dd s=-\frac{1}{c(t)}f_t^* H+ H\,.
 \end{aligned}
 $$ Then, $F=\bar f_{t, c(t)}e^{\bar b_t}\in {\rm Aut}(c(t)\cdot(\TT)_H,   (\TT)_H)$, as we wanted. 
\end{proof}
 \Cref{defn_soliton} is the dynamical definition of solitons for the generalized Ricci flow. As in the classical case, we can deduce from that an equivalent and static definition in terms of classical data, as the next Proposition shows.

\begin{prop}\label{prop_static_soliton}
Let $(E,\gca(t))$ be a generalized Ricci soliton, and consider the time-independent isometry $(E, \gca(t)) \simeq_{\sigma_0}   \left( (\TT)_{H_0}, \, \gca(g(t), b(t)) \right) $ induced by the isotropic splitting $\sigma_0$ associated to $\gca(0)$. Then, $g_0 := g(0)$ and $H_0$ satisfy
\begin{align}
\begin{split}
		  {\rm Ric}_{g_0, H_0}^B  &=  \lambda \, g_0   + \tfrac12 \,  \mc{L}_X g_0, \\
		 \Delta_{g_0} H_0 &=  -2 \, \lambda \, H_0  -  {\mc L}_X H_0,  \label{eqn_static_soliton}
\end{split}
\end{align}
where $\lambda = -\tfrac12  c'(0)$ and $c(t)>0$ are the scalings from \Cref{defn_soliton}. Conversely, if $\left( (\TT)_{H_0}, \, \gca(g_0, 0) \right)$ satisfies \eqref{eqn_static_soliton}, then there exists a soliton solution $\gca(g(t), b(t))$ to \eqref{eqn_GRF} on $(\TT)_{H_0}$ with $\gca(g(0), b(0)) = \gca(g_0,0).$ 
\end{prop}

\begin{proof}
By \Cref{lem_scaling},  the isomorphism associated to  $\sigma_0$ also induces  time-independent isometries 
\begin{align*}
		(c(t)\cdot E,  \gca(0)) &\simeq_{\sigma_0} \left((\TT)_{c(t) H_0}, \, \gca(c(t)g_0, 0) \right).
\end{align*}
It follows that $c(t) g_0$ and $c(t) H_0$ must solve the gauged-fixed equations \eqref{eqn_gaugedGRF} and \eqref{eqn_gaugedGRF_H}.

Thus,
\begin{align*}
		c' g_0 &= - 2 \, {\rm Ric}^B_{g_0, H_0} + c\,  \mc{L}_X g_0, \\
		c' H_0 &= \Delta_{g_0} H_0 +  c\, {\mc L}_X H_0.
\end{align*}
 Evaluating at $t=0$ and using $c(0) = 1$, $c'(0) = -2 \lambda$, \eqref{eqn_static_soliton} follows.
\end{proof}

\begin{remark}
When $H_0 = 0$, \Cref{prop_static_soliton} implies that \Cref{defn_soliton} is equivalent to the classical definition of Ricci soliton for the Ricci flow. 

On the other hand, if $c(t) \equiv 1$ is constant, we recover the definition of soliton solution given in \cite[$\S$4.4]{GRFbook}, which only allows for steady solitons (i.e.~$\lambda = 0$).  Recently, in \cite[Definition 2.1]{PR23}, a definition of generalized Ricci soliton, equivalent to \Cref{defn_soliton}, was given. 

 Moreover, in the non-steady case, the second equation in \eqref{eqn_static_soliton}, together with $\dd H_0=0 $ and the Cartan's formula, forces $[H_0]=0$. 
 In particular,  if $M$ is compact, this implies that  a non-steady generalized Ricci soliton with harmonic torsion is a classical Ricci soliton. 
\end{remark}
 To conclude this Section, we give a family of explicit examples of expanding generalized Ricci solitons.

\begin{example}\label{ex_solitons}
  Let us consider $M_0=H_3:={\rm Heis}(3, \R)$ the $3$-dimensional Heisenberg group.  It is well known that on $M_0$ we can choose coordinates $\{x, y, z\}$ so that 
  $$
  e_1=\frac{\partial }{\partial x}\,, \quad e_2=\frac{\partial }{\partial y}+ x \frac{\partial }{\partial z}\,, \quad e_3=\frac{\partial }{\partial z}
  $$  are left-invariant vector fields, forming a frame of the Lie algebra of $M_0$,  subjected to the following structure equation:
  $$
  [e_1, e_2]=e_3\,.
  $$ We will denote with $\{e^1, e^2, e^3\}$ the dual coframe of $\{e_1, e_2, e_3\}$ which is, in particular,  given by:
  $$
  e^1=\dd x\,, \quad e^2=\dd y\,, \quad e^3=\dd z-x\dd y\,.
  $$We next consider the following one parameter family of diffeomorphisms of $M_0$:
  $$
  \varphi_t(x, y, z)=((1+4t)^{\frac14}x, (1+4t)^{\frac14}y, (1+4t)^{\frac12}z)\,,\quad (x, y, z)\in M_0\,, \quad t\ge 0 \,.
  $$ It is easy to see that $\varphi_t$ is actually an automorphism of $M_0$ viewed as a Lie group. 
Another easy computation allows us to state that 
  $$
  \dd\varphi_t=(1+4t)^{\frac14}e^1+(1+4t)^{\frac14}e^2+ (1+4t)^{\frac12}e^3\,,\quad t \ge 0\,.
  $$ Now we consider the closed left invariant $3$-form  $H=\dd x\wedge \dd y\wedge \dd z$  and observe that $H$ is harmonic with respect to any  metric on $M_0$.  Then, one can easily see that 
$$
  \varphi_t^*H=(1+4t)H\,, \quad t\ge 0\,.
  $$
  Then, denoting $c(t)=1+4t>0$ and  comparing with \Cref{lem_der}, we are in the position to promote $\varphi_t$ to 
   \begin{equation}\label{eqn_geniso}
   F_t=\bar \varphi_{t, c(t)}\in {\rm Aut}(c(t)\cdot(\TT)_H, (\TT)_H)\,, \quad t\ge 0\,. 
   \end{equation} Moreover, we easily compute the vector field generated by $\varphi_t$:
   $$
   X=\frac{\dd}{\dd t}\Big|_{t=0}\varphi_t=x\frac{\partial }{\partial x}+y\frac{\partial }{\partial y}+ 2z\frac{\partial }{\partial z}=xe_1+ye_2+(2z-xy)e_3\,,
   $$ which,  in particular,  satisfies the following bracket relations:
   $$
   [X, e_1]=-e_1\,, \quad [X, e_2]=-e_2\,, \quad [X, e_3]=-2e_3\,.
   $$ Finally, $X\in {\rm Der}_{-2}([\cdot, \cdot])$ satisfies 
 $$
   \mathcal L_XH=4H
   $$ which is  precisely the second equation in \eqref{eqn_static_soliton} with $-2\lambda=c'(0)=4$, since $H$ is harmonic with respect to any Riemannian metric on $M_0$. Now, we consider the following Riemannian metric:
   $$
   g=e^1\odot e^1+ e^2\odot e^2+ e^3\odot e^3= \dd x\odot\dd x+ (1+x^2)\dd y \odot\dd y + \dd z \odot\dd z - 2x\dd y \odot\dd z\,,
   $$ which is the standard left-invariant Riemannian metric on $M_0$ with respect to the frame $\{e_1, e_2, e_3\}$.  It is easy to see that 
   $$
   {\rm Ric}^B_{g, H}=-(e^1\odot e^1+ e^2\odot e^2)  \mbox{ and } \mathcal L_Xg=2(e^1\odot e^1+ e^2\odot e^2+2e^3\odot e^3)\,.
   $$ Then, 
   $$
    {\rm Ric}^B_{g, H}=-2g+\frac12\mathcal L_X g
   $$  giving precisely the first equation in \eqref{eqn_static_soliton}. Now, we can use \Cref{prop_static_soliton} to infer that $((\TT)_H, \gca (g, 0))$ is an expanding soliton for the generalized Ricci flow. Moreover, we can also specify what actually is the one parameter family of generalized isometries in \Cref{defn_soliton}. Indeed,  considering $F_t$ as in \eqref{eqn_geniso}, we have
   $$
   F_t\gca(g, 0)F_t^{-1}=\gca((1+4t)(\varphi_t^{-1})^*g, 0)
   $$ but
   $$
   g(t)=(1+4t)(\varphi_t^{-1})^*g=(1+4t)^{\frac12}e^1\odot e^1+(1+4t)^{\frac12}e^2\odot e^2+e^3\odot e^3
   $$  is precisely the solution of the generalized Ricci flow starting from $g$, see \cite[Section 6]{Par}.  
   
   In the same fashion and with    a bit more effort, we can construct examples of expanding solitons  for the generalized Ricci flow on $M_k:=H_3\times \R^k$, for any $k\ge 1$.
    On $M_k$, we consider the following  left-invariant frame $\{e_1, e_2, e_3, e_4, \ldots, e_{k+3} \}$ of the Lie algebra of $H_3\times \R^k$  where $e_1, e_2, e_3$ are the same defined above while $e_{3+j}=\frac{\partial}{\partial x_{3+j}}$ where $x_{3+j}$ is the standard coordinate on the $j$-th copy of $\R$ in  the abelian factor $\R^k$. 

     Then, it is sufficient to consider $ c(t)=1+4t$ as scaling, 
     $$
      \varphi_t(x, y, z, x_4, \ldots, x_{3+k})=((1+4t)^{\frac14}x, (1+4t)^{\frac14}y, (1+4t)^{\frac12}z, (1+4t)^{\frac12}x_4, \ldots, (1+4t)^{\frac12}x_{3+k})\,,   $$ for all $(x,y,z,x_4, \ldots, x_{3+k})\in M_k$, 
   as the one parameter family of diffeomorphisms of $M_k$ , $H=e^1\wedge e^2\wedge e^3$ and finally the standard left-invariant metric
   $$
   g=\sum_{i=1}^{k+3}e^i\odot e^i\,,$$ with respect to the coframe $\{e^1, \ldots, e^{3+k}\}$ dual with respect to $\{e_1, \ldots, e_{k+3} \}$ and obtain an expanding soliton of the generalized Ricci flow on $M_k$. 
\end{example}

\section{Homogeneous generalized geometry} \label{sec_homGG}
In this Section we specialize our study  to the case where the underlying manifold $M = \G$ is a simply-connected Lie group. We will denote by $e$ the identity of $\G$ and $\f g = T_e\G$ its Lie algebra. We also ask the ECA to be compatible with the algebraic structure in the following sense:
\begin{defn}\label{defn_liECA}
We say that a metric ECA $(E \to \G,\mc G)$ is \emph{left-invariant} if the action of $\G$ on itself by left-translations lifts to an action on $(E,\mc G)$ by isometries (as defined in $\S$\ref{sec_genmet}).
\end{defn}
The action of $\mathsf G$ on $E$ in \Cref{defn_liECA} is a particular instance of a wider class of Lie group actions on ECAs called \emph{lifted} or \emph{extended} actions, see \cite[Definition 2.4]{BK15} and \cite[Definition 2.6]{BCG07}.\\
 The action  of $\mathsf G$ on a left-invariant  metric ECA allows  to consider $\mathsf G$-equivariant isotropic splittings.  The choice of  such isotropic splittings, using \Cref{prop_SplittingFromMetric}, gives rise to left-invariant classical data as the following Proposition highlights.
\begin{prop}\label{prop_Ggb}
Let $(E\to \G,\mc G)$ be a  left-invariant metric ECA. Then,  any $\G$-equivariant isotropic splitting $\sigma \colon T\mathsf G \to E$ of \eqref{eqn_exactseq} induces a $\G$-equivariant isometry 
\[
(E, \mc G) \simeq_\sigma (\TTH,\mc G(g,b)),
\]
where the tensors $g$, $H$, $b$ on $M$ are left-invariant and $H$ depends only on $\sigma$ and not on $\gca$.
\end{prop}
\begin{proof} 
If $\sigma$ is the isotropic splitting given by $\gca$, firstly,  we note that, by \Cref{prop_SplittingFromMetric}, $(E,\mc G)$ is isometric to $(\TTH,\mc G(g,0))$,  for some closed $3$-form $H$ and Riemannian metric $g$. Using this isomorphism, we can define an action of  $\G$ on $\TTH$, and by definition this action will be  by isometries of $\mc G(g,0)$. By \Cref{prop_Isometries},  $H$ and $g$ must be left-invariant.  It remains to prove the invariance of $b$.
As in \eqref{bsigmas}, the $2$-form $b$ can be expressed as 
$$
b=(\pi^*)^{-1}(\sigma_1-\sigma)\,,
$$ where $\sigma_1$ is the isotropic splitting induced by $\mc G$,  which is $\G$-equivariant. Then, the claim follows directly from the fact that, for any $(F, f)\in {\rm Aut}(E)$, we have that $F\circ\pi^*=\pi^*\circ(f^{-1})^*$ and  from  the  expression of $b$ above.
\end{proof}
\begin{remark}
As said, the preferred isotropic splitting induced by $\mc G$ is  $\G$-equivariant. Thus, any homogeneous metric ECA is isometric to $(\TTH,\mc G (g,0))$ for a left-invariant Riemannian metric $g$, and a closed left-invariant $3$-form $H \in \Omega^3(M)^\G$. 
\end{remark}

Notice that as a vector bundle any left-invariant ECA is trivial, since after choosing an isotropic splitting we get an isomorphism with $\TTH$, where 
\begin{equation}\label{eqn_trivialisation}
		T \G \oplus T^* \G \simeq (\gggo) \times \G
\end{equation}
because Lie groups are parallelizable by using left-invariant vector fields. By abuse of notation, we will simply denote 
\[
	(\gggo)_H := \TTH,
\]
where $H\in \Omega^3(\G)^\G \simeq \Lambda^3\ggo^*$ is a closed, left-invariant $3$-form on $\mathsf G$.

\subsection{The space of left-invariant generalized metrics as a homogeneous space}

Let $E \to \G$ be a left-invariant ECA, and let us fix a background left-invariant generalized metric  $\bG$ on $E$. The preferred isotropic splitting $\bar \sigma : T\G \to E$ induced by $\bG$ yields an isometry
\[
		(E \to \G, \bG) \simeq_{\bar \sigma} ( (\gggo)_{\bar H}, \mc G(\bar g, 0)),
\] 
for some left-invariant metric $\bar g$ and closed $3$-form $\bar H$ on $\G$. By abuse of notation we write $\bG = \mc G(\bar g, 0)$.
Let us set
\begin{align*}
	\mca^\G &:= \{ \gca : \gca \hbox{ is a left-invariant generalized metric on } (\gggo)_{\bar H} \}\\
	 &\simeq \{ \gca(g,b) : g \hbox{ left-invariant metric on }\G, b\in \Omega^2(\G)^\G \} 
\end{align*}
 where the  last identification is due to \Cref{prop_Ggb}.   Next, we consider the following Lie subgroup of $\Or(\gggo,\ip)$:
\begin{equation}\label{eqn_defL}
	\mathsf{L} := \left \{ e^b \cdot \begin{pmatrix} h &0\\0&(h^{-1})^*\end{pmatrix} : h\in \mathsf{GL}(\f g) ,\,\, b\in \Lambda^2 \ggo^* \right\},  \qquad e^b := \begin{pmatrix}\id&0\\b &\id\end{pmatrix}.
\end{equation}
As an abstract Lie group, $\Ll$ is isomorphic to the semi-direct product $ \Lambda^2 \ggo^* \rtimes \mathsf{GL}(\ggo)$. Moreover, its Lie algebra is 
\[
	\lgo := \left \{  \begin{pmatrix} A & 0\\ \alpha& -A^*\end{pmatrix} : A\in \f{gl(g)} ,\,\, \alpha\in \Lambda^2 \ggo^* \right\}.
\]

The standard representation of $\mathsf{O}(\f g \oplus \f g^*,\m{\cdot}{\cdot})$ on $\gggo$ gives rise to a representation of $\Ll$ on $\gggo$, and this extends in a standard way to an action of $\Ll$ on all tensor powers of $\gggo$. In particular, $\Ll$ acts on generalized metrics and Dorfman brackets via the following formulas:
\[
	\ell \cdot \gca  := \ell \circ  \gca \circ \ell^{-1}, \qquad \ell \cdot \lb := \ell [\ell^{-1 }\cdot, \ell^{-1} \cdot ].
\]
Recall that we also have the `change of basis' actions of $\mathsf{GL}(\ggo)$ on the spaces of brackets and $3$-forms: 
\[
		h\cdot \mu (\cdot, \cdot) = h \mu(h^{-1} \cdot ,h^{-1} \cdot), \qquad h\cdot \omega (\cdot,\cdot,\cdot) = \omega(h^{-1} \cdot ,h^{-1} \cdot ,h^{-1} \cdot ),  \qquad h\in \mathsf{GL}(\ggo),
\]
for $\mu \in \Lambda^2 \ggo^* \otimes \ggo$ and $\omega \in \Lambda^3 \ggo^*$.\\
The next Proposition allows to present the space of left-invariant metric on $E$ as a homogeneous space.
\begin{prop}\label{prop_L_trans_mcaG}
The action of $\Ll$  on the space of left-invariant generalized metrics $\mca^\G$  is transitive.
\end{prop}

\begin{proof}
For any $\ell \simeq (b, h) \in \Ll$, its action on the background metric is given by
\begin{align*}
	\ell \cdot \bG &= e^b   \begin{pmatrix} h &0\\0&(h^{-1})^*\end{pmatrix}   \begin{pmatrix} 0 & \bar g^{-1} \\  \bar g & 0 \end{pmatrix}   \begin{pmatrix} h^{-1} &0\\0&h^*\end{pmatrix}    e^{-b} = e^b  \begin{pmatrix} 0 & g^{-1}\\ g&0\end{pmatrix}  e^{-b}, 
\end{align*}
where 
\[
	g = (h^{-1})^* \bar g h^{-1} = h \cdot \bar g (\cdot ,\cdot).
\]
Since $\mathsf{GL}(\ggo)$ acts transitively on left-invariant metrics, the claim now follows from \Cref{prop_Ggb}.
\end{proof}



\subsection{Moving generalized metrics is equivalent to moving  Dorfman brackets}  In this Subsection, we will define and study the main properties of the space of left-invariant Dorfman brackets. A similar treatment of them can also be found in \cite{Par}. Then, we will focus on describing  the effect of the $\Ll$-action and its infinitesimal version  on Dorfman brackets.

From the axiom (ECA5), the Dorfman bracket associated with a left-invariant ECA is skew-symmetric.
Thus, we consider the vector space $\Lambda^2 (\gggo)^* \otimes (\gggo)$ of all skew-symmetric bilinear maps 
\[
		\mud : (\gggo) \times (\gggo) \to (\gggo).
\]
 Using \Cref{prop_EsimeqTTH}, we know that choosing an isotropic splitting $\sigma$ gives rise to  an isomorphism $E\simeq_{\sigma} (\gggo)_{ H}$. Then,  from left-invariance and \eqref{eqn_[]H-general}, it follows that the Dorfman bracket of $E$ gets identified with $\lb_H \in \Lambda^2 (\gggo)^* \otimes (\gggo)$ given by
\begin{equation}\label{eqn_[]H-leftinv}
	[X + \xi, Y + \eta]_H = [X,Y]_\ggo - \eta \circ \ad_\ggo(X) + \xi \circ \ad_\ggo(Y) + \iota_Y\iota_X H,
\end{equation}
where $\ad_\ggo(X) = [X,\cdot]_\ggo : \ggo \to \ggo$ and $H\in \Lambda^3 \ggo^*$ is closed. It is not hard to see that $\la \lb_H ,\cdot \ra$ is totally skew-symmetric. This motivates the following. 
\begin{defn}\label{def_Dorf_brackets}
The \emph{space of Dorfman brackets} is the algebraic subset of $\Lambda^2 (\gggo)^* \otimes (\gggo)$ given by
\[
		\mc D := \{ \mud \in \Lambda^2 (\gggo)^* \otimes (\gggo) \,\,  : \,\,  \la \mud(\cdot, \cdot), \cdot\ra \in \Lambda^3(\gggo)^*\,, \,\,  \mud(\ggo^*, \ggo^*) = 0, \,\,  \mc J(\mud) = 0\}.
\]
Here,  $\mc J (\mud) (a,b,c) := \mud(\mud(a,b),c) + \mud(\mud(b,c),a) + \mud(\mud(c,a),b) = 0$ is the Jacobi identity.
\end{defn}
 We can show that all elements in $\mathcal D $ can be written as in \eqref{eqn_[]H-leftinv} for suitable choices of the $3$-form $H$ and the Lie bracket on $\ggo$.
\begin{lem}\label{lem_mudexplicit}
We have that $\mud \in \mc D$ if and only if there exists a Lie bracket $\mu$ on $\ggo$ and a closed $3$-form $H\in \Lambda^3 \ggo^*$ such that, for any $X+\xi, Y+\eta\in \gggo$, 
\begin{equation}\label{eqn_mud}
	\mud(X + \xi, Y + \eta) = \mu(X,Y) - \eta \circ \mu_X + \xi \circ \mu_Y + \iota_Y\iota_X H, \qquad \mu_X := \mu(X,\cdot).
\end{equation} 
\end{lem}
\begin{proof}
Sufficiency was observed above. Regarding necessity, for each $X,Y, Z \in \ggo$,  we set 
\[
		\mu(X,Y) := \pi \circ \mud(X,Y), \qquad H(X,Y,Z) := 2 \, \la \mud(X,Y), Z \ra.
\]
The linear conditions on $\mud$  imply that,  for every $\xi \in \ggo^*$,  we have  $\mud(\xi,Y) \in \ggo^*$. Thus, 
\[
	 \mud(\xi, Y)(X)= 2\, \la\mud(\xi, Y), X\ra = 2\, \la\mud(Y, X), \xi \ra=  \xi\mu(Y, X)\,=   \xi \circ \mu_Y (X).
\]
By skew-symmetry, $\mud(X,\eta) = -\eta \circ \mu_X$, and from this \eqref{eqn_mud} follows. A simple computation shows that the Jacobi identity for $\mud$ implies the Jacobi identity for $\mu$ and that $\dd_\mu H = 0$.
\end{proof}

  Moreover, one can notice that the following conditions hold for brackets in $\mc D$: 
\begin{equation}\label{eqn_centralext}
 \mud (\ggo,\ggo) \subset \gggo,  \qquad
 \mud(\ggo,\ggo^*) \subset \ggo^*, \qquad 
 \mud(\ggo^*,\ggo^*) = 0.
\end{equation}
That is, $(\gggo,\mud)$ is a Lie algebra which is a central extension of $(\ggo,\mu)$. We thus have that the space of Dorfman brackets is contained in the linear subspace:
\begin{equation}\label{eqn_VD}
		V_{\mc D} := \left\{ \mud \in \Lambda^2 (\gggo)^* \otimes (\gggo) \, :\,   \mud \hbox{ satisfies (\ref{eqn_centralext}) and } \la \mud(\cdot, \cdot), \cdot\ra \in \Lambda^3(\gggo)^*  \right\} .
\end{equation}


We  then have two projections $(\cdot)_{\ggo} : \mc D \to \Lambda^2 \ggo^* \otimes \ggo$ and $(\cdot)_{\Lambda^3} : \mc D \to \Lambda^3 \ggo^*$,
\begin{equation}\label{eqn_projections}
	\mud \mapsto \mud_{\ggo} = \mu \in \Lambda^2 \ggo^* \otimes \ggo, \qquad \mud\mapsto  \mud_{\Lambda^3} = H \in \Lambda^3 \ggo^*,
\end{equation}
where here $\mu$ is simply the Lie bracket of $\ggo$.  Furthermore, we notice that, by \eqref{eqn_mud}, these projections completely determine a Dorfman bracket $\mud \in \mc D$. 

Let us first study how the action of $\Ll$ relates to the projections in \eqref{eqn_projections}:

\begin{lem}\label{lem_Lact_mud}
Let $\mud \in \mc D$ with $\mu := \mud_{\ggo}$, $H := \mud_{\Lambda^3}$. Then, for any $\ell=  \bar h\,e^\alpha \in \Ll$, we have
\begin{equation}\label{eqn:LactionOnBrackets}
(\ell \cdot \mud)_{\ggo} = h\cdot \mu, \qquad (\ell \cdot \mud)_{\Lambda^3} = h \cdot (H - \dd_\mu \alpha).
\end{equation}
\end{lem}
\begin{proof}
First, we compute for $X+\xi,Y + \eta \in \f g\oplus \f g^*$:
\begin{align*}
(e^\alpha\cdot \mud)(X + \xi,Y + \eta) &= e^\alpha \mud(X  + \xi - i_X \alpha ,  \, Y + \eta - i_Y \alpha)\\
&= e^\alpha \left( \mud(X+\xi, Y + \eta) + \alpha(Y, \mu(X,\cdot)) - \alpha(X, \mu(Y,\cdot))\right)  \\
&= \mud(X+\xi, Y + \eta)  -  i_Y i_X (\dd_\mu \alpha).
\end{align*}
Similarly,
\begin{align*}
\left(\begin{pmatrix}h&0\\0 &(h^{-1})^*\end{pmatrix} \cdot \mud \right)(X + \xi,Y + \eta) &= \begin{pmatrix}h&0\\0 &(h^{-1})^*\end{pmatrix}\mud(h^{-1}X + \xi \circ h, h^{-1}Y + \eta \circ h)\\
&= (h\cdot \mu)(X,Y) - \eta \circ (h\cdot \mu)_X  + \xi \circ (h\cdot \mu) + i_Y i_X (h\cdot H).
\end{align*}
The result follows from composing the two actions.
\end{proof}

Given the actions of $\Ll$ on Dorfman brackets on $\gggo$, and of $\mathsf{GL}(\ggo)$ on brackets and $3$-forms on $\ggo$, we denote the corresponding Lie algebra representations (or `infinitesimal actions') by
\[
	\Theta : \lgo \to \End(\Lambda^2 (\gggo)^* \otimes (\gggo)), \qquad \theta : \f{gl(g)} \to \End(\Lambda^2 \ggo^* \otimes \ggo), \qquad \rho : \f{gl(g)} \to \End(\Lambda^3 \ggo^*).
\] 
More precisely, they are given by
\begin{align*}
	   \Theta(A) \mud (\cdot ,\cdot) &= A \mud (\cdot, \cdot) - \mud(A \cdot, \cdot)  - \mud(\cdot, A \cdot), \qquad A\in \f{gl}(\gggo), \\
		\theta(A) \mu (\cdot ,\cdot) &= A \mu (\cdot, \cdot) - \mu(A \cdot, \cdot)  - \mu(\cdot, A \cdot),  \qquad A\in \f{gl}(\ggo), \\
		 \rho(A) \omega (\cdot, \cdot, \cdot) &= -\omega(A \cdot, \cdot, \cdot) -\omega(\cdot, A \cdot, \cdot) -\omega(\cdot,  \cdot, A \cdot),  \qquad A\in \f{gl}(\ggo).
\end{align*}
From \Cref{lem_Lact_mud} we immediately deduce: 
\begin{cor}\label{cor_ThetaL}
For any $L \simeq (\alpha, A) \in \lgo \simeq \Lambda^2 \ggo^* \rtimes \f{gl(g)}$,  we have
\[
	 (\Theta(L)\mud)_{\ggo}  = \theta(A) \mu, \qquad (\Theta(L) \mud)_{\Lambda^3} = \rho(A) H - \dd_\mu \alpha.
\]
\end{cor}

Since we will be moving the Dorfman and Lie brackets on $\gggo$ and $\ggo$ respectively, it is convenient to introduce the following notation: given $\mud$ a Dorfman bracket as in \eqref{eqn_mud}, we denote by
\[
		(\gggo)_{\mud} := (\gggo, \ip, \mud, \pi)
\]
the left-invariant ECA structure given by the data $(\ip, \mud, \pi)$, where $\ip$ and $\pi$ are defined in \eqref{eqn_TT}. 
Notice that in order for (ECA2) to be satisfied, we must necessarily have that the Lie bracket on $\ggo$ is $\mu := \mud_{\ggo}$.

\begin{prop} \label{prop:Liso} Let $\ell =  \bar h \, e^{\alpha} \in \Ll$, for $h \in \mathsf{GL}(\ggo)$ and $\alpha\in \Lambda^2 \ggo^*$. Then, the bundle map 
\[
	(\ell, h) : \left((\gggo)_H, (\ell^{-1}) \cdot \gca \right) \to \left( (\gggo)_{\ell \cdot \mud}, \gca \right)
\] is an isometry. Here $\mud_{\ggo} = \lb_\ggo$ is the original Lie bracket of $\ggo$ and $\mud_{\Lambda^3} = H$.
\end{prop}

\begin{proof}
We first prove the claim for $\ell = \bar h$ covering $h \in \mathsf{GL}(\ggo)$. By definition of the action, the map $h : (\ggo, \lb_\ggo) \to (\ggo, h\cdot \lb_\ggo)$ is a Lie algebra isomorphism. Set $\mu := h\cdot \lb_\ggo$, let $\G,\G_\mu$ be respectively the simply-connected Lie groups with Lie algebras $(\ggo,\lb_\ggo)$ and $(\ggo, \mu)$. Then, there exists a Lie group isomorphism $\varphi : \G \to \G_\mu$ for which $\dd_{e}\varphi = h$, where $e\in \G$ is the identity. The corresponding bundle map $\bar \varphi : T\G \oplus T^*\G \to  T\G_\mu \oplus T^*\G_\mu$ covering $\varphi$ clearly maps left-invariant sections to left-invariant sections, and since the trivialization \eqref{eqn_trivialisation} is done via left-invariant sections, it follows that $\bar \varphi$ may be represented by an element of $\mathsf{GL}(\gggo)$. The latter is precisely $\bar h$, and the pair $(\bar h, h)$ is the infinitesimal data corresponding to the ECA isomorphism 
\[
	\bar \varphi : (\gggo)_H \to (\gggo)_{\bar h\cdot \mud}.
\]
Indeed, $\bar h$ preserves the Dorfman bracket by definition of the action, and it preserves the neutral inner product because $\bar h\in \Or(\gggo,\ip)$. Finally, 
\[
		\bar h \cdot  (\bar h^{-1} \cdot \gca)  = \gca,
\]
thus we have an isometry. 

The proof for $\ell = e^\alpha$ covering $\id \in \mathsf{GL}(\ggo)$ is similar (in this case, $\G = \G_\mu$). 
\end{proof}

Let us now consider the effect of scaling when  Dorfman brackets are moving. Recall that, by \Cref{lem_scaling},  a choice of isotropic splitting $\sigma$ induces an isometry $(c\cdot E, \gca) \simeq_\sigma ((\gggo)_{cH}, \gca(cg, cb))$, for any $c>0$.

\begin{prop} \label{prop:scaledmu}
Let $c > 0$
 and let $\mud$ be the Dorfman bracket determined by $\mud_{\ggo} = \lb_\ggo$ and $\mud_{\Lambda^3} = H$. Then, the following bundle map is an isometry:
\[
		\left(\pm\overline{c^{1/2}\id} \,, \, \pm c^{1/2} \id \right) : ((\gggo)_{\pm cH}, \gca(cg, \pm cb)) \longrightarrow  ((\gggo)_{\pm c^{-1/2} \mud}, \gca(g,b)).
\]
\end{prop}
\begin{proof}
First, we see that it is an ECA isomorphism. Indeed, we have that $\pm\overline{c^{1/2}\id} \in \Or(\gggo,\ip)$, 
and 
\[
	 \pm\overline{{c^{1/2}}\id} \left[(\pm\overline{c^{1/2}\id})^{-1}  \,\cdot \,, (\pm\overline{c^{1/2}\id})^{-1} \, \cdot \, \right]_{cH} = \pm c^{-1/2} \mud(\cdot, \cdot)
\] 
can be easily verified using  \eqref{eqn_[]H-leftinv} and \eqref{eqn_mud} (with $\mu = \lb_\ggo$). Regarding the isometry claim, we directly compute:
\begin{align*}
	 \left(\overline{c^{1/2}\id} \right) \, \gca(cg,cb) \, \left(\overline{c^{1/2}\id}\right)^{-1} &=
	 \begin{pmatrix} c^{1/2} \id & 0\\ 0 &  c^{-1/2} \id\end{pmatrix}  
	 \, e^{cb} \,
	 \begin{pmatrix} 0  & c^{-1}g^{-1}\\ c g &  0 \end{pmatrix} 
	 \, e^{-cb} \,
	 \begin{pmatrix} c^{-1/2} \id & 0\\ 0 &  c^{1/2} \id\end{pmatrix}   \\
	 & = \begin{pmatrix} -g^{-1} b &  g^{-1}\\ g - b g^{-1} b & b g^{-1} \end{pmatrix}  
	  = \gca(g,b).
\end{align*}
\end{proof}

\subsection{The moment map} 

In this section we define a moment map for the action of $\Ll$ on Dorfman brackets, in the sense of real GIT (see \cite{RS90,HSchw07,GIT20}). In \Cref{prop_Rcmm} we show that this moment map encodes the most complicated part of the generalized Ricci curvature of left-invariant generalized metrics.

As before, we fix a background generalized metric $\bG$ on a left-invariant ECA, yielding an isomorphism with $((\gggo)_{\bar H}, \gca(\bar g, 0))$. The metric $\bar g$ gives rise to an inner product on $\Lambda^2 (\gggo)^* \otimes (\gggo)$,  also denoted by $\bar g$. To define it, we fix a $\bar g$-orthonormal basis $\{ e_i\}$ of $\ggo$  with dual basis $\{e^i\}$ and set
\[
		\bar  g(\mud,\nud) := 2 \, \sum_{i,j} \left\la \bG \mud(e_i,e_j), \nud(e_i,e_j)\right\ra + 4 \, \sum_{i,j} \left\la \bG \mud(e_i,e^j), \nud(e_i,e^j)\right\ra + 2 \, \sum_{i,j} \left\la \bG \mud(e^i,e^j), \nud(e^i,e^j)\right\ra.
\] 
The overall factor of $2$ is due to the $\tfrac12$ factor in the definition of $\ip$. Since we are only interested in Dorfman brackets and their infinitesimal variations, we will mostly work with brackets in $V_{\mc D}$ (see \eqref{eqn_VD}). Then, the only non-vanishing structure coefficients  are: 
\[
		\mud(e_i, e_j) = \mud_{ij}^k e_k + \mud_{ijk} e^k, \qquad \mud(e_i, e^j) = -\mud_{ik}^j e^k = -\mud(e^j, e_i),
\]
using Einstein's summation convention (summing over all $i,j,k$ and \emph{not} just $i<j$). The fact that the coefficient of $e^k$ in $\mud(e_i, e^j)$ is related to that of $e_j$ in $\mud(e_i,e_k)$ follows from $\la \mud(\cdot,\cdot), \cdot \ra \in \Lambda^3(\gggo)^*$. We then have
\begin{equation}\label{eqn_gDorfman}
	\bar g(\mud, \nud)	:= \sum_{i,j,k} \left( 3 \,\mud_{ij}^k \nud_{ij}^k  + \mud_{ijk} \nud_{ijk}\right) = 3 \, \bar g(\mu, \nu) + \bar g(H, \tilde H), \qquad \mud, \nud \in V_{\mc D}, 
\end{equation}
where $\mu = \mud_{\ggo}, \nu = \nud_{\ggo},  H = \mud_{\Lambda^3},  \tilde H = \nud_{\Lambda^3}$. Here we have used the extension of $\bar g$ to a positive-definite inner product on $\Lambda^2 \ggo^* \otimes \ggo$, and on  $\Lambda^k \ggo^*$, given respectively by
\[
		\bar g(\mu,\nu) :=  \sum_{i,j,k} \mu_{ij}^k \nu_{ij}^k, \qquad  \mu,\nu \in \Lambda^2 \ggo^* \otimes \ggo, \qquad \mu_{ij}^k := \bar g(\mu(e_i,e_j), e_k).
\]
and
\[
	\bar g(\alpha, \beta) := \sum_{i_1, \ldots, i_k} \alpha_{i_1 \cdots i_k}  \beta_{i_1 \cdots i_k}, \qquad \alpha,\beta \in \Lambda^k \ggo^*, \qquad \alpha_{i_1\cdots i_k} := \alpha(e_{i_1}, \ldots, e_{i_k}). 
\]

%
%

The metric $\bar g$ also determines  a  maximal compact subgroup $\Or(\ggo, \bar g) \leq \Ll$. We fix on $\lgo$ the $\Ad(\Or(\ggo,\bar g))$-invariant, positive-definite inner product $\bar g_\lgo$ given by 
\begin{equation}\label{iponmfl}
	 \bar g_\lgo \left(\begin{pmatrix}A& 0 \\ \alpha & -A^*\end{pmatrix}, \begin{pmatrix}B& 0 \\ \beta & -B^*\end{pmatrix} \right) :=  2 \tr A B^T + \frac 16 \tr \beta^* \alpha = 2 \,  \sum_{i,j} A_{ij} B_{ij} + \frac16 \sum_{i,j} \alpha_{ij} \beta_{ij},
\end{equation}
where $A_{ij} := g(A e_i, e_j)$, $\alpha_{ij} = \alpha(e_i, e_j)$, etc, for a $g$-orthonormal basis $\{ e_i \}$. While the  $1/6$-factor looks arbitrary, it will play a key role in the forthcoming computations.


\begin{defn}\label{def_momMap}
The \emph{moment map} for the action of $\Ll$ on $(\Lambda^2 (\gggo)^* \otimes (\gggo), \bar g)$ is the map 
\[
	\mm : \Lambda^2 (\gggo)^* \otimes (\gggo) \to \lgo, \qquad  \mud \mapsto \mm_\mud,
\]
 implicitly defined by
\[
		\bar g_\lgo(\mm_\mud, L) =  \frac 16  \, \bar g(\Theta(L)\mud, \mud), \qquad   L\in \lgo.
\]
\end{defn}

\begin{prop}
The moment map is $\Or(\ggo,\bar g)$-equivariant. Moreover,  for each $\mud \in \Lambda^2 (\gggo)^* \otimes (\gggo)$ and $ L\in \lgo$,  it satisfies
\[
		\bar g_\lgo(\mm_\mud, L) = \frac1{12} \, \frac{\dd}{\dd t}\Big|_{t=0} \lvert \exp(t L)\cdot \mud \rvert^2_g\,.
\]
\end{prop}

\begin{proof}
 Let $K\in \Or(\ggo,\bar g)$, for every $L\in \lgo$,  we have that 
$$
\bar g_{\lgo}(\mm_{\overline K \cdot \mud}, L)=
\frac16\bar g(\overline K ^{-1}\cdot\Theta(L)(\overline K\cdot\mud), \mud)=\frac16\bar g(\Theta(\overline K ^{-1}L\overline K )\mud, \mud)
= \bar g_{\lgo}(\mm_{\mud}, \overline K ^{-1}L\overline K )=\bar g_{\lgo}(\overline K \mm_{\mud}\overline K^{-1},L )
$$
yielding the first claim. The second one is trivial computing the derivative on the right hand side. 
\end{proof}

We now define ${\mc Rc}_\mud \in \End(\gggo)$ to be the generalized Ricci curvature of $((\gggo)_\mud, \bG)$ (at the identity). We also denote by ${\rm Ric}_\mu \in \End(\ggo)$ the $(1,1)$-Ricci tensor of $(\ggo, \mu,\bar g)$, and by $\mm_\mu \in \End(\ggo)$ the moment map for the action of $\mathsf{GL}(\ggo)$ on $\Lambda^2 \ggo^* \otimes \ggo$:
\[
		\bar g(\mm_\mu, A) = \frac14 \bar g(\theta(A) \mu, \mu), \qquad A\in \f {gl(g)}, \qquad \mu \in \Lambda^2 \ggo^* \otimes \ggo.
\]
The factor $\tfrac14$ makes this consistent with the standard notation in the homogeneous Riemannian setting, see for instance \cite{alek}. In this notation, for $\mud \in {\mc D}$ with $\mud_{\ggo} = \mu$, $\mud_{\Lambda^3} = H$, we have
\begin{equation}\label{eqn_Rcmud}
	{\mc Rc}_\mud = 
	\begin{pmatrix} \Ricb_{\mu,H} &\frac 12 \bar g^{-1}(\dd_\mu^*H) \, \bar g^{-1}\\
		-\frac 12 \dd^*_\mu H&  - (\Ricb_{\mu,H})^*\end{pmatrix}, \qquad 
		\Ricb_{\mu,H} := {\rm Ric}_\mu - \frac 14 \,  \bar g^{-1} H^2.
\end{equation}

\begin{prop}\label{prop_Rcmm}
The generalized Ricci curvature and the moment map for the action of $\Ll$ on Dorfman brackets are related by 
\[
		{\mc Rc}_\mud =  \mm_\mud +  \overline{{\rm Ric}_\mu} - \overline{\mm_\mu} + A_\mud\,, 
\]
for some $A_\mud\in \f{so}(\gggo,\bG) \cap \f{so}(\gggo,\ip)$.
\end{prop}

\begin{proof}
Let us set
\begin{equation}\label{eqn_A_mud}
		 A_\mud := 
		\begin{pmatrix} 0	  &	\frac 12 \bar g^{-1}(\dd_\mu^*H) \, \bar g^{-1}\\
								\frac 12 \dd^*_\mu H &  0 \end{pmatrix}.
\end{equation}
It is not hard to see that indeed $A_\mud\in \f{so}(\gggo,\bG) \cap \f{so}(\gggo,\ip)$. We are thus left with
\[
		{\mc Rc}_\mud - A_\mud = 
		\begin{pmatrix} \Ricb_{\mu,H} & 0 \\
		- \dd^*_\mu H&  - (\Ricb_{\mu,H})^*\end{pmatrix}  \in \lgo.
\]
Furthermore, one can easily see that \eqref{eqn_A_mud} defines the unique element in $\f{so}(\gggo,\bar{\gca}) \cap \f{so}(\gggo,\ip)$ such that ${\mc Rc}_\mud - A_\mud \in \lgo $.

 Moreover, for any $L \simeq \left(  B, \beta  \right)\in \lgo$, we compute using the definition of $\bar g_\lgo$:
\begin{align*}
	\bar g_\lgo \left({\mc Rc}_\mud - A_\mu, L\right)  &= 
		2\operatorname{tr}(\operatorname{Ric}_\mu B)  - \frac12 \tr (H^2 B) - \frac 16 \bar g(\dd^*_\mu H,\beta).
\end{align*}
Furthermore, we notice that
\begin{align*}
 \bar g(\rho(B)H,H) &= \sum_{i,j,k} (\rho(B)H)(e_i,e_j,e_k)H(e_i,e_j,e_k) = -3 \sum_{i,j,k} H(Be_i,e_j,e_k)H(e_i,e_j,e_k)\\
&= -3 \sum_i \bar g(\iota_{e_i}H,\iota_{Be_i}H) = -3 \bar g(H^2e_i,Be_i) = -3\tr(H^2B),
\end{align*}
from which we get
\[
	\bar g_\lgo \left( {\mc Rc}_\mud - A_\mud  , L\right)  = 
		\bar g_\lgo \left(\overline{{\rm Ric}_\mu}, L \right) +  \frac 16  \, \bar  g(\rho(B)H - \dd_\mu \beta,H).
\]
On the other hand, by \Cref{cor_ThetaL} and \eqref{eqn_gDorfman} the moment map satisfies:
\begin{equation}\label{eqn_mommap}
\begin{aligned}
 \bar g_\lgo \left( \mm_\mud, L \right) = \frac16 \, \bar g(\Theta(L) \mud, \mud) &= \frac12 \, \bar g(\theta(B) \mu,\mu) + \frac16\, \bar g( \rho(B) H - \dd_\mu \beta, H) \\
	&= 2 \, \bar g(\mm_\mu, B) + \frac16\, \bar g( \rho(B) H - \dd_\mu \beta, H) \\
	& = \bar g_\lgo \left( \overline{\mm_\mu}, L \right) + \frac16\, \bar g( \rho(B) H - \dd_\mu \beta, H),
\end{aligned}
\end{equation}
and the Proposition follows.
\end{proof}

\begin{cor}\label{cor_mmnil}
If the Lie algebra $(\ggo,\mu)$ is nilpotent, then 
\[
	{\mc Rc}_\mud  - A_\mud =  \mm_\mud, \qquad A_\mud\in \f{so}(\gggo,\bG) \cap \f{so}(\gggo,\ip).
\]
\end{cor}
\begin{proof}
Lauret's formula for the Ricci curvature of left-invariant metrics in terms of the moment map on the variety of Lie algebras \cite{minimal} yields
\[
		{\rm Ric}_\mu = \mm_\mu - \frac12 \, B_\mu - \frac12 (\ad_\mu U + (\ad_\mu U)^t ),
\] 
where $\bar g(B_\mu \cdot, \cdot)$ is the Killing form of $(\ggo,\mu)$, and $U\in \ggo$ denotes the \emph{mean curvature vector} of the metric Lie algebra $(\ggo,\mu,\bar g)$, which vanishes if and only if $(\ggo,\mu)$ is unimodular. In particular, for nilpotent Lie algebras --for which it is known that the Killing form vanishes-- this gives ${{\rm Ric}_\mu} = \mm_\mu$, and the stated formula follows from \Cref{prop_Rcmm}.
\end{proof}

\section{A  flow of Dorfman brackets} \label{sec_BF}

The fact that the Lie group $\Ll$ (see \eqref{eqn_defL}) acts transitively on the space of generalized left-invariant metrics $\mca^\G$ on $(\gggo)_H$ (\Cref{prop_L_trans_mcaG}) and the equivalence between acting on generalized metrics and acting on Dorfman brackets (\Cref{prop:Liso}) lead us to consider what would be a natural counterpart to the generalized Ricci flow of left-invariant metrics, on the space of Dorfman brackets $\mc D \subset \Lambda^2(\gggo)^* \otimes (\gggo)$.

As in previous Sections, we fix a left-invariant ECA $(\gggo)_{\bar H}$ and a background left-invariant generalized metric $\bG = \gca(\bar g, 0)$ on it.

\begin{defn}
The \emph{generalized bracket flow}  with background $\left( (\gggo)_H, \bG \right)$ is the initial value problem for a curve of brackets $\mud(t)$ in $\Lambda^2 (\gggo)^* \otimes (\gggo)$:
\begin{equation}\label{eqn_genbf}
		\frac{\dd}{\dd t} \mud = - \Theta({\mc Rc}_\mud - A_\mud) \mud, \qquad \mud(0) = \mud_0,
\end{equation}
where ${\mc Rc}_\mud$ and $A_\mud$ are respectively defined in \eqref{eqn_Rcmud} and \eqref{eqn_A_mud}. 
\end{defn} 

Since a Dorfman bracket $\mud \in \mc D$  (see \Cref{def_Dorf_brackets}) is determined by its two projections $\mud_{\ggo}, \mud_{\Lambda^3}$, it is natural to expect the generalized bracket flow to be equivalent to a coupled flow of Lie brackets on $\ggo$ and of closed $3$-forms. Indeed:

\begin{prop}
The generalized bracket flow \eqref{eqn_genbf} with initial condition a Dorfman bracket $\mud_0 \in \mc D$ is equivalent to the following coupled flow of brackets $\mu(t)$ and $3$-forms $H(t)$:
\begin{equation}\label{eqn_muH}
\begin{cases}
		 \frac{\dd}{\dd t} \mu = - \theta(\Ricb_{\mu,H}) \mu ,\qquad &\mu(0) = \mu_0 := (\mud_0)_{\ggo},  \\
		\frac{\dd}{\dd t} H =  \Delta_\mu H - \rho(\Ricb_{\mu,H}) H, \qquad &H(0) = H_0 := (\mud_0)_{\Lambda^3}.
		\end{cases}
	\end{equation}
\end{prop}

\begin{proof}

 Let $(\mu_t, H_t)_{t\in [0,T)} $ be a solution of \eqref{eqn_muH}.  Then, we  define 
 $$
 \mud_t(X+\xi, Y+\eta)=\mu_t(X, Y)-\eta\circ (\mu_t)_X+ \xi \circ (\mu_t)_Y + \iota_Y\iota_X H_t\,, \quad X, Y\in \ggo\,,\,\, \xi, \eta\in \ggo^*,\,\, t \in[0,T).
 $$ 
 Using \Cref{lem_mudexplicit}, we have that $\mud_t\in \mc D$\,, for all $t\in [0,T)$. Now, using \eqref{eqn_muH} and \Cref{cor_ThetaL}, we have that 
 $$
 \frac{\dd}{\dd t}\mud_t=-\Theta(\mc Rc_{\mud_t}- A_{\mud_t})\mud_t\,, \quad \mud(0)=\mud_0\,.
 $$
 Then,  since they solve the same ODE with the same initial datum, $(\mud_t)_{t\in[0,T)}$ coincides with the solution of the generalized bracket flow starting from $\mud_0$.

Viceversa, let $(\mud_t)_{t\in [0,T')}$ be  a solution of the generalized bracket flow starting from $\mud_0$.  For the sake of simplicity, we will denote $\tilde{\mu}_t:=(\mud_t)_{\ggo}$ and $\tilde{H}_t:=(\mud_t)_{\Lambda^3}$.  By definition, $\mud_t$ belongs to the $\mathsf L$-orbit of $\mud_0$. Using the fact that  the $\mathsf L$-action does not change $\pi$ and \Cref{cor_ThetaL}, we obtain  that
\begin{equation}\label{eqn_mutildet}
\frac{\dd}{\dd t}\tilde{\mu}=\pi\frac{\dd}{\dd t}\mud=-\pi\Theta(\mc Rc_{\mud}- A_{\mud})\mud=-\theta({\rm Ric}_{\tilde{\mu}, \tilde H}^B)\tilde{\mu}\,.
\end{equation}
Moreover, using again the fact that the $\mathsf L $-action acts trivially on $\ip$, we have, for all $X, Y, Z\in \ggo$, 
\begin{equation}\label{eqn_Htildet}
\begin{aligned}
\frac{\dd}{\dd t }\tilde{H}(X, Y, Z)=&\,2\frac{\dd}{\dd t}\la\mud( X,  Y),  Z\ra=-2\la \Theta (\mc Rc_{\mud}- A_{\mud})\mud( X,  Y),  Z\ra\\
=&\, -(\Theta (\mc Rc_{\mud}- A_{\mud})\mud)_{\Lambda^3}(X, Y, Z)=(\Delta_{\tilde{\mu}}\tilde{H}-\rho({\rm Ric}_{\tilde{\mu}, \tilde{H}}^B)\tilde{H})(X, Y, Z)\,. 
\end{aligned}
\end{equation}
where the last equality is due again to \Cref{cor_ThetaL}. \Cref{eqn_Htildet} and \Cref{eqn_mutildet} readily guarantees that $(\tilde{\mu_t}, \tilde{H}_t)_{t\in[0,T')}$ coincides with the solution of \eqref{eqn_muH}, concluding the proof. \end{proof}

 Before going deeply into the discussion on how the generalized bracket flow is connected to the generalized Ricci flow, we state the following Lemma which will be useful in  what follows.
 \begin{lem}\label{lem:evdstar}
  Let $((\gggo)_{\mud(t)}, \bar{\mc G})_{t\in [0, T)}$ be a solution of the generalized bracket flow starting from $\mud$. Then, 
  $$
   \frac{\dd}{\dd t}\lvert \dd_{ \mu}^* H\rvert^2= 2(\bar g(\rho({\rm Ric}_{ \mu, H}^B)\dd_{ \mu}^* H,\dd_{ \mu}^* H) -2\bar g(\rho({\rm Ric}^B_{\mu, H}) H, \dd_\mu\dd^*_{\mu} H )-\lvert \Delta_{ \mu} H\rvert^2)\,.
   $$ In particular, if $\dd_{\mu(0)}^*H(0)=0$, then $\dd_{\mu}^*H=0$, for any $t\in[0,T)$. 
 \end{lem}
\begin{proof}
First of all, we  easily see that, for all  $A \in \f{gl}(\f g)$ and brackets $\mu$, one has
   \begin{equation}\label{eqn_dRhocom}
   \dd_{\theta(A)\mu}^* = [\dd_\mu^*,\rho(A^t)].
    \end{equation}
Then, we can use \eqref{eqn_dRhocom} to infer that 
   $$
   \begin{aligned}
   \frac{\dd}{\dd t}\dd^*_{ \mu} H
   =&\, [\rho({\rm Ric}_{ \mu, H}^B), \dd^*_{\mu}] H-\dd^*_{ \mu}\rho({\rm Ric}_{ \mu, H}^B)H+\Delta_{ \mu}\dd^*_{ \mu} H
   =  \rho({\rm Ric}_{ \mu, H}^B)\dd^*_{ \mu} H - 2\dd^*_{ \mu}\rho({\rm Ric}_{ \mu, H}^B) H+ \Delta_{ \mu}\dd^*_{ \mu} H\,.
   \end{aligned}
   $$ 
     Then, we easily obtain that 
   $$
   \begin{aligned}
   \frac{\dd}{\dd t}\lvert \dd_{ \mu}^* H\rvert^2  =&\, 2(\bar g(\rho({\rm Ric}_{ \mu, H}^B)\dd_{ \mu}^* H,\dd_{ \mu}^* H) -2\bar g(\rho({\rm Ric}^B_{\mu, H}) H, \dd_{\mu}\dd^*_{\mu} H )-\lvert \Delta_{ \mu} H\rvert^2)\,,
   \end{aligned}
   $$ as claimed. The second part follows trivially from the first one. 
  \end{proof}

\begin{remark} \label{rmk_par}
A similar approach was firstly introduced by Paradiso in \cite{Par}. As a main difference from our setting, the author considered the action of $\mathsf{GL}(\ggo)\subset \Ll$ on $\mc D$. This  choice yields  a coupled flow as follows:
\[
\begin{cases}
		 \frac{\dd}{\dd t} \mu = - \theta(\Ricb_{\mu,H}) \mu ,\qquad &\mu(0) = \mu_0 := (\mud_0)_{\ggo},  \\
		\frac{\dd}{\dd t} H =  - \rho(\Ricb_{\mu,H}) H, \qquad &H(0) = H_0 := (\mud_0)_{\Lambda^3} 
		\end{cases}
	\]
 which, in view of \Cref{lem:evdstar},  is equivalent to \eqref{eqn_muH} only in the particular case in which $H_0$ is harmonic. 
\end{remark}

The following is the main result of this Section. It shows that from a left-invariant  solution of the generalized Ricci flow one can construct a generalized bracket flow solution whose generalized geometry is gauge-equivalent to the original solution, and that the converse is also true. This in particular implies that the maximal existence times for both flows coincide, and that any generalized geometry question can be studied by means of the generalized bracket flow.

\begin{thm}\label{thm_grf=bf}
Let $\left( (\gggo)_{H_0}, \gca_t\right)_{t\in [0,T)}$ be a left-invariant solution to the generalized Ricci flow equation \eqref{eqn_GRF} and let $(\mud(t))_{t\in [0,T')}$ be the  solution to the generalized bracket flow \eqref{eqn_genbf} with background  $\left( (\gggo)_{H_0}, \gca_0\right)$ and initial condition $\mud(0)$ defined by $\mud(0)_{\ggo} = \lb_\ggo$, $\mud(0)_{\Lambda^3} = H_0$, with both solutions defined on a maximal interval of time. Then, $T = T'$ and there exists a one-parameter family of generalized isometries
\[
		F_t : \left( (\gggo)_{H_0}, \gca_t\right) \longrightarrow \left( (\gggo)_{\mud(t)}, \gca_0 \right), \qquad t\in [0,T).
\]
\end{thm}

\begin{proof}
First, given any metric ECA $((\f g \oplus \f g^*)_\mud, \mathcal{G})$ , we define the endomorphism
\[
A_{\mud,\mathcal{G}} := \operatorname{Skew}_{\mathcal{G}}(\dd^{*_g}_\mu(H + \dd_\mu b)) \in \f{o}(\f g \oplus \f g^*,\m{\mathcal{G}\cdot}{\cdot}),
\]
where $\mu = \mud_{\f g}, H := \mud_{\Lambda^3}$, and $g \in \operatorname{Sym}^2_+(\f g), b \in \Lambda^2\f g^*$ satisfy
\[
\mathcal{G} = e^b \begin{pmatrix}0&g^{-1}\\g&0\end{pmatrix}e^{-b}.
\]
One readily checks that the map $(\mud , \mathcal{G}) \mapsto A_{\mud , \mathcal G}$
is $\mathsf{L}$-equivariant: $A_{F\cdot \mud, F \cdot \mathcal{G}} = F \cdot A_{\mud , \mathcal{G}}$. Then, since the generalized Ricci curvature is similarly $\mathsf{L}$-equivariant, it follows from \Cref{eqn_A_mud} that
\[
\mc Rc(\mud,\mathcal{G}) - A_{\mud,\mathcal{G}} \in \f l,
\]
where $\mc Rc(\mud,\mathcal{G})$ is the generalized Ricci curvature of $((\f g\oplus \f g^*)_\mud,\mathcal{G})$.

Now, given a solution to the generalized bracket flow $(\mud(t))_{[0,T')}$ \eqref{eqn_genbf}, we define  the family of left-invariant endomorphisms  $(\mc G_{\mud(t)})_{t\in[0, T')}$  as follows:
\begin{equation}\label{eqn_gcamud}
 \mc G_{\mud(t)}(x)=\overline{{\rm L}_{\mud(t)}(x)}\cdot \mc G_{0}(e)\,, \quad  x\in \mathsf G,
\end{equation} where $\overline{{\rm L}_{\mud(t)}(x)}$ is the lift  on $\gggo$ of the left-translation in $\mathsf G_{\mu(t)}$ by $x$  while  $\mc G_{0}(e)$ and  $\mc G_{\mud(t)}(x)$ are, respectively, the endomorphisms on the fibers in $e$ and $x$.  We easily see that 
$
\mc G_{\mud(t)}
$ are left-invariant generalized metric on $\gggo$. 
Let us then consider
$$
\frac{\dd}{\dd t}F_t=F_t(\mc Rc_{\mud(t)}- A_{\mud(t)})\,, \quad F_0={\rm Id}\,,
$$ where $A_{\mud(t)}\in \f{so}(\gggo,\gca_0) \cap \f{so}(\gggo,\ip)$ is defined as in \eqref{eqn_A_mud}.   By a standard result in ODE theory, the family $F_t$ is defined on $[0,T')$ and $(F_t)_{t\in[0,T')}\subseteq \Ll$. We then define $ \bm{\lambda}(t):=F_t^{-1}\cdot\mud(0)$ yielding that 
\begin{equation}\label{evolutionlambda}
\frac{\dd}{\dd t}\bm{\lambda}=-\Theta(F_t^{-1}F_t')\bm{\lambda}=-\Theta(\mc Rc_{\mud(t)}- A_{\mud(t)})\bm{\lambda}\,.
\end{equation}
 Thus, $\bm{\lambda}(t)$ and $\mud(t)$ satisfy the same ODE with equal initial datum  concluding that $\mud(t)=F_t^{-1}\cdot\mud(0).$ From this,  we can deduce that 
$$
F_t\colon((\gggo)_{\mud(t)}, \mc G_{\mud(t)})\to ((\gggo)_{H_0}, F_t\cdot\mc G_{\mud(t)})\,,
$$  is an isometry. So,  denoting 
$
\tilde{\mc G}_t:=F_t\cdot\mc G_{\mud (t)}
$ we  obtain that,  in $e$, 
$$
\begin{aligned}
\tilde{\mc G}_t^{-1}\frac{\dd}{\dd t}\tilde{\mc G}_t
=&\,  F_t\mc G^{-1}_{\mud (t)}(\mc Rc_{\mud(t)}- A_{\mud(t)}) \mc G_{\mud (t)}F_t^{-1}- F_t(\mc Rc_{\mud(t)}- A_{\mud(t)})F_t^{-1}
= -2\mc Rc(\tilde{\mc G}_t)\,,
\end{aligned}
$$ where the last equality follows from  $\mc Rc(\mc G)\mc G =-\mc G\mc Rc(\mc G)$ and from $[A_{\mud(t)}, \mc G_{\mud (t)}]=0$.   This allows us to conclude  that $\mc G_t=\tilde{\mc G_t}$,  since they solve the same ODE with the same  initial condition.

Conversely, we define
$$
\frac{\dd}{\dd t}F_t=(\mc Rc(\mc G_t)- A(\mathcal G_t))F_t\,, \quad F(0)={\rm Id},
$$ where $A(\mathcal{G}_t) := A_{\mud,\mathcal G_t}$. As before,  we consider
$
\tilde{\mc G_t}:=F_t\cdot\mc G_0\
$  and easily note that
$$
\tilde{\mc G_t}^{-1}\frac{\dd}{\dd t}\tilde{\mc G_t}=\tilde{\mc G_t}^{-1}(\mc Rc(\mc G_t)- A(\mathcal G_t))\tilde{\mc G_t}- (\mc Rc(\mc G_t)- A(\mathcal G_t))\,.
$$On the other hand, 
$$
\mc G_t^{-1}\frac{\dd}{\dd t}\mc G_t=-2\mc Rc(\mc G_t)
=\mc G_t^{-1}(\mc Rc(\mc G_t)-A(\mathcal G_t)) \mc G_t- (\mc Rc(\mc G_t)-A(\mathcal G_t)),
$$  since, again,   $\mc Rc(\mc G_t)\mc G_t =-\mc G_t\mc Rc(\mc G_t)$ and   $[A(\mc G_t), \mc G_t]=0$.  So, in particular, we have that 
$$
\frac{\dd}{\dd t}\tilde{\mc G_t}=[\mc Rc(\mc G_t)- A(\mathcal G_t), \tilde{\mc G_t}]\,, \quad \frac{\dd}{\dd t}\mc G_t=[\mc Rc(\mc G_t)- A(\mathcal G_t), \mc G_t]\,,
$$ Then, $\tilde{\mc G_t}$ and $\mc G_t$ satisfy the same ODE with the same initial condition yielding that $\mc G_t=F_t\cdot\mc G_0 $. 
Therefore,  if we let  $\bm{\lambda}(t)=F_t^{-1}\cdot \mud(0)$, 
 $$
F_t\colon((\gggo)_{\bm{\lambda}(t)}, \mc G_{\bm{\lambda}(t)})\to ((\gggo)_{H_0}, \mc G_t)
$$
 is an isometry. In particular,  this implies that 
$$
F_t\mc Rc _{\bm{\lambda}(t)}F_t^{-1}=\mc Rc(\mc G_t)\,, \quad  F_tA_{\bm{\lambda}(t)}F_t^{-1}=A(\mc G_t), 
$$ using the  ${\rm Aut}(E)$-equivariance of the generalized Ricci curvature and the uniqueness of $A(\mc G)\in\f{so}(\gggo,\gca) \cap \f{so}(\gggo,\ip)$ such that $\mc Rc(\mc G)- A(\mc G)\in \mathfrak l $.  Then we have that 
$$
\frac{\dd}{\dd t}F_t=F_t(\mc Rc_{\bm{\lambda}(t)}- A_{\bm{\lambda}(t)})\,, \quad F_0={\rm Id}
$$ and repeating the computations in \eqref{evolutionlambda}, we infer that $\mud(t)=\bm{\lambda}(t)$, concluding the proof. 
\end{proof}

\begin{remark}
We do not necessarily need to use $\gca_0$ as background metric, but this simplifies the formulas.
\end{remark}

\section{Global existence on solvmanifolds} \label{sec_solv}

As a first application of \Cref{thm_grf=bf}, we prove in this Section the long-time existence of invariant solutions on any solvmanifold. Recall from \cite{GRFbook} that the \emph{generalized scalar curvature} of a metric Courant algebroid $(\TTH,\mc G(g,0))$ is given by
\begin{equation}\label{genscal}
\mc S_{g, H} := \scal_g - \tfrac{1}{12}|H|^2,
\end{equation}
where $\scal_g$ denotes the Riemannian scalar curvature of the metric $g$. Let us remark that this is not the trace of the generalized Ricci curvature (which is in fact traceless), but it is obtained instead by studying Lichnerowicz formulae  for certain Dirac operators \cite[~\S 3.9]{GRFbook}.
 Within the moving Dorfman bracket framework, a  direct consequence of  \cite[Lemma 4.2]{nilRF} allows us to infer  that on a nilpotent Lie group
\begin{equation}\label{eqn_gscmu}
\mc S_{\mud}=-\frac{1}{12}(3\lvert\mu\rvert^2+\lvert H \lvert^2)=-\frac{1}{12}\lvert \mud\rvert^2\,.
\end{equation}
Before stating the main Theorem of this Section, we will need a preliminary Lemma. 
\begin{lem}\label{lem_normmud}
Let $(\mud(t))_{t\in  (\varepsilon_-,\varepsilon_+)}$ be a solution of the generalized bracket flow defined on its maximal time interval. Then, there exists a uniform constant $C>0$ such that:
\begin{enumerate}
\item if $\varepsilon_+<\infty$, then
$$
\lvert\mud(t)\rvert\ge  \frac{C}{(\varepsilon_+-t)^{\frac12}}\,, \quad t\in[0,\varepsilon_+)\,;
$$
\item if $\varepsilon_-> -\infty$, then
$$
\lvert\mud(t)\rvert\ge \frac{C}{(t-\varepsilon_-)^{\frac12}}\,, \quad t\in(\varepsilon_-, 0]\,.
$$
\end{enumerate}
\end{lem}
\begin{proof} For the ease of notation, we will always denote with $C$ a uniform  and positive constant which may change from line to line.  
 First of all, we observe that,  
 as in \cite[Lemma 3.1]{scalar}, we have $\lvert {\rm Ric}_{\mu}\rvert^2\le C \lvert\mu\rvert^4$. Moreover, one can easily show that $\lvert H^2\rvert^2\le C \lvert H\rvert^4$. Finally,  we  obtain that $\lvert \dd_{\mu}^*H\rvert^2\le C\lvert \mu\rvert^2\lvert H\rvert^2\le C\lvert\mud\rvert^4$, using that $*_{\mu}$ is an isometry and that $\lvert\dd_{\mu}\alpha\rvert^2\le C \lvert \mu\rvert^2\lvert \alpha\rvert^2$, for any form $\alpha$ and any Lie bracket $\mu$. Then, one can deduce 
 \begin{equation}\label{eqn_estimate}
 \lvert\mc Rc_{\mud}- A_{\mud}\rvert^2=2\lvert {\rm Ric}^B_{\mu, H}\rvert^2+\frac16 \lvert \dd_{\mu}^*H\rvert^2\le C\rvert\mud\lvert^4\,.
 \end{equation}
 Now,  we fix $t_0\in[0,\varepsilon_+)$ and we can use \eqref{eqn_estimate} and the linearity of the infinitesimal $\mathsf L$-action $\Theta $ to infer that 
 $$
 \frac{\dd}{\dd t}\lvert\mud\rvert^2=2\bar g\left(\mud, \frac{\dd}{\dd t}\mud\right)\le C\lvert\mud\rvert^4\,, \quad t\in [t_0, \varepsilon_+)\,.
 $$ This implies, by comparison, that 
 \begin{equation}\label{eqn_estimates2}
 \lvert\mud\rvert^2\le \frac{1}{-C(t-t_0)+ \lvert\mud(t_0)\rvert^{-2}}\,, \quad  t\in [t_0, \varepsilon_+)\,.
 \end{equation}
  On the other hand, \eqref{eqn_estimates2} readily guarantees   that $\varepsilon_+\ge t_0+\frac1C\lvert\mud(t_0)\rvert^{-2}$, giving us the  first claim.  The second claim can be obtained  by reversing the time variable.
\end{proof}
We can now state  the following blow-up result:

\begin{thm}\label{thm_blowup} Let $(E, \gca_t)_{t \in (\varepsilon_-,\varepsilon_+)}$ be a left-invariant solution of the generalized Ricci flow over a simply-connected Lie group $\mathsf G$ defined in its maximal time interval. Then,  
\begin{enumerate}
\item if $\varepsilon_+<\infty$, then $\mc S_t\to \infty$, as $t\to\varepsilon_+$\,;
\item if $\varepsilon_->-\infty$, then $\mc S_t\to -\infty$, as $t\to\varepsilon_-$\,.
\end{enumerate}
\end{thm}
\begin{proof}
Fixing the  preferred isotropic splitting $\sigma_0$ associate with $\gca _0$, we have the time-varying isometry $ (E,\mc G_t)_{t \in (\varepsilon_-,\varepsilon_+)} \simeq_{\sigma} ((\f g \oplus \f g^*)_{H_t},\mc G(g_t,0))_{t \in (\varepsilon_-,\varepsilon_+)} $. By \cite[~Proposition 1.1]{GRFEntropyStreets}, if $\phi\colon \G \times (\varepsilon_-,\varepsilon_+) \to \R$ satisfies
\begin{equation}\label{eqn_phi}
\left(\frac{\partial}{\partial t} - \Delta_{g_t}\right)\phi_t = \tfrac 16 |H_t|^2_{g_t};\qquad t\in (\varepsilon_-,\varepsilon_+),
\end{equation}
then, 
\[
\left(\frac{\partial}{\partial t} - \Delta_{g_t}\right)\left(\mc S_{g_t, H_t} + 2\Delta_{g_t}\phi_t - |\nabla_{g_t}\phi_t|^2\right) = 2 |\Ricb_{g_t, H_t} + \nabla^2_{g_t}\phi_t - \tfrac 12 (\dd^*_{g_t}H_t + \iota_{\nabla_{g_t} \phi_t}H_t)|^2;\qquad t \in  (\varepsilon_-,\varepsilon_+).
\]
In this case, we set $\phi_t := \int_{0}^t \tfrac 16 |H_s|^2\dd s$, which is constant in space and hence satisfies \Cref{eqn_phi}. Thus, by left-invariance, we have
\begin{equation}\label{eqn_scalEvo}
\frac{\dd}{\dd t} \mc S_{g_t, H_t} = 2 |\Ricb_{g_t,H_t} - \tfrac 12 \dd^*_{g_t}H_t|^2 = |\mc Rc(\mc G_t)|^2.
\end{equation}
The remainder of the proof follows \cite{scalar} for the classical Ricci flow. By \Cref{thm_grf=bf}, $((\f g \oplus \f g^*)_{H_0},\mc G_t)_{t \in (\varepsilon_-,\varepsilon_+)}$ is isometric to $\left( (\gggo)_{\mud(t)}, \gca(g_0,0) \right)_{t\in (\varepsilon_-,\varepsilon_+)}$, where $\mud$ satisfies \Cref{eqn_genbf}. Let $\mc S_{\mud(t)} = \mc S_{g_t}$ denote the generalized scalar curvature of this latter metric ECA and note that by \Cref{eqn_scalEvo} it satisfies $\tfrac{\dd}{\dd t}\mc S_{\mud(t)} = |\mc Rc_{\mud(t)}|^2$. Also observe that
\[
\tfrac{\dd}{\dd t}|\mud|^2 = - 2 \m{\Theta(\mc Rc_{\mud} - A_{\mud})\mud}{\mud} \leq C |\mc Rc_{\mud} - A_{\mud}||\mud|^2,
\]
for some constant $C > 0$. Hence, observing that $|\mc Rc_\mud - A_\mud|^2 \leq 2 |\mc Rc_\mud|^2$, we have,  for $t \in (\varepsilon_-,\varepsilon_+)$,
\begin{align*}
\log |\mud(t)|^2 - \log |\mud(0)|^2 &= \int_0^t \tfrac {\dd}{\dd s} \log |\mud(s)|^2 \dd s \leq C \int_0^t |\mc Rc_{\mud(s)} - A_{\mud(s)}|\dd s \leq C \int_0^t |\mc Rc_{\mud(s)}| \dd s\\
&\leq C \int_0^t 1 + |\mc Rc_{\mud(s)} - A_\mud(s)|^2 \dd s = C\left(t + \mc S_{\mud(t)} - \mc S_{\mud(0)}\right).
\end{align*}
On the other hand, by \Cref{lem_normmud}, we have that $|\mud(t)|^2 \to \infty$,  as $t \to \varepsilon_+$,  forcing $\mc S_{\mud(t)} \to \infty$,  as $t \to \varepsilon_+$. A similar argument shows the statement for $t \to \varepsilon_-$.
\end{proof}

Recall that left-invariant metrics on solvable Lie groups have non-positive scalar curvature, and zero scalar curvature if and only if they are flat \cite{BB}. Hence, we yield the immediate Corollary:

\begin{cor} \label{cor:SolvInfTime}
Any left-invariant  solution  of the generalized Ricci flow on a solvable Lie group exists for all positive times.
\end{cor}

In the case of a left-invariant solution on a nilpotent group, we are able to say much more thanks to \Cref{cor_mmnil}. In particular, we can describe the precise asymptotic behaviour of the generalized scalar curvature:
\begin{thm} For any left-invariant  solution $(g_t,H_t)$ of the generalized Ricci flow on a nilpotent Lie group $\G$, the generalized scalar curvature satisfies $\mathcal{S}_{g_t,H_t} \sim -\frac 1t$, as $t \to \infty$.
\end{thm}
\begin{proof} Let $(g_t,H_t)$ be a solution of the generalized Ricci flow  on $\G$ and $\mud(t)$ the corresponding generalized bracket flow. Then by \Cref{eqn_genbf}, \Cref{cor_mmnil}  and \Cref{def_momMap}, we obtain
\[
\frac{\dd}{\dd t}|\mud(t)|^2 = -2\m{\Theta(\mm_{\mud(t)})\mud(t)}{\mud(t)} = -6\B{g}_{\f l}(\mm_{\mud(t)},\mm_{\mud(t)}) = -6|M_{\mud(t)}|^2.
\]
We now claim that $|\mm_{\mud(t)}|^2 \geq c|\mud|^4$, for some small constant $c > 0$. To see this, first note that the map
\[
\mathbb{S}(\Lambda^2(\f g \oplus \f g^*) \otimes (\f g \oplus \f g^*)) \ni \mud \mapsto |\mm_\mud|^2,
\]
is never zero on the unit sphere. Indeed $\mm_{\mud} = 0 $ implies that $0 = \tr(\mm_{\mud}) = -\frac{1}{6}|\mud|^2 = -\frac 16$, which is a contradiction. The claim therefore follows from compactness of the unit sphere. Thus,
\[
-C|\mud(t)|^4 \leq \frac{\dd}{\dd t}|\mud(t)|^2 \leq -c |\mud(t)|^4,
\]
so $|\mud(t)|^2 \sim \frac 1t$ by ODE comparison. The result now follows from \Cref{eqn_gscmu}.

\end{proof}

\section{Generalized nilsolitons}\label{sec_gensol}

	In this Section,  following the approach by Lauret in \cite{homRF} and \cite{soliton}, we give the definition of algebraic generalized solitons, proving their main  properties  in the nilpotent case. 	
	 First of all, we  introduce the normalized generalized bracket flow. 
	 \begin{defn}
	  The $\ell$-\emph{normalized generalized bracket flow} with background data $((\gggo)_H, \overline {\mc G} )$ is the  initial value problem for a curve of brackets $\mud(t)$ and $\ell=\ell(t)$:
	\begin{equation}\label{eqn_ngbf}
	  \frac{\dd}{\dd t}\mud=-\Theta({\mc Rc}_\mud - A_\mud) \mud+ \ell\mud\,, \qquad \mud(0)=\mud_0 \,.
	  \end{equation}
	 \end{defn}
	 The multiplication of $\mud$ in \eqref{eqn_ngbf} by the factor $\ell$ has to be interpreted using \Cref{lem_scaling} and \Cref{prop:scaledmu}.  On the other hand, given a function $\ell$,  one can  obtain a solution of the $\ell$-normalized generalized bracket flow from a solution of the generalized bracket flow \eqref{eqn_genbf}  via a time reparametrization and a scaling. 
	 	   \begin{lem}\label{lem_norm}
	    Let  $((\gggo)_H, \bG)$ be a left-invariant metric ECA over a Lie group $\mathsf G$. Let  $\mud^{\ell }(t)$ and $\mud(t)$, respectively, be  the solution of the $\ell$-normalized generalized bracket flow and the solution of the generalized bracket flow  starting from $\mud_0$.  Then, there exist  a family of scaling $c(t)>0$ such that $c(0)=1$ and a time reparametrization $\tau=\tau(t)$ so that 
	    $$
	    \mud^{\ell}(t)=c(t)\mud(\tau)\,.
	    $$
	   \end{lem}
	\begin{proof} Following  \cite{homRF}, we choose 
	$$
	c'=\ell c\,, \quad c(0)=1\quad \mbox{and} \quad\tau'=c^2\,, \quad \tau(0)=0.
	$$  Then,  we have  
	$$
	\frac{\dd}{\dd t}(c(t)\mud(\tau))=-c^3\Theta({\mc Rc}_{\mud(\tau)} - A_{\mud(\tau)}) \mud(\tau)+\ell c\mud(\tau).
	$$  
	Using that ${\mc Rc}_{c\mud }- A_{c\mud}=c^2({\mc Rc}_\mud - A_\mud)$, we obtain the claim. 
	\end{proof}
	
	   
	 The case in which we will be interested the most is when  we choose $c=\frac{1}{\lvert \mud\rvert}.$ Applying  \Cref{lem_norm}, we have that $\bar \mud:=\frac{\mud}{\lvert \mud\rvert}$ is a solution of the $\ell$-normalized generalized bracket flow with 
	 $$
	 \ell_{\bar \mud}=\frac{\bar g(\Theta({\mc Rc}_{\mud} - A_{\mud}) \mud, \mud)}{\lvert \mud\rvert^4}
	 =6\bar g(\mc Rc_{\bar \mud} - A_{\bar \mud},\mm_{\bar \mud} )\,.
	 $$  With this choice of the normalization, we will refer to  the corresponding normalized generalized bracket flow as \emph{scalar-normalized generalized bracket flow}.  In this case,  the norm of the solution of the scalar-normalized generalized bracket flow will remain constantly $1$, provided the initial datum has unit norm. Indeed, 
	 $$
	 \frac{\dd}{\dd t}\lvert\mud\rvert^2= 2\bar g(\Theta({\mc Rc}_{\mud} - A_{\mud}) \mud, \mud)(-1+\lvert\mud \rvert^2)\,,
	 $$ then $\lvert\mud \rvert^2=1$, since $\lvert \mud_0\rvert^2=1$.  Recalling \eqref{eqn_gscmu}, this  also implies that $\mc S_{\mud}=\mc S_{\mud_0}=-\frac{1}{12}$ along the scalar-normalized generalized bracket flow.
	 
	 With this in mind, we can give the definition of generalized algebraic soliton.
	 \begin{defn}\label{def_algsol}
	 Let $\mud\in \mc D$. We say that $\mud$ is an algebraic soliton for the generalized Ricci flow if it is  a fixed point of the scalar-normalized generalized bracket flow,  i.e.  the following is satisfied
	 \begin{equation}\label{eqn_algsoli}
	 \Theta({\mc Rc}_\mud - A_\mud) \mud= \ell_{\mud}\mud\,.
	 \end{equation}
	 We say that $\mud$ is expanding, steady or shrinking if, respectively, $\ell_{\mud}>0$, $\ell_{\mud}=0$ or $\ell_{\mud}<0$. 
	 \end{defn}
	  Building from this definition,  we can derive equivalent conditions for a Dorfman bracket to be an algebraic soliton for the generalized Ricci flow. 
	 \begin{prop}\label{prop_equisol}
	 Let $\mud\in\mc D$. Then, the following are equivalent:
	 \begin{enumerate}
	 \item\label{item_algsol}$\mud $ is an algebraic soliton;
	 \item\label{GBFscaling} the generalized bracket flow starting at $\mud$   evolves only by scaling;
	 \item \label{staticas} ${\mc Rc}_\mud - A_\mud=\lambda {\rm Id}+ \bm{D}$ with $\bm{D}\in {\rm Der}_{\lambda}(\mud)$ and $\lambda\in \mathbb R $;
	 \item \label{staticclassicas}there exist  $D\in {\rm Der}(\mu)$ and $\lambda\in \R$ such that 
	 \begin{equation}\label{eqns_soli}
	 \begin{cases}
	 {\rm Ric}_{\mu, H}^{B}=\lambda{\rm Id}+ D\,,  \\
	 \Delta_{\mu}H=\lambda H+\rho({\rm Ric}_{\mu, H}^{B})H\,.  
	 \end{cases}
	 \end{equation}
	 \end{enumerate}
	  \end{prop}



	  \begin{proof}
	  The equivalence between \Cref{item_algsol} and \Cref{GBFscaling} is trivial using \eqref{eqn_algsoli} and \Cref{lem_norm}.  On the other hand, if $\mud$ is an algebraic soliton,   the fact that $\Theta({\rm Id})\mud=-\mud$, for all $\mud\in \mathcal D$,  implies ${\mc Rc}_\mud - A_\mud=\lambda {\rm Id}+ \bm{D}$ with  $\bm{D}\in{\rm Der}(\mud)$. Moreover,  ${\mc Rc}_\mud - A_\mud\in \mathfrak{so}(\gggo, \ip)$, then,  $\bm{D}$ has to satisfy
	  $$
	  2\lambda\ip+ \la {\bm D}\cdot, \cdot\ra+ \la\cdot,  {\bm D}\cdot\ra=0\,,
	  $$ which gives \Cref{staticas}.
	  Viceversa, it is easy to show that if \Cref{staticas} holds, then \eqref{eqn_algsoli} is satisfied with $\ell_{\mud}=-\lambda$. 
	     Furthermore, \Cref{cor_ThetaL} guarantees the equivalence between  \eqref{eqn_algsoli} and 
	  $$
	  \begin{cases}
	  \theta({\rm Ric}^B_{\mu, H}+\ell_{\mud}{\rm Id})\mu=0\,, \\
	  \rho({\rm Ric}_{\mu, H}^B)H-\Delta_{\mu }H=\ell_{\mud} H \, 
	  \end{cases}
	  $$ which is equivalent to \eqref{eqns_soli} with $\lambda=-\ell_{\mud}$, concluding the proof. 
	  \end{proof}
	  	  
	 \Cref{staticas}  and \Cref{staticclassicas} in \Cref{prop_equisol} can be considered as the static definition,  in terms of  both generalized  and classical  objects, of  algebraic solitons   and so \Cref{prop_equisol} has to be regarded as the analogue of \Cref{prop_static_soliton} in the varying  Dorfman bracket framework.

Moreover, \Cref{prop_equisol} allows us to construct  from algebraic solitons for the generalized Ricci flow  generalized metrics which are  generalized Ricci  solitons as in \Cref{defn_soliton}.
\begin{lem}\label{lem_algrealsol}
Let $\mud$ be an algebraic soliton for the generalized Ricci flow. Then, $\mc G_{\mud}$, defined as in \eqref{eqn_gcamud},  is a generalized Ricci soliton. 
\end{lem}
\begin{proof}
	  The first equation in \eqref{eqns_soli} can be equivalently written, in terms of symmetric $(0,2)$-tensors, as  
	  $$
	  {\rm Ric}_{g_{\mu}, H}^B=\lambda g_{\mu}+\frac12(g_{\mu}(D\cdot, \cdot)+ g_{\mu}(\cdot, D\cdot)=\lambda g_{\mu}-\frac12\mc L_{X_{D}}g_{\mu}\,,
	  $$ where $X_D=\frac{\dd}{\dd t}\big|_{t=0} \varphi_t$  with $\varphi_t \in {\rm Aut}(\mathsf G_{\mu})$ such that $\dd_e\varphi_t=e^{tD}$. Moreover,  plugging   the first equation of \eqref{eqns_soli} into the second one, we have 
	  $$
	  \Delta_\mu H=\lambda H+\rho(\lambda  {\rm Id}+ D)H=-2\lambda H +\rho(D)H =-2\lambda H +\mc L_{X_D}H\,.
	  $$ The claim follows  using \Cref{prop_static_soliton}.
	  \end{proof}
	   Before  going into the discussion of algebraic solitons in the nilpotent case, we will need  the following preliminary Lemma about $\lambda$-derivations, which is nothing but an adaptation of \Cref{lem_der} to the invariant case. 
	 \begin{lem}\label{lem_invlambdader}
	Let  $(\gggo)_{\mud}$ be a   left-invariant  ECA  over a Lie group $\mathsf G $. Then,  we have that, for any $c\in\mathbb R\backslash\{0\}$ and $\lambda\in \R$, 
	 $$
	 {\rm Aut}(c\cdot(\gggo)_{\mud}, (\gggo)_{\mud})=\left\{ \bar A_ce^b  :  b\in \Lambda^2\ggo^*\,,\quad A\in {\rm Aut}(\mu)\,,\quad A^{-1}\cdot H=c( H-\dd b )\right\}
	 $$
and 	 $$
	 {\rm Der}_{\lambda}(\mud)=\left\{\begin{pmatrix} D& 0 \\ \alpha & -2\lambda{\rm Id}-D^*\end{pmatrix} : D\in {\rm Der}(\mu), \quad \rho(D)H-\dd_{\mu}\alpha -2 \lambda H=0\right\}\,.
	 $$
	 \end{lem}
	 \begin{proof} The first assertion follows directly from the discussion in the non-invariant case.  \Cref{defn_lambdader} guarantees that  a $\lambda$-derivation  has to be of the form 
	 $$
	\bm{D}= \begin{pmatrix} D& 0 \\ \alpha & -2\lambda{\rm Id}-D^*\end{pmatrix}\,,
	 $$ where $D\in{\rm Der}(\mu)$ and $\rho(D)H-\dd \alpha -2\lambda H =0 $\,. 
	 \end{proof}

	The moment map formulation in the nilpotent case   allows us to derive many other properties concerning algebraic solitons, which, in this particular case, will be called generalized nilsolitons, mimicking the classical nomenclature, see \cite{soliton}.
\begin{prop}
Let $\mud\in \mc D$ be a generalized nilsoliton.  Then, it is expanding.  Moreover, the derivation $D$ satisfies the following property:
$$
{\rm tr}\left(\left(D+\frac14H^2\right)D'\right)=-\lambda {\rm tr}(D')\,,$$  for all $D'\in {\rm Der}( \mu)$\, defining $\bm{D}'=(D', \alpha)\in {\rm Der}_{\lambda'}(\mud)$, for some $\lambda'\in \R$. Finally,  $\bar g(\iota_XH, \dd_{\mu}^*H) = 0$,  for all $X \in \f g$.
\end{prop}
\begin{proof}
   We have that
   $$
   \ell_{\mud}=6\bar g(\mc Rc_{\mud} - A_{\mud},\mm_{\mud} )=6\lvert \mc Rc_{\mud} - A_{\mud}\rvert^2\ge 0\,, $$
   where the last equality follows from \Cref{cor_mmnil}. On the other hand, 
   $
   \lvert \mc Rc_{\mud} - A_{\mud}\rvert^2= 0
   $ would, in  particular,  imply that
   $$
   {\rm Ric}_{\mu}=\frac14 H^2\ge 0 \,,
   $$  which cannot happen in the nilpotent case, see \cite[Lemma 4.2, (iii)]{nilRF}, giving us the first  claim. 
   
   As regards the second,  fixed  $\bm{D}'=(D', \alpha)\in {\rm Der}_{\lambda'}(\mud)$, we have that 
   \begin{equation}\label{eqn_mmlambda'der}
   \bar g_{\mathfrak l }(\lambda {\rm Id}+\bm{D}, \lambda'{\rm Id}+\bm{D}')=\frac16\bar g(\Theta(\lambda'{\rm Id})\mud, \mud)
   =-\frac{\lambda'}{6}\lvert\mud\rvert^2=\lambda'\left(2{\rm R}_{\mu}-\frac{1}{6}\lvert H\rvert^2\right)\,. 
   \end{equation}  On the other hand,  we can use \Cref{prop_equisol} to deduce  that $\bm{D}=(D, -\dd_{\mu}^*H)$ and then obtain 
   $$
  \bar g_{\mathfrak l }(\lambda {\rm Id}+\bm{D}, \lambda'{\rm Id}+\bm{D}')=2{\rm tr}((\lambda{\rm Id}+D)(\lambda'{\rm Id}+D'))- \frac16\bar g(\dd_{\mu}^*H, \alpha)\,.
   $$ The fact that $\bm{D}'\in {\rm Der}_{\lambda'}(\mud)$ implies, thanks to \Cref{lem_invlambdader}, that 
     $\dd_{\mu}\alpha=\rho(D')H-2\lambda'H $. This can be used to infer that
       \begin{equation}\label{eqn_glambda'der}
   \begin{aligned}
   \bar g_{\mathfrak  l }(\lambda {\rm Id}+\bm{D}, \lambda'{\rm Id}+ \bm{D}')=&\,  2\lambda'{\rm tr}({\rm Ric}_{\mu, H}^B)+ 2\lambda{\rm tr}(D')+ 2{\rm tr}(DD')-\frac16\bar g(\alpha, \dd_{\mu}^*H)\\
   =&\, 2\lambda'\left({\rm R_{\mu}}-\frac14 \lvert H\rvert^2\right)+ 2\lambda{\rm tr}(D')+ 2{\rm tr}(DD')+\frac12{\rm tr}(H^2D')+ \frac13\lambda'\lvert H\rvert^2\\
   =&\, \lambda'\left(2{\rm R}_{\mu}-\frac16 \lvert H\rvert^2\right)+ 2\lambda{\rm tr}(D')+ 2{\rm tr}(DD')+\frac12{\rm tr}(H^2D')\,.
   \end{aligned}
   \end{equation} Now, it is sufficient to use \eqref{eqn_glambda'der} in \eqref{eqn_mmlambda'der} to obtain the claim.
      Finally, recall that for metric nilpotent Lie algebras, any symmetric derivation $D$ satisfies $\tr(\ad_XD) = 0$, for all $X \in \f g$. Then, the soliton equation in $\mu$  implies
   \[
   		0 = \operatorname{tr}(\ad_XH^2) = -\frac{1}{3}\bar g (\rho(\ad_X)H, H) = -\frac 13 \bar g(\mathcal{L}_XH, H) = -\frac 13 \bar g (\dd_{\mu} \iota_XH, H) = -\frac 13 \bar g (\iota_XH, \dd_{\mu}^*H),
   \]
   as claimed.
\end{proof}

Motivated by the wide range of results concerning  classical nilsolitons proved throughout the years, see for instance \cite{soliton}, many questions about generalized nilsolitons can be raised.  First of all, as usual,  let $(E, \bG)$ be a left-invariant  metric ECA over a nilpotent Lie group $\mathsf G$ and assume that $\bG$ is  a generalized nilsoliton.  Using the preferred isotropic splitting induced by $\bG$, we can identify $(E,\bG)\simeq((\gggo)_H, \mc G(\bar g, 0 ))$, for some  closed $H\in \Lambda^3\ggo^*$.  So, in this setting, one can wonder if $\bG$ is unique, up to scaling  and up to the action of ${\rm Aut}([\cdot, \cdot])$, among left-invariant generalized metrics defining a preferred representative $H'\in [H]\in H^3(\ggo, \R)$.
\begin{remark}
Considering $[H]\in H^3(\mathsf G, \R )$, the uniqueness is not true in general.  Indeed,  on $\mathfrak n_3={\rm Lie}({\rm Heis}(3, \R))$, the following 
$$
   \mu(e_1, e_2)= e_3\,, \quad H=e^{123}\,, \quad \lambda=-2\,,\quad D={\rm diag}(1,1,2)\,
$$ defines a generalized nilsoliton, see \Cref{heisnils} for the proof,  with $[H]=0\in H^3({\rm Heis}(3, \R), \R)$. One can easily observe that the above generalized nilsoliton defines exactly the expanding soliton for the generalized Ricci flow described in \Cref{ex_solitons}.  On the other hand, $\mu(e_1, e_2)= e_3$ defines  also the classical nilsoliton, see \cite[Theorem 4.2]{finding}, i.e. $ H=0$.
\end{remark}

In order to address the uniqueness question, we need to understand better the geometric properties of the space of left-invariant generalized metrics.

   As we saw in \Cref{prop_L_trans_mcaG}, 
    fixed $E$ a left-invariant ECA over a Lie group $\mathsf G$, not necessarily nilpotent,  the space of left-invariant generalized metric $\mc M ^{\mathsf G}$ can be presented as a homogeneous space where $\mathsf L $ is acting transitively by conjugation. 
   
   Of course, \Cref{prop_L_trans_mcaG} gives also   that the action of $\mathsf{SO}(\langle\cdot,\cdot \rangle)$ is transitive on $\mathcal M^{\G}$. This allows us  to  present  $\mc M^{\mathsf G}$ as a homogeneous space in a different manner. 
 Clearly, the isotropy  of $\mathsf{SO}(\langle\cdot,\cdot \rangle)$ in $\bG\in \mc M^{\mathsf G}$ is 
    $$
    \bG_{\mathsf{SO}(\langle\cdot,\cdot \rangle)}=\left\{F\in \mathsf{SO}(\langle\cdot,\cdot \rangle) :  F\bG F^{-1}=\bG\right\}=:{\rm Isom}(\bG, \mathsf{SO})\,.$$
     On the other hand, we observe that
     $$
     \theta\colon \mathsf{SO}(\langle\cdot,\cdot \rangle)\to\mathsf{SO}(\langle\cdot,\cdot \rangle)\,, \quad \sigma(F)=\bG F\bG\,,
     $$ is an automorphism of $\mathsf{SO}(\langle\cdot,\cdot \rangle)$ which is  clearly involutive, 
and the set of fixed points of $\theta$  precisely coincides with  
     $
     {\rm Isom}(\bG, \mathsf{SO}).
     $ 
     Thus, we obtain that $(\mathsf{SO}(\langle\cdot,\cdot \rangle), {\rm Isom}(\bG, \mathsf{ SO}))$ is a Riemannian  symmetric pair, as defined in \cite[IV 3.4]{Helgason01}. As a consequence,  we obtain a reductive  decomposition of $\mathfrak {so}(\langle\cdot, \cdot\rangle)=\mathfrak h \oplus\mathfrak m $  such that $[\mathfrak m , \mathfrak m ]\subset \mathfrak h $,  where, using $(E, \bG)\simeq ((\gggo)_H, \mc G(\bar g, 0))$, 
    $$
    \mathfrak h:={\rm Lie}({\rm Isom}(\bG, \mathsf{SO}))=\f{so}(\gggo, \bG)=    \left\{\begin{pmatrix} A & \bar g^{-1}\alpha \bar g^{-1}\\
    \alpha& -A^{*}\end{pmatrix} : A\in \mathfrak {so}(n)\,,\quad \alpha\in \Lambda^2\mathfrak g^*\right\}
    $$
     and
     $$\mathfrak m:=\ker(\dd\theta + {\rm Id}) =\left\{\begin{pmatrix} A & -\bar g^{-1}\alpha \bar g^{-1}\\
    \alpha& -A^{*}\end{pmatrix}: A\in {\rm Sym}(n)\,,\quad  \alpha\in \Lambda^2\mathfrak g^*\right\}\,.
         $$As usual,  we will be  identifying $\mathfrak m $ with $T_{\bG}\mathcal M ^{\mathsf G}$. Now, using again the  identification $(E, \bG)\simeq((\gggo)_H, \mc G(\bar g , 0))$,  we can define   a non degenerate bilinear form $\bar g$ on $\mathfrak{ so}(\ip) $  such that, for all $(A, \alpha), (A', \alpha')\in \mathfrak m $ and $(B, \beta), (B', \beta')\in \mathfrak h $, 
         $$
         \bar g((A, \alpha), (A', \alpha'))=2{\rm tr}(AA')+ 2\bar g(\alpha, \alpha')\,, \quad \bar g((B, \beta), (B', \beta'))=-\frac34({\rm tr}(BB')+ \bar g(\beta, \beta'))\,.
         $$ It is not hard to see that $\bar g $ is an  Ad$({\rm Isom}(\bG, \mathsf{SO}))$-invariant  bilinear form which,  when restricted to $\mathfrak l$,  coincides with the inner product $\bar g_{\mathfrak l}$ defined in \eqref{iponmfl}.  Moreover, using \cite[Theorem 3.5]{KobNom96ii}, the $\mathsf{SO}(\langle\cdot, \cdot \rangle)$-invariant Riemannian metric induced by it  on $\mathcal M^{\mathsf G}$ will be naturally reductive. This readily  guarantees  that all the geodesics emanating from $\bG$ are of the form $\exp(tX)\cdot\bG $, for some $X\in \mathfrak  m$, see \cite[Theorem 3.3]{KobNom96ii}. Finally,  \cite[Theorem 3.5]{KobNom96ii} guarantees that $(\mathcal M^{\G}, \bar g)$ has all non-positive  sectional curvatures.
        
Keeping in mind this and  building from \cite{Heb}, we can define the generalized scalar functional
$$
      \mathcal S\colon \mathcal M^{\G}\to \mathbb R\,, 
      $$ which associates to   a given  left-invariant generalized metric its  generalized scalar curvature as in \eqref{genscal}. We will study some analytic properties of this functional. 
      \begin{prop}\label{genscalfunc} Let  $E$ be  a left-invariant ECA over a nilpotent Lie group  $\mathsf G $. Then, we have:\begin{enumerate}
      \item  the gradient of $\mc S$ at any $\bG \in \mc M^{\mathsf G}$ is given by 
      $$
      \nabla_{\bG}\mc S=\tilde{\mm }(\bG):=\begin{pmatrix}{\rm Ric}^B_{\bar g, H} & \frac16\bar g^{-1}\dd_{\bar g} ^*H\bar g^{-1}\\
        -\frac16\dd_{\bar g}^{*}H & -({\rm Ric}_{\bar g, H}^{B})^*
        \end{pmatrix}
      $$ after using $(E, \bG)\simeq_{\sigma}((\gggo)_H, \mc G(\bar g, 0 ))$;
      \item $\mc S$ is concave along  geodesics, i.e. if $\gamma $ is a geodesic emanating from $\bG \in \mc M^{\mathsf G}$, then $(\mc S\circ \gamma)^{''}(0 )\le 0$;
      \item $ (\mc S\circ \gamma)^{''}(0 )= 0$ if and only if $\gamma(t)=\exp(tX)\cdot\bG$ where $X\in {\rm Der}_0([\cdot, \cdot])\cap \mathfrak m $. In this case, $\mc S\circ \gamma$ is constant. 
      \end{enumerate}
            \end{prop}
\begin{proof} We fix $\bG \in \mc M^{\mathsf G}$ and,  as usual, we identify $(E,\bG)\simeq_{\sigma }((\gggo)_H, \mc G(\bar g, 0))$.  We will make use of the moving Dorfman brackets approach. First of all, we observe that, if $\alpha \in \Lambda^2\mathfrak g^*$ and $\mud\in \mathcal D$, we have that 
         $$
         2\left\langle\bG\Theta\begin{pmatrix}0 & -\bar g^{-1}\alpha \bar g^{-1}\\
         0 & 0 \end{pmatrix}\mud(e_i, e_j), \mud(e_i, e_j)\right\rangle=-\mud_{ijk}\mud_{ij}^s\alpha_{ks}\,,
         $$ while 
         $$
         4\left\langle\bG\Theta\begin{pmatrix}0 & -\bar g^{-1}\alpha \bar g^{-1}\\
         0 & 0 \end{pmatrix}\mud(e_i, e^j), \mud(e_i, e^j)\right\rangle=-2\mud_{is}^j\mud_{iks}\alpha_{jk}
         $$ which gives us that 
         \begin{equation}\label{eqn_premm}
        \bar g\left(\Theta\begin{pmatrix}0 & -\bar g^{-1}\alpha \bar g^{-1}\\
         0 & 0 \end{pmatrix}\mud, \mud\right)=3\mud_{ijk}\mud_{ij}^s\alpha_{sk}=-\bar g(\dd_{\mu}\alpha, H)\,.
         \end{equation} Then, combining  \eqref{eqn_premm} with  \eqref{eqn_mommap}, we obtain that, for $X=(A, \alpha)\in \mathfrak m $, 
         $$
         \begin{aligned}
         \bar g(\Theta(X)\mud, \mud)=&\, \bar g\left(\Theta\begin{pmatrix}A & 0\\
         \alpha & -A^* \end{pmatrix}\mud, \mud\right)+\bar g\left(\Theta\begin{pmatrix}0 & -\bar g^{-1}\alpha \bar g^{-1}\\
         0 & 0 \end{pmatrix}\mud, \mud\right)
         = 12{\rm tr}({\rm Ric}_{\mu, H}^BA)-2\bar g(\dd_{\mu}\alpha, H)\,.
         \end{aligned}
         $$ On the other hand, considering $$
        \tilde{\mm }_{\mud}=\begin{pmatrix}{\rm Ric}^B_{\mu, H} & \frac16\bar g^{-1}\dd^*_{\mu}H\bar g^{-1}\\
        -\frac16\dd^{*}_{\mu}H & -({\rm Ric}_{\mu, H}^{B})^*
        \end{pmatrix}
         $$
          we then obtain  that 
         $$
         \bar g(\tilde{\mm }_{\mud}, X)=2{\rm tr}({\rm Ric}_{\mu, H}^BA)-\frac13\bar g(\alpha,\dd_{\mu}^{*}H)=\frac16\bar g(\Theta(X)\mud, \mud)=\frac{1}{12}\frac{\dd}{\dd t}\Big|_{t=0}\lvert\exp(tX)\cdot\mud \rvert^2\,.
         $$
         
   We now compute the gradient of $\mathcal S$ at $\bG$. If $\mud\in \mathcal D$ is the given Dorfman bracket on $E$,  we know that, for any $t$, the generalized metric $\exp(tX)\cdot \bG$ is isometric   to $\mathcal G_{\exp(-tX)\cdot \mud}$. So,   given $X\in \mathfrak{m}$, we have that 
      $$
      \begin{aligned}
      \dd_{\bG}\mathcal S(X) =&\, \frac{\dd}{\dd t}\Big|_{t=0}\mathcal S (\exp(tX)\cdot \bG)=-\frac{1}{12}\frac{\dd }{\dd t}\Big|_{t=0}\lvert\exp(-tX)\cdot\mud \rvert^2=\frac16\bar g(\Theta(X)\mud, \mud)
      = \bar g(\tilde{\mm }_{\mud}, X)=\bar g(\tilde {\mm }(\bG), X)\,,
      \end{aligned}
           $$   giving us the first claim. 
          
           Let now  $\gamma(t)=\exp(tX)\cdot\bG$, for some  $X\in \mathfrak m $,  be a geodesic emanating from $\bG$,   we have          
            \begin{equation}\label{concavity}
          (\mathcal S\circ \gamma)^{''}(0)= -\frac{1}{12}\frac{\dd^2}{\dd t^2}\Big|_{t=0}\lvert\exp(-tX)\cdot\mud \rvert^2
          = -\frac16(\bar g(\Theta(X)\Theta(X)\mud, \mud)+ \lvert \Theta (X)\mud\rvert^2)=-\frac13\lvert\Theta(X)\mud\rvert^2\le 0 \,,
          \end{equation}giving us the second claim. 
          
          Finally,  clearly, if $X\in{\rm Der}_0(\mud)\cap \mathfrak m $, $\exp(tX)\in {\rm Aut}(\mud)$ which guarantees that $\mathcal S\circ \gamma$ is constant for all $t$, the viceversa holds too thanks to \eqref{concavity}.
\end{proof}
\begin{remark}\label{rmrkmmso}
The element $\tilde \mm_{\mud}\in \mathfrak m $ can also be viewed as the moment map of the $\mathsf{ SO}(\ip)$-action on the space of Dorfman brackets $\mc D$.  Indeed, it precisely satisfies the relation:

\[
		\bar g(\tilde{\mm}_\mud, X) =  \frac 16  \, \bar g(\Theta(X)\mud, \mud)\, , \qquad X\in \mathfrak m\, .
\] Unfortunately, $\tilde{ \mm} _{\mud}$ coincides with the generalized Ricci curvature only on the subspace of Dorfman bracket with harmonic torsion.
\end{remark}
\Cref{genscalfunc} can be considered as the analogue, in the generalized setting,  of \cite[Lemma 3.3]{Heb} specialized in the nilpotent case.  So, following the steps in  \cite[Theorem 3.5]{soliton},   we can give a partial answer to the uniqueness question, proving that generalized nilsolitons with harmonic torsion are unique.
\begin{thm}\label{thm_unique}
 Let $E$ be a left-invariant ECA over a nilpotent Lie group $\mathsf G$ and  let $\bG$ be a generalized nilsoliton with harmonic torsion. Then, up to scaling and up to the action of ${\rm Aut}(E)$, $\bG$ is unique.
\end{thm}
\begin{proof}
Let us now consider the following closed Lie subgroup of $\mathsf{ SO}(\ip)$:
            $$
            \mathsf K:=\pm\prod_{c\in \R^+} c^{-\frac12}{\rm Aut}(c\cdot E, E)\,.
            $$
 First of all, we  need some clarifications  about $  \mathsf K$. The action of an element in $\mathsf K $ on $\bG\in \mc M^{\mathsf G}$ is defined as follows: an element $ c^{-\frac12}F\in   \mathsf K$ acts on $\bG $ as the following composition:
            \begin{equation}\label{actionK}
             c^{-\frac12}F\colon (E, \ip,  c^{-\frac12}\mud,\bG)\to (c\cdot E, \overline{\mc G})\to (E, F\cdot \bG)
            \end{equation}
            where the first arrow is defined, using the isometry induced by the preferred isotropic splitting of $\bG$, by the transformation
            $$
            \overline{ c^{-\frac12}{\rm Id}}\colon((\gggo)_{ c^{-\frac12}\mud}, \mc G(\bar g, 0 ))\to ((\gggo)_{cH}, \mc G(c\bar g, 0))
            $$ which is an isometry of ECA's thanks to \Cref{prop:scaledmu}. Furthermore,  we can use the explicit isometry in  \Cref{rmk_scal} to identify           
             $$
             ((\gggo)_{cH}, \mc G(c\bar g, 0))\simeq(c\cdot E,\bG)\,.
            $$In conclusion, the first arrow in \eqref{actionK} is the composition of the above isometries and it is then induced by the transformation
            $
            c^{-\frac12}{\rm Id}.$
             Finally, the last arrow in \eqref{actionK} is just the usual action of the  element $F\in{\rm  Aut}(c\cdot E, E)$. 
              Moreover,   by \Cref{lem_invlambdader}, any  element $c^{-\frac12}F$ can be represented with
                           $$
             c^{-\frac12}{\rm Id}\circ \bar A_ce^{b}=\begin{pmatrix}c^{-\frac12}A& 0 \\ ((c^{-\frac12}A)^{-1})^* b & ((c^{-\frac12}A)^{-1})^*             \end{pmatrix}=\overline{c^{-\frac12}A}e^b\in \mathsf {SO}(\ip)\,,$$
              where $ b\in \Lambda^2\ggo^*\,,$ and $ A\in {\rm Aut}(\mu)$  such that $A^{-1}\cdot H=c(H-\dd b)$. Finally, $-c^{\frac12}F$ is given by the composition of $$
              \overline{-{\rm Id}}\colon ((\gggo)_{ -c^{-\frac12}\mud}, \mc G(\bar g, 0 ))\to ((\gggo)_{ c^{-\frac12}\mud}, \mc G(\bar g, 0 ))
              $$ with the previously defined action.
             Now,  we can prove that  $  \mathsf K$ is actually a  Lie subgroup of $\mathsf{SO}(\ip).$ Clearly, it is sufficient to prove that, if  $c^{-\frac12}F\in c^{-\frac{1}{2}}{\rm Aut}(c\cdot E, E)$  and $\gamma^{-\frac12}G\in\gamma^{-\frac{1}{2}}{\rm Aut}(\gamma\cdot E, E)$, then 
             $c^{-\frac12}F(\gamma ^{-\frac12}G)^{-1}\in \mathsf K $. 
             It is not hard to see that 
             $$
         c^{-\frac12}\overline{ A_c}e^b(\gamma^{-\frac12}\overline{B_{\gamma}}e^{b'})^{-1}=\left(\frac{c}{\gamma}\right)^{-\frac12}(\overline{AB^{-1}})_{\frac{c}{\gamma}}e^{\beta}\,, \quad \beta:=\gamma B\cdot(b-b')\,.
             $$ Then,  we observe that $B\cdot H=\frac{1}{\gamma}H+ B\cdot \dd b' $.  This implies that  
             $$
             BA^{-1}\cdot H=c(B\cdot H-B\cdot \dd b )=\frac{c}{\gamma}(H-\dd\beta)\,,
             $$giving us the claim.
             Moreover,   $\mathsf K $ is closed in $\mathsf{SO}(\ip)$.  Thanks to the structure of the subgroup $\mathsf  K$, the claim is equivalent to prove that if $$
          c_i^{-\frac12}{\rm Aut}(c_i\cdot E, E)\ni  c_i^{-\frac12}F_i\to F_{\infty}\in \mathsf {SO}(\ip)\,, \quad i \to \infty\,, 
             $$ then $F_{\infty }\in \mathsf K$. Following \cite[Remark 3.7]{Heb},   let us fix a generalized metric  $\bG\in \mc M^{\mathsf G}$, then 
             $$
             c_i^{-\frac12}F_i\cdot \bG\to \tilde{\mc G}\in \mc M^{\mathsf G}.
                         $$
                         But, as a consequence of \eqref{genscal},  the effect of scaling on the scalar curvature and on $\lvert H\rvert^2$, recalling that $F_i$ preserves the Dorfman bracket, for all $i$,  we have that,
             $$
             \mc S( c_i^{-\frac12}F_i\cdot \bar{\mc G})=c_i\mc S(F_i\cdot\bar{\mc G})=c_i\mc S(\bar{\mc G})\to \mc S(\tilde{\mc G })\,, \quad i\to \infty\,.
             $$
              If  $S(\tilde{\mc G })\ne 0 $, then $c_i\to c$, as $i\to \infty$. This,  in particular,  implies that $ F_i \to \tilde{F}\in \mathsf {SO}(\ip)$ and so $F_{\infty}=c^{-\frac12}\tilde{F}$. We just need to prove that $\tilde F\in {\rm Aut}(c \cdot E, E)$. This is guaranteed by the fact that  
              $
              F_i= \overline{ (A_i)}_{c_i}e^{b_i}\in {\rm Aut}(c_i\cdot E, E ) 
               $ if and only if $$  A_i^{-1}\cdot H=c_i(H-\dd b_i)\,\, \mbox{and }  A_i\in {\rm Aut}(\mu)
              $$ and passing to the limit in the last equality gives us the claim.  Finally,  if  $\mc S(\tilde{\mc G})=0$, then, automatically, the Dorfman bracket on $E$ is abelian, i.e. $H=0$ and $\mu $ is abelian.  In this case, 
              $$
              {\rm Aut}(c\cdot E, E)=\{\bar A_ce^{b}:  A \in \mathsf{ GL}(\ggo)\,, \quad  b\in \Lambda^{2}\ggo^*\}\,.
              $$ Then,  $\mathsf K={\rm Aut}({E})$, which is clearly closed. 
            
            Easily, we see that 
          $$
          T_{\bG}(   \mathsf K\cdot\bG)=\bigoplus_{\lambda\in \R}(\lambda{\rm Id}+({\rm Der}_{\lambda}(\mud)\cap \mathfrak m))\,.
                    $$  Moreover, we  note that, using again \Cref{lem_invlambdader} 
                    $$
              \lambda{\rm Id}+({\rm Der}_{\lambda}(\mud)\cap \mathfrak m)=\left\{\begin{pmatrix} \lambda{\rm Id}+ D& 0 \\ 0 & -\lambda{\rm Id}- D^* \end{pmatrix}: D\in {\rm Der}(\mu)\cap {\rm Sym}(n),\quad \rho(D)H-2\lambda H=0\right\}\,.
                    $$
Now, suppose that $\bG$ and $\bG'$ are two  generalized nilsolitons with harmonic torsion which do not belong to the same $  \mathsf K$-orbit.
           Since $\mathcal O=  \mathsf  K\cdot\bG$ and $\mathcal O'=  \mathsf K\cdot\bG'$  consist in generalized nilsolitons,  we can assume that $d(\bG, \bG')=d(\mathcal O, \mathcal O')$ and that the unique geodesic $\gamma \colon[0,1]\to \mathcal M^{\G}$ such that $\gamma(0)=\bG$ and $\gamma(1)=\bG'$ meets $\mathcal O$ and $\mathcal O'$ orthogonally.
            Thanks to \Cref{rmrkmmso}, we know   $\tilde{\mm}(\bG)=\mathcal Rc(\bG)= \lambda {\rm Id}+ {\bm D}$\,, ${\bm D}\in {\rm Der}_{\lambda}(\mud)\cap \mathfrak m $ and using \Cref{genscalfunc} (1),  we have that $0=g(\gamma'(0), \nabla_{\bG}\mathcal S)=g(\gamma'(0), \mathcal Rc(\bG))$ and $0=g(\gamma'(1), \nabla_{\bG'}\mathcal S)=g(\gamma'(1), \mathcal Rc(\bG'))$ which gives us that 
           $$
           (\mathcal S\circ \gamma)'(0)=(\mathcal S\circ \gamma)'(1)=0\,.
           $$ On the other hand, using \Cref{genscalfunc} (2),  $\mathcal S$ is concave along the geodesics, so 
       $\mathcal S\circ \gamma$ is constant, which, thanks to \Cref{genscalfunc} (3),  happens if and only if $\gamma=\exp(tX)\cdot \bG$ where $X\in{\rm Der}_0(\mud)\cap \mathfrak m$, giving us  a contradiction.
\end{proof}
A first consequence of \Cref{thm_unique} is the uniqueness among generalized nilsolitons within $0\in H^3(\ggo, \R)$  of the classical nilsoliton, if it exists. 

Unfortunately, the proof of \Cref{thm_unique} is exclusive of the  harmonic torsion case. In view of this and of \Cref{thm_unique} itself, we can  weaken the uniqueness question as follows.
\begin{quest}\label{questharm}
Let $(E, \bG)$ a left-invariant metric  ECA over a nilpotent Lie group $\G$ and assume $\bG$ is  a generalized nilsoliton. Using $(E, \bar{ \mc G})\simeq ((\gggo )_H, \mc G(\bar g,0))$, is it $H$ $\bar g$-harmonic?
\end{quest}
 With last objective of addressing \Cref{questharm}, we can find a equivalent condition for it to be true. 
 \begin{prop} Let $\mud \in \mc D$ be an algebraic soliton. Then, $\dd_{\mu}^*H=0$ if and only if    $\tr((\dd_\mu^* H)^2 \operatorname{Ric}^B_{\mu, H}) \le  0$.
   \end{prop}
  \begin{proof} 
 
     Using the fact that \eqref{eqn_projections} completely determines $\mud \in \mc D$ and \eqref{eqn_algsoli}, we have that $\mud$ can be viewed as a pair $(\mu,H)$ solving
   \[
   	\begin{cases}
   		\theta(\operatorname{Ric}^B_{\mu, H})\mu = -\lambda \mu,\\
   		\rho(\operatorname{Ric}^B_{\mu, H})H + \dd_\mu \dd_\mu^*H = -\lambda H,
   	\end{cases}
   \]
   for some $\lambda \in \R$. 
   Applying $\dd^*$ to the second equation yields,  using \eqref{eqn_dRhocom},
   \begin{equation}\label{eqn_dstarHsol}
   		0 = \lambda \dd_\mu^*H + \dd_\mu^*\rho(\operatorname{Ric}^B_{\mu, H})H + \dd_\mu^*\dd_\mu \dd_\mu^*H
   		= \rho(\operatorname{Ric}^B_{\mu, H})\dd_\mu^*H + \dd_\mu^*\dd_\mu \dd_\mu^*H,
   \end{equation}
   where in the last equality we used $\theta(\operatorname{Ric}^B_{\mu, H})\mu = -\lambda \mu$. Now, just as in the proof of \Cref{prop_Rcmm}, we have that $\bar g(\rho(A)\alpha, \alpha) = -2 \tr(\alpha^2 A)$,  for all $2$-forms $\alpha\in \Lambda^2\ggo^*$ and $A\in \f{gl}(\f g)$. Now taking an inner product of \eqref{eqn_dstarHsol} with $\dd_\mu^*H$ gives
   \[
   0 = \bar g(\rho(\operatorname{Ric}^B_{\mu, H})\dd_\mu^*H, \dd_\mu^*H) + |\dd_\mu\dd_\mu^* H|^2 = - 2\tr((\dd_\mu^* H)^2 \operatorname{Ric}^B_{\mu, H}) + |\dd_\mu \dd_\mu^* H|^2.
   \]
	Since  $-\tr((\dd_\mu^* H)^2 \operatorname{Ric}^B_{\mu, H}) \geq 0$, then, we yield $\dd_\mu^*H = 0$ as required.
   \end{proof}
   Besides \Cref{questharm}, one can also  pose a finiteness question on generalized nilsolitons, namely:
   \begin{quest}\label{questfinite}
   Let $\G$ be a nilpotent Lie group. Is the set of generalized nilsolitons on $\G$ finite?
   \end{quest}
   In \Cref{sec_classifi} we will provide, within the full classification of generalized nilsolitons up to dimension $4$, a counterexample to \Cref{questfinite}. 
   
 Finally, it is well known, see for instance \cite{soliton},  that there exist nilpotent Lie algebras which do not admit any classical nilsoliton.  So,  one may hope to produce generalized nilsolitons even on such Lie algebras. This cannot be done in general. Indeed, on characteristically nilpotent Lie algebras, i.e. such that the Lie algebra of derivations is nilpotent, we cannot find neither classical nor generalized nilsolitons,  due to  the fact that any derivation is nilpotent, see \cite[Theorem 5]{To61}.

\section{Long-time behavior on nilmanifolds} \label{sec_LongNil}
In this Section, we will recall a definition of Cheeger-Gromov convergence for exact Courant algebroids introduced in \cite{GRFbook}. Then, we will discuss the long-time behavior of the Generalized Ricci flow on nilpotent Lie groups in the particular case in which the starting ECA has harmonic torsion, showing that, in this case, the asympotics are precisely the generalized nilsolitons defined in the previous Section. 

Given a subset $V \subset M$ of a manifold,  we will denote by $i_V \colon V \to M$ the inclusion map.
\begin{defn}[{\cite[Definition 5.25]{GRFbook}}]\label{def_ CCGconv}
A sequence of pointed metric ECAs $(E_i \to M_i,\mc G_i,p_i)_{i\geq 1}$ converges to $(\B E \to \B M , \B{\mc G},\B p)$ in the \emph{generalized Cheeger--Gromov} topology if there exists a  sequence  $\{U_i\}_{i\geq 1}$  of neighbourhoods of $ \B p$ exhausting $\B M$ and Courant algebroid isomorphisms 
\[
	(F_i,f_i) \colon (i_{U_i}^*\B E \to U_i) \to (F_i(i_{U_i}^*\B E) \to f_i(U_i))\subset (E_i \to M_i),
\]
satisfying:
\begin{enumerate}
\item $f_i^{-1}(p_i) \to \B p$, as $i \to \infty$;
\item $F_i^{-1} \circ \mc G_i \circ F_i \to \B{\mc G}$ in $C^\infty_{\operatorname{loc}}(\B E)$, as $i \to \infty$ .
\end{enumerate}
\end{defn}

As usual, fixing isotropic splittings of any ECA of the sequence, we can rewrite this definition as follows.
\begin{lem}\label{lem:CCGconv}
The sequence $((TM_i \oplus T^*M_i)_{H_i},\mc G(g_i,b_i),p_i)_{i\ge 1}$ converges to $((T\B M \oplus T^* \B M)_{\B H},\mc G(\B g,\B b),\B p)$ in the generalized Cheeger--Gromov topology if and only if there exists an exhaustion $\{U_i\}_{i\ge 1 }$ of $\B M$ consisting in neighbourhoods of $\B p$, diffeomorphisms $f_i\colon U_i \to f_i(U_i)\subset M_i$, and 2-forms $a_i \in \Omega^2(U_i)$ such that
\begin{enumerate}
	\item $i_{U_i}^*\B H = f_i^*H_i + \dd a_i$,  for all $i \geq 1$;
	\item $f_i^{-1}(p_i) \to \B p$,  as $i \to \infty$;
	\item $f_i^*g_i \to \B g$ in $C^{\infty}_{\operatorname{loc}}(\B M)$,  as $i \to \infty$;  
	\item $f_i^*b_i - a_i \to \B b$ in $C^{\infty}_{\operatorname{loc}}(\B M)$,  as $i \to \infty$.
\end{enumerate}
\end{lem}
\begin{proof} For the first direction, we see that in these isotropic splittings, the isomorphisms in \Cref{def_ CCGconv} are given by
\[
(F_i,f_i) = (\B{f_i}\circ e^{a_i},f_i),
\]
for some $2$-forms $a_i \in \Omega^{2}(U_i)$ satisfying
\[
i_{U_i}^*\B{H} = f_i^*H_i + \dd a_i.
\]
In this case
\[
F_i^{-1} \circ \mc G_i(g_i,b_i) \circ F_i = \mc G(f_i^* g_i,f_i^*b_i - a_i), 
\]
so $F_i \circ \mc G_i(g_i,b_i) \circ F_i^{-1} \to \mc G (\B g,\B b)$ if and only if Items (3) and (4) hold. The converse is easily seen.
\end{proof}
Moreover, we can specify \Cref{lem:CCGconv} when the preferred isotropic splitting of $\mathcal G_i$ are chosen. 
\begin{cor} \label{cor:ConvCanSplitting}
The sequence $((TM_i \oplus T^*M_i)_{H_i},\mc G(g_i,0),p_i)$ converges to $((T\B M \oplus T^* \B M)_{\B H},\mc G(\B g,0),\B p)$ in the generalized Cheeger--Gromov topology if and only if there is an exhaustion $\{U_i\ni \B p\}$ of $\B M$, diffeomorphisms $f_i\colon U_i \to f_i(U_i)\subset M_i$, and 2-forms $a_i \in \Omega^2(U_i)$ such that
\begin{enumerate}
	\item $i_{U_i}^*\B H = f_i^*H_i + \dd a_i$,  for all $i \geq 1$;
	\item $f_i^{-1}(p_i) \to \B p$,  as $i \to \infty$;
	\item $f_i^*g_i \to \B g$ in $C^{\infty}_{\operatorname{loc}}(\B M)$,  as $i \to \infty$;  
	\item $a_i \to 0$ in $C^{\infty}_{\operatorname{loc}}(\B M)$,  as $i \to \infty$.
\end{enumerate}
In particular, $f_i^*H_i \to \B H$ in $C^{\infty}_{\operatorname{loc}}(\B M)$,  as $i \to \infty$.
\end{cor}

We now describe how such convergence is achieved at the level of Lie brackets. Fix a vector space $\f g$ and a generalized metric $\overline{\mc G} = \mc G(\bar g,0)$ on $\f g \oplus \f g^*$. For a Dorfman bracket $\mud$ on $\f g \oplus \f g^*$, denote by $((\f g \oplus \f g)_{\mud},\G_{\mud_{\f g}},\mc G_{\mud})$ the corresponding metric ECA.
\begin{thm} \label{thm_BToCG}Let   $(\mud_i)_{i\geq 1}$ be a sequence of nilpotent Dorfman brackets converging to $\B \mud$ on $\f g \oplus \f g^*$. Then, $\left((\f g \oplus \f g^*)_{\mud_i},\G_{(\mud_i)_{\f g}},\mc G_{\mud_i},e\right)_{i\ge 1}$ converges to $\left((\f g \oplus \f g^*)_{\B \mud},\G_{\B\mud_{\f g}},\mc G_{\B \mud},e\right)$ in the generalized Cheeger--Gromov topology.
\end{thm}
\begin{proof} As usual, we consider $\mc G( g,0)$ as a background generalized metric, hence 
\[
\left((\f g \oplus \f g^*)_{\mud_i},\G_{(\mud_i)_{\f g}},\mc G_{\mud_i},e\right) = \left((\f g \oplus \f g^*)_{H_i},\G_{\mu_i},\mc G(g_{\mu_i},0),e\right),
\]
where $H_i = \mud_{\Lambda^3}$ and $\mu_i = \mud_{\f g}$. We also set $\B \mu = \B \mud_{\f g}$ and $\B H = \B \mud_{\Lambda^3}$. Then by nilpotence, the exponential maps of $\G_{\mu_i}$ and $\G_{\B \mu}$ are diffeomorphisms. We consider
\[
f_i := \operatorname{exp}_{\mu_i}\circ \operatorname{exp}_{\B \mu}^{-1} \colon \G_{\B \mu} \to \G_{\mu_i},
\]
which is a diffeomorphism with $f_i(e) = e$. Recall (see for instance \cite{Lauret2012}) that for $X \in \f g$, the exponential map $\exp_\mu \colon \f g \to \G_\mu$ of the Lie bracket $\mu \in \Lambda^2 \f g^* \otimes \f g$ satisfies
\[
(\dd \exp_\mu)_X = (\dd L_{\exp_\mu(X)})_e \circ \frac{\id - e^{-\ad_\mu X}}{\ad_\mu X},
\]
where
\[
\frac{\id - e^{-\ad_\mu X}}{\ad_\mu X} = \sum_{k=0}^\infty \frac{(-1)^k}{(k+1)!}(\ad_\mu X)^k.
\]
Hence,  $\dd f_i$ depends continuously on $\mu_i$, and an argument analogous to \cite[Corollary 6.10]{Lauret2012} shows that  $f_i^*g_{\mu_i} \to g_{\B \mu}$ and $f_i^*H_i \to \B H$ smoothly and uniformly on compact subsets. Then, we define
\[
a_i := \iota_{\id}(\B H - f_i^*H_i),
\]
where $\id \in \Gamma(T\R^n)$ is the identity vector field on $\R^n$. By Cartan's formula,
\[
\dd a_i = \mathcal{L}_{\id} (\B H - f_i^*H_i) = \B H - f_i^*H_i.
\]
Hence, $\B H = f_i^*H_i + \dd a_i$ with $a_i \to 0$ smoothly and uniformly on compact subsets. Thus, the result follows by \Cref{cor:ConvCanSplitting}.
\end{proof}

With the technical details established, we now turn to studying the long-time behaviour of the generalized Ricci flow on nilpotent Lie groups when the initial metric ECA has harmonic torsion.
\begin{prop}\label{prop_BFConv}
Let $(\mud(t))_{t \in [0,\infty)}$ be a solution to the generalized bracket flow \eqref{eqn_genbf} such that $\mud(0)_{\ggo}$ is a nilpotent Lie bracket and $\mud(0)_{\Lambda^3}$ is harmonic. Then, the rescaled brackets $\tfrac{\mud(t)}{|\mud(t)|}$ converge as $t\to \infty$ to an algebraic soliton bracket $\mud_\infty \neq 0$.
\end{prop}
\begin{proof}
Long-time existence follows immediately from \Cref{cor:SolvInfTime}. Up to a time-reparametrization, the rescaled brackets $\B \mud(t) = \frac{\mud(t)}{|\mud(t)|}$ solve \eqref{eqn_ngbf} with $\ell = \ell_{\B{\mud}}$ (defined in \S 7.) That is, the scalar-normalized bracket flow. Since the brackets are nilpotent, \Cref{cor_mmnil} implies that
\[
\frac{\dd}{\dd t}\B \mud = - \Theta(M_{\B \mud} + 6 |M_{\B{\mud}}|^2\id)\B \mud .
\]
 On the other hand,  using the same strategy as in  \cite[Lemma 7.2]{realGIT}, \Cref{lem:evdstar} and similar computations as in the proof of \Cref{genscalfunc}, one can easily see that this is, up to scaling, the negative gradient flow of the real-analytic functional 
\[
\Lambda^2 (\f g \oplus \f g^*)^* \otimes (\f g \oplus \f g^*) \ni \mud \longmapsto \frac{|M_{\mud}|^2}{|\mud|^4} \in \R.
\]
By compactness and Łojasiewicz’s Theorem on real-analytic gradient flows \cite{Loj63}, the $\omega$-limit consists of a unique fixed point $\mud_\infty$ to the scalar-normalized bracket flow, and hence  it is an algebraic soliton.
\end{proof}
Let us remark some important facts.   First of all, we should observe that the group $\mathsf L $ is not reductive. This does not allow us, in general,  to use directly  \cite[Lemma 7.2]{realGIT} to deduce the fact that the scalar-normalized bracket flow is a gradient flow. Moreover, it is easy to be shown that, removing the harmonicity of the initial torsion, the scalar-normalized bracket flow  is not the negative  gradient flow  the above functional.  Finally, this gives an alternative and dynamical proof of the uniqueness of generalized nilsolitons with harmonic torsion.

Let us now prove the main convergence results on nilpotent Lie groups.
\begin{thm}
Let $(E\to \G,\mc G_t)$ be a left-invariant  solution of the generalized Ricci flow on a left-invariant ECA over a simply-connected, nilpotent, non-abelian Lie group $\G$. Assume that the preferred  representative of the Ševera class of $E$  determined by $\mc G_0$ is harmonic. Then, the scaled ECA's $(-\mc S(\mc G_t)\cdot E,\mc G_t)$ converge in the generalized Cheeger--Gromov sense to a left-invariant, nilpotent, non-abelian, generalized Ricci soliton.
\end{thm}
\begin{proof}
After choosing an isotropic splitting, the left-invariant solution is equivalent to the bracket flow \eqref{eqn_genbf} by \Cref{thm_grf=bf}. Denoting this solution by $(\mud(t))_{t\in [0,\infty)}$, we have by \eqref{eqn_gscmu} and \Cref{prop_BFConv} that the rescaled brackets $\frac{\mud(t)}{(-\mc S (\mc G_t))^{1/2}}$ converge to a nilpotent, non-abelian algebraic soliton bracket. These rescaled brackets correspond to the rescaled ECA's $(-\mc S(\mc G_t)\cdot E,\mc G_t)$ by \Cref{prop:scaledmu} and \Cref{lem_scaling}. Finally, by \Cref{thm_BToCG}, we obtain convergence of $(-\mc S(\mc G_t)\cdot E,\mc G_t)$ to an algebraic soliton.
\end{proof}
From a more classical perspective, we obtain the following slightly weaker statement:
\begin{cor}
Let $(\G,H_t,g_t)$ be a left-invariant  solution of the generalized Ricci flow on a simply-connected, non-abelian, nilpotent Lie group $\G$ and assume that $H_0$ is $g_0$-harmonic. Then,  the rescaled family\\ $(\G,-\mc S(H_t, g_t)H_t,-\mc S(H_t, g_t)g_t)$ converges to a nilpotent, generalized soliton $(\B \G,\B H,\B g)$ in the following sense: for every sequence of times, there is a subsequence $t_k \to \infty$, and diffeomorphisms $f_k \colon \B \G \to \G$ fixing the identity, such that $(f_k^*H_{t_k},f_k^*g_{t_k}) \to (\B H,\B g)$,  as $k \to \infty$,  in $C^\infty_{\textrm{loc}}(\B \G)$.
\end{cor}

\section{Connections to the pluriclosed flow} \label{sec_PCF}

 The pluriclosed flow is an important geometric flow in Hermitian Geometry firstly introduced by Streets and Tian in \cite{ST10}.  We quickly recall its definition and its link with the generalized Ricci flow. 
\begin{defn}
Let $(M, J, g_0)$ be a Hermitian manifold, i.e. $J\in {\rm End}(TM)\cap \mathsf {O}(g_0)$, $J^2=-{\rm Id}$ and such that 
$$
N_J(X, Y)=[X, Y]+J[JX, Y]+J[X, JY]-[JX, JY]=0\,, \quad  X, Y\in \Gamma (TM)\,.
$$ The pluriclosed flow is the following evolution equation:
$$
\frac{\partial}{\partial t}g=-S+Q\,,\quad  g(0)=g_0
$$ where $S_{i\bar j }=g^{k\bar l }R_{k\bar l i\bar j }$ is the {\em second Chern-Ricci tensor} while $Q_{i\bar j }=g^{k\bar l }g^{m\bar n }T_{ik\bar n }T_{\bar j \bar l m }.$
\end{defn}
The tensors $R$ and $T$  in the definition above are, respectively, the curvature tensor and the torsion of the \emph{Chern connection} $\nabla$, namely the unique linear connection on $M$ such that 
 $$
 \nabla g=0 \,, \quad \nabla J=0 \quad \mbox{and }\,\, T^{1,1}=0\,.
 $$
 
The pluriclosed flow is part of a larger class of geometric flows called \emph{Hermitian curvature flows}, introduced in \cite{ST11},  which share many common properties such as the  short-time existence on compact complex manifolds and the preservation of  K\"ahlerianity along the flow. However, among this class, the pluriclosed flow is the natural flow  in the realm of  strong K\"ahler with torsion (SKT, for short) or pluriclosed metrics. The latter are defined by requiring that their fundamental form $\omega(\cdot, \cdot):=g(J\cdot, \cdot)\in \Omega^{1,1}(M)$ is $
\dd\dd^c
$-closed, where $\dd^c:=J^{-1}\dd J$ is the twisted exterior differential.
\begin{prop}
Let $(M, J, g_0)$ be a  SKT Hermitian manifold. Then, the pluriclosed flow preserves the SKT condition. Moreover,  $(g(t))_{t\in[0, T)}$ is a solution of the pluriclosed flow starting at $g_0$ if and only if the family $(\omega(t))_{t\in[0,T)}$ of the fundamental forms associated to $g(t)$ is a solution of 
$$
\frac{\partial }{\partial t}\omega=-(\rho^B)^{1,1}(\omega)\,, \quad \omega(0)=\omega_0
$$ where $(\rho^B)^{1,1}(\omega)$ denotes the $(1,1)$-part of the \emph{Bismut-Ricci form} of $\omega$.
\end{prop}
  Moreover,  we have the following Theorem. 
  \begin{thm}[\cite{Str13}, Proposition 6.3, 6.4] \label{thm_ST}Let $(M, J)$  be a complex manifold  and $(g(t))_{t\in [0,T)}$ be a solution of the pluriclosed flow starting from a  SKT metric $g_0$. Then, $(g(t), H(t))_{t\in [0,T)}$ solves the following:
  $$
  \begin{aligned}
  \frac{\partial }{\partial t}g=&\, -{\rm Ric}_{g, H}^{B}-\frac12\mathcal L_{\theta^{\sharp}}g\,, \quad  g(0)=g_{0}\,,\\
  \frac{\partial }{\partial t }H=&\, \frac12\Delta_gH-\frac12\mathcal L_{\theta^{\sharp}}H\,, \quad H(0)=H_0\,,
  \end{aligned}
  $$ where $\theta\in \Omega^1(M)$ is the \emph{Lee form}  of the time-varying metric defined by $\dd\omega^{n-1}=\theta\wedge \omega^{n-1}$   while $H(t):=\dd^c\omega(t)$ is the torsion of the Bismut connection.
  \end{thm}
  Therefore,  up to a time reparametrization, a  solution of the pluriclosed flow is a solution of  a gauged fixed generalized Ricci flow where the gauge is generated by the generalized vector field $-\theta^{\sharp}-\dd\theta+\iota_{\theta^{\sharp}}H_0+\dd\iota_{\theta^{\sharp}}b\in\f{aut}((\TT)_{H_0})$.
   \begin{remark}\label{equipluriGRF}
   Streets and  Tian in \cite[Corollary 3.3]{ST12} proved that, given a SKT Hermitian  manifold $(M, J, g)$, the generalized Ricci flow coupled with  an evolution equation for $J$:
  \begin{equation}\label{eqn_J}
     \frac{\partial }{\partial t}J=\Delta J+ [J, g^{-1}{\rm Ric}_g]+ \mc Q(DJ)\,,
   \end{equation}
    where $\mc Q(DJ)$ is an appropriate quadratic term in the covariant derivative of $J$ with respect to the Levi-Civita connection, see \cite[(3.24)]{ST12}  for the precise expression of $\mc Q(DJ)$, preserves the SKT condition, namely $g(t)$ is $J(t)$-hermitian and SKT for all times of existence.  Then, gauging  with the family of diffeomorphisms generated by $(-J(t) \dd_{g(t)}^*\omega(t))^{\sharp }$ the generalized Ricci flow coupled with  \eqref{eqn_J} and  using \cite[Proposition 3.1]{ST12}, we obtain a solution for the pluriclosed flow, up to a time reparametrization. 
    \end{remark} 
    As a direct consequence of \Cref{thm_ST},  \Cref{equipluriGRF}  and \Cref{cor:SolvInfTime}, we have the following. 
    \begin{cor}\label{cor_longtimepluri}
    Let $ (\mathsf G,J, g)$ be a solvable Lie group endowed with a left-invariant SKT structure. Then, the solution of the pluriclosed flow starting from $g$ is immortal.
    \end{cor}
     \Cref{cor_longtimepluri} recovers many known results in the literature such as those in \cite{Bol2016} on compact  complex surfaces, in \cite{EFV15} on SKT nilmanifolds, in \cite{AL19} on almost-abelian SKT solvmanifolds and in \cite{FV23} on Oeljeklaus-Toma manifolds.
     
The gauge-equivalence between pluriclosed and generalized Ricci flow allows us also  to  give an equivalent condition for a generalized Ricci soliton to be a pluriclosed soliton.
   \begin{thm}\label{thm_grfsvspls}
   Let $(M, J)$ be a complex manifold endowed with a SKT Hermitian metric $g$. Then, a soliton  $(g, H, X)$ for the generalized Ricci flow is a pluriclosed soliton if and only if $H=\dd^c\omega$ and $X+\theta^{\sharp}$ is holomorphic. 
   \end{thm}
   \begin{proof}
   We recall that $(\omega, X)$ is a pluriclosed soliton if and only if 
   $$
   (\rho^B)^{1,1}(\omega)=\lambda\omega+\frac12\mc L_X\omega\,, \quad X \mbox{ holomorphic}\,, \quad \lambda \in \R\,.
   $$ By means of the holomorphicity of $X$,   we have that $[\mc L_X, J]=0$ on any form on $M$,  yielding that $\dd^c\mc L_X\omega=\mc L_X\dd^c\omega$. Then,  using \cite[Proposition 6.4]{Str13} and defining $H=\dd^c\omega$, we obtain 
   $$
  \frac12\Delta_gH-\frac12\mc L_{\theta^{\sharp}}H= -\dd^c(\rho^B)^{1,1}(\omega)=-\lambda H-\frac12\mc L_XH\,,
   $$ which gives us 
   $$
   \Delta_gH=-2\lambda H-\mc L_{X-\theta^{\sharp}}H\,.
   $$ On the other hand, we have that
   $$
   {\rm Ric}^B_{g, H}+\frac12\mc L_{\theta^{\sharp}}g=(\rho^B)^{1,1}(\cdot, J\cdot)=\lambda g+\frac12\mc L_Xg.
   $$ This guarantees that 
   $$
   {\rm Ric}_{g, H}^B=\lambda g+\frac12\mc L_{X-\theta^{\sharp }}g\,,
   $$giving us the first claim. The converse can be proved following the same steps backwards.    
   \end{proof}
     
   In view of  \Cref{thm_ST} and \Cref{equipluriGRF},  we readily obtain the gauge-equivalence  between  the generalized bracket flow and  the bracket flow for the pluriclosed flow introduced in \cite{EFV15} and \cite{AL19}.  First of all, we recall the definition of the latter.
  \begin{defn} Let $(\mathsf G, \mu_0)$ be  a Lie group endowed with left-invariant complex structure $J$ and a left-invariant SKT metric $\omega$. The bracket flow is the following evolution equation: 
  \begin{equation}\label{eqn_plu}
  \frac{\dd}{\dd t}\mu=-\frac12\theta(P)\mu\,, \quad \mu(0)=\mu_0\,
  \end{equation} where  $P\in  \mathfrak{gl}(\ggo, J  )$ defined by $\omega(P\cdot, \cdot)=(\rho^B)^{1, 1}(\cdot, \cdot)$.
  
  \end{defn}

   Combining \Cref{thm_grf=bf}, \Cref{thm_ST}, \Cref{equipluriGRF} and \cite[Theorem 4.2]{EFV15}, we have the following.
  \begin{cor}\label{cor_gaugeequi} Up to a time reparametrization by $2$,  the bracket flow \eqref{eqn_plu} is gauge equivalent to the generalized bracket flow \eqref{eqn_genbf}.
  \end{cor}

   Finally, \Cref{thm_grfsvspls} can be specialized in the case of algebraic solitons,  as defined in \Cref{def_algsol} and in \cite[Definition 2.7]{AL19}. For the sake of clarity, we quickly  recall  the definition of algebraic solitons for the pluriclosed flow.
   \begin{defn}[\cite{AL19}, Definition 2.7]A Lie group $\mathsf G$ endowed with a left-invariant  SKT structure $(J, g)$ is called algebraic pluriclosed soliton if there exists $D \in {\rm Der}(\ggo)\cap \mathfrak{gl}(\ggo, J)$ and $ \lambda\in \mathbb R$ such that 
   $$
   P=\lambda{\rm Id}+\frac12(D+D^t). 
   $$ 
   \end{defn}
   \begin{cor}\label{cor_algplurivsgensol}
    Let $\mathsf G$ be a Lie group endowed with a left-invariant SKT structure $(J, g)$. Then, an algebraic soliton $(g, H, D)$ for the generalized Ricci flow is an algebraic pluriclosed  soliton if and only if $D- {\rm ad}_{\theta^{\sharp}}\in {\rm Der}(\ggo)\cap \mathfrak{gl}(\ggo, J)$ and $H=\dd^c\omega$. 
   \end{cor}
   \begin{proof}
  By \Cref{lem_algrealsol},  $(g, H, D)$ is an algebraic soliton for the generalized Ricci flow if and only if $(g, H, -X_D)$ is a  soliton for the generalized Ricci flow where $X_D$ is the vector field defined by 
  $$
  X_D=\frac{\dd}{\dd t}\Big|_{t=0}\varphi_t \mbox{ with } \varphi_t\in {\rm Aut}(\mathsf G) \mbox{ such that }\dd_e\varphi_t=e^{tD}\,.
  $$ 
  On the other hand, by \Cref{thm_grfsvspls},  $(g, H, -X_D)$ is a pluriclosed soliton if and only if $H=\dd^c\omega$ and $-X_D+\theta^{\sharp}$ is holomorphic. But, $-X_D+\theta^{\sharp}$ is holomorphic if and only if $D-{\rm ad}_{\theta^{\sharp }}\in \mathfrak{gl}(\ggo, J)$, giving the claim.
   \end{proof}

\section{Classification of solitons up to dimension 4}\label{sec_classifi}
 In this Section, we will provide a classification of generalized nilsolitons on nilpotent Lie algebras of dimension up to $4$.  As it is clear  from the discussion in \Cref{sec_gensol}, we can only hope to classify generalized nilsolitons up to scaling and up to generalized automorphisms of the ECA we are considering.  Classical nilsolitons were classified in \cite{finding}. So, we will focus on classifying the non-classical ones.

  As it is well known, besides the abelian ones, nilpotent Lie algebras of dimension up to $4$ are $3$: $\mathfrak n_3$, the Lie algebra associated to the Heisenberg group ${\rm Heis}(3, \mathbb R)$,  which is the only non-abelian nilpotent Lie algebra of dimension $3$, $\mathfrak n_3\oplus \mathbb R$ and $\mathfrak n_4$, the only  indecomposable non-abelian nilpotent Lie algebra of dimension $4$.  
 
 We will go through the classification case by case using \eqref{eqns_soli}. 
 
 In dimension 3, there is only one generalized soliton  which is non-classic. This  is mainly due to the fact that $H^3(\mathfrak n_3, \R)=\mathbb R$.
 \begin{prop}\label{heisnils}
  On $\mathfrak n_3$, the unique  non-classic generalized nilsoliton  is given by:
  $$
   \mu(e_1, e_2)= e_3\,, \quad H=e^{123}\,, \quad \lambda=-2\,,\quad D={\rm diag}(1,1,2)\,.
  $$
  \end{prop}
  \begin{proof}
  Given a left-invariant  metric on  ${\rm Heis}(3, \mathbb R)$,  thanks to \cite[pag 305, (4.2)]{Mln},   we can always find an orthonormal basis $\{e_1, e_2, e_3\}$ of $\mathfrak n_3$ such that 
 $$
 \mu(e_1, e_2)=a e_3\,, \quad a\in\mathbb R\backslash\{0\}\,.
 $$ Moreover, thanks to \cite[Theorem 4.3]{Mln},   we know that
 $$
 {\rm Ric}_{\mu}=\frac12a^2{\rm diag}(-1,-1,1)\,.
 $$ Now, every  $0\ne H\in\Lambda^3\mathfrak n_3^*$, it will be of the form $H=be^{123}$ with  $b\ne 0 $ and it will be trivially closed and harmonic.  On the other hand, we have that 
 $$
 H^2=2b^2{\rm Id}
 $$ which implies that 
 $$ {\rm Ric}_{\mu, H}^B=\frac12 {\rm diag}(-(a^2+b^2), -(a^2+b^2), a^2-b^2)
 $$ Moreover,  every  $D\in {\rm Der}(\mathfrak n_3, \mu)$ has to satisfy 
 $$
 D_3^3=D_1^1+D_2^2\,.
 $$  This guarantees that 
 $$
 {\rm Ric}_{\mu, H}^B-\lambda{\rm Id}\in{\rm Der}(\mathfrak n_3, \mu)
 $$ if and only if 
 $$
  \lambda=-\frac12(3a^2+b^2)\,.
 $$
  On the other hand,  by straightforward computations we have that 
 $$
 \begin{aligned}
0= &\, \rho({\rm Ric}_{\mu, H}^B)H-\Delta_{\mu}H+\ell H
=-(a^2-b^2)H\,,
\end{aligned}
 $$  which is satisfied if and only if $ a^2=b^2$, giving us the claim, using \Cref{prop_equisol}.
 \end{proof}
 
  Now, we will focus on the $4$-dimensional case.  A first difference between the $3$-dimensional one is  a  larger third cohomology group which allows a  wider  possibility for a generalized metric to be a generalized nilsoliton. 
  
 We will, first of all, analyse the case of $\mathfrak n_3\oplus \mathbb R$. An important remark to be done is that this Lie algebra is the only one among the $4$-dimensional ones which admits left-invariant complex structures, see for instance \cite{MadsenSwann} or \cite{Sal}. The complex structure on $ \mathfrak n_3\oplus \mathbb R$ is unique  up to biholomorphisms and quotients by co-compact  lattices give rise to the so-called \emph{primary Kodaira surfaces}, the only compact complex surfaces with zero  Kodaira dimension and first Betti number equal to 3.  Clearly, every left-invariant  Hermitian metric on $\mathfrak n_3\oplus \mathbb R$ is SKT. So, using \Cref{cor_algplurivsgensol}, we will be able to detected which generalized nilsoliton on $\mathfrak n_3\oplus \mathbb R$ is an algebraic pluriclosed soliton.
 \begin{prop}\label{prop_n3}
   On $\mathfrak n_3\oplus \mathbb R$,  the only non-classic generalized nilsolitons are the following:
   \begin{enumerate}
   \item\label{firstcase} $\displaystyle{ \mu(e_1, e_2)=e_3\,, \quad H=e^{123}\,,\quad  \lambda =-2\,, \quad D={\rm diag}(1,1,2,2)\,.}$
   \item \label{inffamilysol}$\displaystyle{\mu(e_1, e_2)=e_3\,, \quad H=\lambda_1e^{234}+\lambda_2e^{134}\,, \quad \lambda_1^2+\lambda_2^2=1\,, \quad  \lambda =-\frac 32\,, \quad D={\rm diag}\left(1-\frac{\lambda_2^2}{2},1-\frac{\lambda_1^2}{2},\frac32,1\right)}\,.$
   \end{enumerate}
   In particular, the first one is the unique algebraic pluriclosed soliton on  $\mathfrak n_3\oplus \mathbb R$.
 \end{prop}
 \begin{proof}
 Thanks to  \cite[ Theorem 3.1]{VanT}, every inner product on $\mathfrak n_3\oplus \mathbb R$ is isometric to an inner product such that an  orthonormal basis $\{e_1, \ldots, e_4\}$   satisfies the following
 $$
 \mu(e_1, e_2)=a e_3\,, \quad a\ne 0\,. 
 $$We know that $D\in {\rm Der}(\mathfrak n_3\oplus \mathbb R, \mu)\cap {\rm Sym}(\mathfrak n_3\oplus \mathbb R)$ if and only if 
 $$
  D=\begin{pmatrix}
  a_1& a_3 & 0 & a_5\\
  a_3&a_2& 0 & a_6\\
  0 & 0 &a_1+a_2& 0\\
  a_5&a_6&0& a_4
  \end{pmatrix}\,, \quad a_i\in \mathbb R\,, \quad i=1,\ldots, 6\,,
  $$ with respect to the frame $\{e_1, \ldots, e_4\}$.
  Moreover,  using \cite[Lemma 4.2]{nilRF}, it is easy to see that 
$$
  {\rm Ric}_{\mu}=\frac12a^2 {\rm diag}(-1, -1, 1, 0)\,.
  $$
It is straightforward to check  that  
   a generic closed $H\in \Lambda^3(\mathfrak n_3\oplus \mathbb R)^*$ is of the form: 
 $$
 H=\lambda_4e^{123}+\lambda_3e^{124}+\lambda_2e^{134}+\lambda_1e^{234}\,,\quad \lambda_i\in \mathbb R\,, \quad i=1,\ldots, 4\,.
 $$Consequently, we have that 
 $
 \Delta_{\mu}H=-a^2\lambda_3e^{124}\,
 $
 and 
 \begin{equation}\label{eqn_sqh}
  H^{2}=2\begin{pmatrix}
  \lambda_4^2+\lambda_2^2+ \lambda_3^2& \lambda_2\lambda_1&  -\lambda_3\lambda_1& \lambda_4\lambda_1\\
  \lambda_2\lambda_1& \lambda_4^2+\lambda_1^2+\lambda_3^2& \lambda_2\lambda_3& -\lambda_4\lambda_2\\
  -\lambda_3\lambda_1 & \lambda_3\lambda_2 & \lambda_1^2+\lambda_2^2+\lambda_4^2& \lambda_4\lambda_3 \\
  \lambda_4\lambda_1& -\lambda_4\lambda_2& \lambda_4\lambda_3 & \lambda_1^2+\lambda_2^2+\lambda_3^2
  \end{pmatrix}\,.
 \end{equation}
  This readily implies that we need to impose that  $\lambda_3=0 $ or $  \lambda_4=\lambda_2=\lambda_1=0$ in order to hope for   ${\rm Ric}_{\mu, H}^{B}-\lambda{\rm Id}\in {\rm Der}(\mathfrak n_3\oplus \mathbb R, \mu)$.\\

 First of all,  we suppose $\lambda_3=0$. In this case,   ${\rm Ric}_{\mu, H}^{B}-\lambda{\rm Id}\in{\rm Der}(\mathfrak n_3\oplus \mathbb R, \mu)$  if and only if 

 $$
 \lambda=-\frac12(\lambda_4^2+3a^2)\,.
 $$ The condition 
  $$
  \rho({\rm Ric}_{\mu, H}^{B})H-\Delta_{\mu}H+\lambda H=0
  $$
  is equivalent to 
  \begin{equation}\label{sist}
  \begin{cases}
 \lambda_4(-2a^2+3\lambda_2^2+2\lambda_4^2+3\lambda_1^2)=0\,,\\
\lambda_2(-3a^2+3\lambda_2^2+2\lambda_4^2+3\lambda_1^2)=0\,,\\
\lambda_1(-3a^2+3\lambda_2^2+2\lambda_4^2+3\lambda_1^2)=0\,.
  \end{cases}
 \end{equation} 
 However, the first equation implies that $a^2=\frac12\left( 3\lambda_2^2+2\lambda_4^2+3\lambda_1^2\right)$ or $ \lambda_4=0\,.$ 
On the other hand, if $a^2=\frac12( 3\lambda_2^2+2\lambda_4^2+3\lambda_1^2)$, we have that 
$$
-3a^2+3\lambda_2^2+2\lambda_4^2+3\lambda_1^2=-a^2\ne 0 \,.
$$ From this, we obtain that $\lambda_1=\lambda_2=0\,,$ giving us \Cref{firstcase}. 

  Now let $\lambda_4=\lambda_3=0$. In this case, the second equation in \eqref{sist} implies that 
  $\lambda_2=0$ or $\lambda^2=\lambda_1^2+\lambda_2^2$. If  the last condition holds, then,  the third equation is satisfied too, giving us \Cref{inffamilysol}.  
  
  It remains to analyse  when $\lambda_4=\lambda_3=\lambda_2=0$. In this case,  the third equation in \eqref{sist} gives us that 
  $ a^2=\lambda_1^2 $, since $\lambda_1\ne 0$ if we want non-classical generalized nilsoliton. Thus, we  recover a particular generalized nilsoliton belonging to \Cref{inffamilysol}.

  It remains to analyse when $\lambda_3\ne 0 $ and $\lambda_4=\lambda_2=\lambda_1=0$. In this case, $ {\rm Ric}_{\mu, H}^{B}-\lambda{\rm Id}  \in {\rm Der}(\mathfrak n_3\oplus \mathbb R, \mu)$ if and only if $\lambda=-\frac12(3a^2+2\lambda_3^2)
$.
 On the other hand, the condition
  $$
  \rho({\rm Ric}_{\mu, H}^B)H-\Delta_{\mu}H+\lambda H =0 
  $$ is equivalent to 
  $
  \lambda_3(a^2+\lambda_3^2)=0
  $ which cannot happen unless $\lambda_3=0$, concluding the classification.\\
 To address the last part of the statement, we  
 define $Je_1=e_2$, $Je_4=e_3$ and  consider 
 $$
 \omega^1=\frac{1}{\sqrt{2}}(e^1+\sqrt{-1}e^2)\,, \quad \omega^2=\frac{1}{\sqrt{2}}(e^4+\sqrt{-1}e^3)
 $$ obtaining that 
 $$
 \dd\omega^1=0\,,\quad \dd\omega^2=\frac{a}{\sqrt{2}}\omega^{1\bar 1}\,,
 $$ which, in particular, gives us the integrability of $J$.
 The fundamental form of the metric  is then $\omega=\sqrt{-1}(\omega^{1\bar 1}+ \omega^{2\bar 2})$,  which is SKT,  while the torsion of the Bismut connection is given by 
 $$
 H:=\dd^c\omega=\sqrt{-1}(\bar\partial\omega-\partial \omega)=-\frac{1}{\sqrt{2}}a(\omega^{1\bar 1\bar 2}-\omega^{1\bar 12})=a e^{123}\,.
 $$ On the other hand, we have that 
 $$
 \dd\omega=\frac{\sqrt{-1}}{{\sqrt{2}}}a(\omega^{1\bar 12}+\omega^{1\bar 1 \bar 2})=a e^{124}\,.
 $$Then, it is easy to see that Lee form associated to $\omega$ has the following form:
 $$
 \theta=\frac{a}{\sqrt{2}}(\omega^2+\omega^{\bar 2})=a e^4\,.
 $$ 
 This gives easily that ${\rm ad}_{\mu }\theta^{\sharp}=0.$
 Moreover, given a diagonal derivation $D$, it will commute with $J$ if and only if  it is of the form 
 $$
 D= b{\rm diag}(1,1,2,2)\,,\quad b\in \mathbb R\,.
 $$  But, the derivation in \Cref{firstcase}   is of the above form. Then, applying \Cref{thm_grfsvspls},  we obtain the claim, concluding the proof. 
   \end{proof}
  \Cref{inffamilysol} in \Cref{prop_n3} in particular provides the first example of nilpotent Lie group admitting a infinite family of generalized nilsolitons, answering negatively to \Cref{questfinite}.  
 
  Now, we focus on the study of generalized nilsolitons in the case of $\mathfrak n_4$.
  \begin{prop}
  On $\mathfrak n_4$, the unique non-classic generalized nilsolitons are the following:
  \begin{enumerate}
  \item\label{firstcasen4}$\displaystyle{\mu(e_1, e_2)=e_3\,,\quad \mu(e_1, e_3)=\frac{\sqrt{3}}{2}e_4\,, \quad H=\frac{\sqrt{3}}{2}e^{134}\,, \quad \lambda=-\frac32\,, \quad D={\rm diag}\left(\frac23, 1, \frac54, \frac32\right)}\,,$
  \item \label{secondcase}$\displaystyle{\mu(e_1, e_2)=e_3\,,\quad \mu(e_1, e_3)=e_4\,, \quad H=e^{234}\,, \quad \lambda=-\frac32\,, \quad D={\rm diag}\left(\frac12, \frac12, 1, \frac12\right)}$\,.
  \end{enumerate}
  \end{prop}
  \begin{proof}
  Thanks to \cite[Theorem 3.2]{VanT}, up to isometries, given any inner product on $\mathfrak n_4$, we can always find an orthonormal frame  $\{e_1, \ldots, e_4\}$ of $\mathfrak n_4$ such that 
 $$
 \mu(e_1,e_2)=ae_3+be_4\,, \quad \mu(e_1, e_3)=ce_4\,, \quad a, c>0\,.
 $$
In this framework, a symmetric derivation $D$ is of the form 
 $$
 D={\rm diag}(\alpha, \beta, \alpha+\beta, 2\alpha + \beta)\,,\quad \beta, \alpha \in \mathbb R\,,$$
 see for instance \cite[Lemma 1]{Tho22}.  Moreover, the Ricci endomorphism, using again \cite{nilRF}, takes the following form:
\begin{equation}\label{eqn_ric}
{\rm Ric}_{\mu}=\frac12\begin{pmatrix}-(a^2+b^2+c^2)& 0 & 0 & 0 \\
0 & -(a^2+b^2)& -bc& 0 \\
0 & -bc&a^2-c^2& ab\\
0 & 0 &ab & b^2+c^2
\end{pmatrix}\,.
\end{equation}
 Again, a generic closed $H\in \Lambda^3\mathfrak n_4^*$ will be of the following form:
 $$
 H=\lambda_4e^{123}+\lambda_3e^{124}+\lambda_2e^{134}+\lambda_1e^{234}\,,\quad \lambda_i\in \mathbb R\,, \quad i=1,\ldots, 4\,.
 $$ Consequently, it is easy to see that 
 $$
 \Delta_{\mu}H
 =(-\lambda_4(c^2+b^2)+\lambda_3ab)e^{123} +(-\lambda_3a^2+\lambda_4ab)e^{124}\,
 $$ and the endomorphism $H^2$ will have the same form as in \eqref{eqn_sqh}. 
Combining \eqref{eqn_ric} and \eqref{eqn_sqh}, we can see that,  in order for ${\rm Ric}_{\mu, H}^B-\lambda{\rm Id}$ to be a derivation, we need to impose that 
$$
\lambda_2\lambda_1=\lambda_3\lambda_1=\lambda_4\lambda_1=\lambda_4\lambda_2=0\,,  \quad \lambda_4\lambda_3=ab\,,\quad \lambda_2\lambda_3=-bc\,,
$$ which gives us that $b=0$ and just one among $\lambda_i$'s can be non-zero.

First of all, we consider  $\lambda_1\ne 0 $.  In this particular case,  we have that ${\rm Ric}_{\mu, H}^{B}-\lambda{\rm Id}\in{\rm Der}(\mathfrak n_3\oplus \mathbb R, \mu)$ if and only if 
$$
 \lambda=-\frac32a^2
\quad \mbox{and }\quad  c^2=a^2\,.
$$
On the other hand,
$$
0=\rho({\rm Ric}_{\mu, H}^B)H-\Delta_{\mu}H+\lambda H=\frac32\lambda_1\left( \lambda_1^2-a^2\right)e^{234}
$$ which is satisfied if and only if $\lambda_1^2=a^2$, giving us \Cref{secondcase}.

 
 Now, we  assume $\lambda_2\ne 0 $.  In this case, we have that ${\rm Ric}_{\mu, H}^B-\lambda{\rm Id}\in{\rm Der}(\mathfrak n_4, \mu)$ if and only if 
\begin{equation}\label{eqn_lac}
 \lambda=-\frac32a^2 \quad \mbox{and }\quad 3c^2+\lambda_2^2=3a^2\,.
\end{equation}
 On the other hand, 
 $$
 \rho({\rm Ric}_{\mu, H}^B)H-\Delta_{\mu}H+\lambda H=0
 $$ is equivalent to 
 $$
 \frac{\lambda_2}{2}(c^2+3\lambda_2^2-3a^2)=0
 $$ which gives that $3a^2=c^2+3\lambda_2^2$. Combining this last equality with the second relation in \eqref{eqn_lac}, we obtain that $c^2=\lambda_1^2=\frac34a^2$, giving us \Cref{firstcasen4}. 
 
Now, we consider the case in which $\lambda_3\ne 0 $.  Then,  we have that ${\rm Ric}_{\mu, H}^B-\lambda{\rm Id}\in{\rm Der}(\mathfrak n_4, \mu)$ if and only if 
$$
\lambda=-\frac12(2\lambda_3^2+3a^2)\quad \mbox{and }\quad c^2=a^2+\frac23\lambda_3^2\,.
$$ 
Moreover,   the condition
$$
\rho({\rm Ric}_{\mu, H}^B)H-\Delta_{\mu}H+\ell H=0
$$ is equivalent to 
$$
\frac12\lambda_3\left(a^2+\lambda_3^2\right)=0
$$ which is impossible since we are assuming that   $\lambda_3\ne0$.

The left case to analyse is that in which $\lambda_4\ne 0$. In this hypothesis, we have that ${\rm Ric}_{\mu, H}^{B}-\lambda{\rm Id}\in{\rm Der}(\mathfrak n_3\oplus \mathbb R, \mu)$ if and only if 
\begin{equation}\label{eqn_lac2}
 \lambda=-\frac12(3a^2+\lambda_4^2)\quad \mbox{and }\quad  a^2=\frac13(\lambda_4^2+3c^2)\,.
\end{equation}
Moreover, 
$$
\rho({\rm Ric}_{\mu, H}^B)H-\Delta_{\mu}H+\lambda H =0
$$ is equivalent to 
$$
\lambda_4(-a^2+2c^2+\lambda_4^2)=0$$
 which is satisfied if and only if $ 2c^2+\lambda_4^2=a^2$. On the other hand, combining this with the second relation in \eqref{eqn_lac2} gives that 
$c^2=-\frac23\lambda_4^2$ which is not possible, concluding the proof.
 \end{proof}


\bibliography{grf}
\bibliographystyle{amsalpha}

\end{document}